\documentclass[10pt]{article}

\usepackage[a4paper,margin=25mm]{geometry}
\usepackage[T1]{fontenc}
\usepackage[utf8]{inputenc}
\usepackage{lmodern}
\usepackage{microtype}
\usepackage{graphicx}
\usepackage{float}
\usepackage{amsmath,amssymb,amsthm,mathtools}
\usepackage{enumitem}
\usepackage{xcolor}
\usepackage[numbers,sort&compress]{natbib}
\usepackage[colorlinks=true,linkcolor=blue!50!black,citecolor=blue!50!black,urlcolor=blue!50!black]{hyperref}

\newtheorem{theorem}{Theorem}[section]
\newtheorem{proposition}[theorem]{Proposition}
\newtheorem{lemma}[theorem]{Lemma}
\newtheorem{corollary}[theorem]{Corollary}
\theoremstyle{definition}
\newtheorem{definition}[theorem]{Definition}
\theoremstyle{remark}
\newtheorem{remark}[theorem]{Remark}

\newcommand{\R}{\mathbb{R}}

\newcommand{\Hh}{\mathbb{H}}
\newcommand{\eps}{\varepsilon}
\newcommand{\1}{\mathbf{1}}

\newcommand{\calE}{\mathcal{E}}

\title{Optimal location of small favourable regions for Robin eigenvalues with indefinite weights}
\author{Baruch Schneider$^{1}$, Diana Schneiderov\'a$^{1}$, Yifan Zhang$^{1,2,3}$\\[0.45em]
\small $^{1}$Department of Mathematics, University of Ostrava, Ostrava, Czech Republic\\
\small $^{2}$Department of Algebra, Charles University, Prague, Czech Republic\\
\small $^{3}$Department of Applied Mathematics, V\v{S}B--Technical University of Ostrava, Ostrava, Czech Republic}
\date{}

\begin{document}
\maketitle

\begin{abstract}
We consider the positive principal eigenvalue of an elliptic problem with a Robin boundary condition and bang--bang indefinite weight $\kappa\mathbf 1_D-\mathbf 1_{\Omega\setminus D}$, and ask where a favourable region $D$ of prescribed small volume $|D|=\delta$ should be located.  Put $\varepsilon=\delta^{1/N}$ and $\tau_\delta=\alpha_\delta\varepsilon$.  We prove that there is a finite threshold $\tau_*=\tau_*(N,\kappa)$, independent of the ambient domain, such that $\delta^{2/N}\Lambda_\delta(\alpha_\delta)\to\Lambda_{\mathbb H}(\tau)$ whenever $\tau_\delta\to\tau<\infty$.  If $\tau<\tau_*$, optimal small regions concentrate at the boundary.  If $\tau>\tau_*$, their concentration centres move to distances much larger than $\varepsilon$ from the boundary, and after recentring and rescaling the favourable sets converge in measure to the whole-space optimal ball.  At $\tau=\tau_*$, the half-space problem admits a compact boundary optimiser as well as minimising sequences escaping to infinity.  In the finer regime $\tau_\delta=\tau_*+\sigma\varepsilon+o(\varepsilon)$, we determine the first-order competition between the boundary and interior configurations.  Tangential symmetry of compact threshold optimisers reduces the geometric correction to mean curvature.  In particular, every fixed finite Robin coefficient is asymptotically in the boundary regime.  Numerical experiments illustrate the boundary--interior transition, the three cases in the first-order selection law, and the curvature-dependent boundary location; for $N=2$ and $\kappa=1$ they suggest a transition near $\tau_*\approx3.2$.
\end{abstract}



\section{Introduction}

\subsection{Problem and main conclusion}

Let $\Omega\subset\mathbb R^N$, $N\ge2$, be a bounded smooth domain.  For a measurable favourable set $D\subset\Omega$ of prescribed volume $|D|=\delta$, consider the bang--bang weight $m_D=\kappa\mathbf 1_D-\mathbf 1_{\Omega\setminus D}$ with $\kappa>0$, and the principal eigenvalue problem
\begin{equation}\label{eq:intro-eigenproblem-submission}
 -\Delta u=\lambda m_Du\quad\hbox{in }\Omega,
 \qquad
 \partial_\nu u+\alpha_\delta u=0\quad\hbox{on }\partial\Omega.
\end{equation}
We minimise its positive principal eigenvalue over all such sets $D$.  In population models this eigenvalue determines a linear persistence threshold for a diffusive population in a heterogeneous habitat, while the Robin condition models partial loss through the outer boundary \citep{CantrellCosner1989,CantrellCosner2003,LouYanagida2006,MinorsDawes2017,SearleVanVuuren2021,SearleVanVuuren2022}.  The geometric optimisation question is therefore: when only a small favourable volume is available, does an optimal patch benefit from touching the boundary, or does boundary loss make an interior patch preferable?  At the Neumann endpoint the boundary is advantageous through reflection, whereas the Dirichlet limit favours interior concentration; our aim is to identify the scale at which the change occurs.

The relevant comparison is between the patch length scale $\varepsilon=\delta^{1/N}$ and the Robin length $\alpha_\delta^{-1}$.  We therefore set
\begin{equation}\label{eq:intro-critical-scale}
 \varepsilon:=\delta^{1/N},\qquad \tau_\delta:=\alpha_\delta\varepsilon.
\end{equation}
The main result gives a complete leading-order answer whenever $\tau_\delta$ has a finite limit.  There is a finite, domain-independent number $\tau_*=\tau_*(N,\kappa)$ such that optimal regions concentrate at the boundary when $\tau_\delta\to\tau<\tau_*$, whereas for $\tau_\delta\to\tau>\tau_*$ their concentration centres lie at distances large compared with the patch scale and, after recentring, their rescaled favourable sets converge in measure to the whole-space optimal ball.  Thus every fixed Robin coefficient $\alpha$ is asymptotically in the boundary regime as $\delta\downarrow0$; the first possible transition to an interior configuration occurs when $\alpha_\delta$ reaches the scale $\delta^{-1/N}$.  At the transition scale we also determine the first correction, which decides between the boundary and interior configurations and, when the boundary configuration is selected, records the effect of mean curvature.

Robin eigenvalue optimisation with indefinite weights has been studied in \citet{AfrouziBrown1999,HintermullerKaoLaurain2012,Daners2013,LamboleyLaurainNadinPrivat2016}; weighted anisotropic variants appear in \citet{PellacciPisanteSchiera2024,PellacciPisanteSchieraCorrigendum2026}.  Exact information is available in several special Robin geometries \citep{SchneiderSchneiderovaZhang2025,SchneiderSchneiderovaZhangOneD2025}.  Singular small-volume limits have recently been analysed separately at the Neumann and Dirichlet endpoints: the Neumann problem concentrates at the boundary, with a refined location law governed by mean curvature, whereas the Dirichlet problem concentrates in the interior and approaches the whole-space ball \citep{MazzoleniPellacciVerzini2020,MazzoleniPellacciVerzini2023,FerreriVerzini2024,FerreriMazzoleniPellacciVerzini2026}; see also the periodic setting in \citet{Verzini2026}.  The new point here is that the Robin parameter is allowed to vary with the favourable volume.  This reveals the critical scale $\alpha_\delta\asymp\delta^{-1/N}$, a universal transition threshold on that scale, and a first-order selection law at the threshold.  These features are not visible when the boundary condition is kept fixed while $\delta\downarrow0$.

\subsection{Main theorem}

Let $\Lambda_\delta(\alpha_\delta)$ denote the optimal value of \eqref{eq:intro-eigenproblem-submission}.  Let $\Lambda_{\mathbb R^N}=:k^2$ be the unit-volume whole-space optimum, and let $\Lambda_{\mathbb H}(\tau)$ be the corresponding half-space value defined in Section~\ref{sec:limit-problems}.  The profile statements below are made precise in Theorem~\ref{thm:bounded-profile-critical}.

\begin{theorem}[Small-volume location transition]\label{thm:intro-main}
Assume that $\Omega$ is bounded and connected, with $C^2$ boundary for \emph{(i)}--\emph{(ii)} and $C^{2,1}$ boundary for \emph{(iii)}.  There exists $\tau_*=\tau_*(N,\kappa)\in(0,\infty)$ with the following properties.
\begin{enumerate}[label=(\roman*)]
\item If $\tau_\delta\to\tau<\infty$, then
\begin{equation}\label{eq:intro-leading-law}
 \varepsilon^2\Lambda_\delta(\alpha_\delta)\longrightarrow\Lambda_{\mathbb H}(\tau),
\end{equation}
and
\begin{equation}\label{eq:intro-phase-law}
 \Lambda_{\mathbb H}(\tau)<k^2\quad(0\le\tau<\tau_*),
 \qquad
 \Lambda_{\mathbb H}(\tau)=k^2\quad(\tau\ge\tau_*).
\end{equation}
The half-space infimum is attained for $\tau\le\tau_*$ and is not attained for $\tau>\tau_*$.

\item If $\tau<\tau_*$, optimal small favourable sets have a boundary-attached half-space limit profile.  If $\tau>\tau_*$, there are concentration centres whose distance from the boundary, measured in units of $\varepsilon$, tends to infinity; after recentring at these points, the rescaled favourable sets converge in measure to the whole-space optimal ball.  At $\tau=\tau_*$ these are the only two possible subsequential behaviours.

\item If
\begin{equation}\label{eq:intro-critical-window-summary}
 \tau_\delta=\tau_*+\sigma\varepsilon+o(\varepsilon),
\end{equation}
then an explicitly defined boundary correction $\Gamma_{\rm bd}(\sigma)$ satisfies
\begin{equation}\label{eq:intro-first-order-summary}
 \frac{\varepsilon^2\Lambda_\delta(\alpha_\delta)-k^2}{\varepsilon}
 \longrightarrow \min\{0,\Gamma_{\rm bd}(\sigma)\}.
\end{equation}
Every compact threshold optimiser is rotationally symmetric in the tangential variables, up to a common tangential translation.  Consequently the geometric part of $\Gamma_{\rm bd}$ depends on $\partial\Omega$ only through its mean curvature.
\end{enumerate}
\end{theorem}

In terms of the original optimisation problem, Theorem~\ref{thm:intro-main} says that the boundary is advantageous for sufficiently small Robin loss and disadvantageous once the scaled loss $\alpha_\delta\delta^{1/N}$ crosses $\tau_*$.  The critical scale is therefore $\alpha_\delta\sim\tau_*\delta^{-1/N}$.  Below this scale the optimal favourable region concentrates at the boundary; above it the concentration centre moves away from the boundary on the patch scale and, after recentring, the rescaled favourable set converges in measure to the whole-space optimal ball.  At critical scaling, the $O(\delta^{1/N})$ correction determines whether the boundary or interior configuration is selected, and the location of a boundary-attached critical patch is selected through mean curvature.

Section~\ref{sec:numerics} gives numerical illustrations organised around the same decision problem as the theorem.  For $N=2$ and $\kappa=1$, side-by-side computations show that the boundary configuration has the smaller value below the threshold and the interior configuration has the smaller value above it; in the critical window they show both signs of the first-order selection law and a finite-$\delta$ preference in the first-order equality case.  A smooth ellipse then illustrates the mean-curvature selection of the boundary location.  Further contrast values and an axisymmetric $N=3$ computation show that the transition is not confined to one parameter choice or to two dimensions.

The proof is organised around the two limit geometries.  Section~\ref{sec:limit-problems} analyses the half-space problem and identifies the finite threshold, its attainment properties, the compactness/escape alternatives, and tangential symmetry at the threshold.  Section~\ref{sec:bounded-targets} transfers these results to bounded domains by localisation and blow-up, and then derives the first-order critical correction by boundary flattening.  The numerical experiments are independent of the proofs.  We do not address uniqueness of the threshold optimiser, detailed free-boundary regularity, quantitative near-sphericity, or the higher-order analysis of the degenerate case $\Gamma_{\rm bd}(\sigma)=0$.

\section{Variational setting and small-volume scaling}\label{sec:framework}

Assume throughout that $\Omega\subset\R^N$, $N\ge2$, is bounded, open, and connected.  A $C^2$ boundary suffices for the leading-order results; for the first-order theorem we assume $C^{2,1}$.  With outward unit normal $\nu$, we use the convention $\mathsf S_P\xi=D_\xi\nu(P)$ and $h_P:=\operatorname{tr}\mathsf S_P=(N-1)H_{\partial\Omega}(P)$, so an outward-oriented sphere of radius $R$ has principal curvatures $+1/R$.

Fix $\kappa>0$ and, for measurable $D\subset\Omega$ with $|D|=\delta$, set
\begin{equation}\label{eq:mD}
 m_D:=\kappa\1_D-\1_{\Omega\setminus D}.
\end{equation}
We work in the small-volume range $0<\delta<|\Omega|/(\kappa+1)$, for which $\int_\Omega m_D<0$.  For $\alpha\ge0$, define
\begin{equation}\label{eq:Rayleigh}
 \lambda_1(D;\alpha)
 :=\inf_{\substack{u\in H^1(\Omega)\\ \int_\Omega m_Du^2>0}}
 \frac{\displaystyle\int_\Omega|\nabla u|^2+\alpha\int_{\partial\Omega}u^2}
 {\displaystyle\int_\Omega m_Du^2},
\end{equation}
and
\begin{equation}\label{eq:Lambda-delta}
 \Lambda_\delta(\alpha):=\inf_{\substack{D\subset\Omega\ \mathrm{measurable}\\ |D|=\delta}}\lambda_1(D;\alpha).
\end{equation}
The standard principal-eigenvalue theory for indefinite weights gives positivity and simplicity on this branch \citep{AfrouziBrown1999,Daners2013}.

\begin{proposition}[Existence and bathtub structure]\label{prop:bounded-existence-bathtub}
For every finite $\alpha\ge0$ and $0<\delta<|\Omega|/(\kappa+1)$, \eqref{eq:Lambda-delta} has an optimiser $D_\delta$.  Its positive principal eigenfunction $u_\delta$, normalised by $\int_\Omega m_{D_\delta}u_\delta^2=1$, satisfies for some $t_\delta>0$
\begin{equation}\label{eq:bounded-bathtub-level}
 \{u_\delta>t_\delta\}\subset D_\delta\subset\{u_\delta\ge t_\delta\}
\end{equation}
up to null sets.
\end{proposition}

\begin{proof}
In the relaxed class $m=-1+(\kappa+1)\theta$, $0\le\theta\le1$, $\int_\Omega\theta=\delta$, weak-$*$ compactness and the direct method give a minimiser; the Neumann endpoint uses $\int_\Omega m<0$.  For a fixed eigenfunction the denominator is maximised by the bathtub principle, so a bang--bang minimiser exists \citep{HintermullerKaoLaurain2012,LamboleyLaurainNadinPrivat2016}.  Applying the same principle to the eigenfunction of any shape optimiser gives \eqref{eq:bounded-bathtub-level}; otherwise replacing $D_\delta$ by a set with larger $\int_Du_\delta^2$ would strictly lower its Rayleigh quotient.
\end{proof}

\subsection{Blow-up scales}\label{sec:scaling}

Set $\varepsilon=\delta^{1/N}$.  Under the dilation $x=P+\varepsilon y$, gradient energy scales as $\varepsilon^{N-2}$ and the weighted $L^2$ denominator as $\varepsilon^N$, hence eigenvalues are of order $\varepsilon^{-2}=\delta^{-2/N}$.  If $P\in\partial\Omega$, the rescaled domain converges locally to the half-space $\mathbb H:=\{y_N>0\}$, and the Robin condition becomes $\partial_{\nu_y}v+\tau_\delta v=0$ with $\tau_\delta=\varepsilon\alpha_\delta$.  Curvature disappears at leading order and enters at order $\varepsilon$.  If instead $\operatorname{dist}(P_\delta,\partial\Omega)/\varepsilon\to\infty$, the blow-up domain is $\R^N$ and the boundary is invisible.  A finite nonzero rescaled distance again gives a half-space, allowing profiles that need not touch the boundary.  These are precisely the competing geometries behind Theorem~\ref{thm:intro-main}.

\section{Limit variational problems}\label{sec:limit-problems}

\subsection{The whole-space problem}

For a measurable set $A\subset\R^N$ with $|A|=1$, set $m_A=\kappa\1_A-\1_{\R^N\setminus A}$.  The whole-space energy is
\begin{equation}\label{eq:whole-space-value-A}
\lambda_{\R^N}(A)
:=
\inf\left\{
\int_{\R^N}|\nabla v|^2\,dy:
 v\in H^1(\R^N),\
 \int_{\R^N}m_Av^2\,dy=1
\right\},
\end{equation}
and the optimal whole-space value is
\begin{equation}\label{eq:whole-space-Lambda}
\Lambda_{\R^N}
:=
\inf_{|A|=1}\lambda_{\R^N}(A).
\end{equation}
The whole-space problem also appears in the Dirichlet and periodic small-volume limits.  In the relevant formulations, rearrangement arguments show that the optimal favourable region is a ball up to translation; see \citet{FerreriVerzini2024,Verzini2026} and references therein.

We also use the same variational problem with arbitrary favourable volume. For $V>0$, let
\begin{equation}\label{eq:whole-space-volume-V}
\Lambda_{\R^N}(V)
:=
\inf_{\substack{A\subset\R^N\\|A|=V}}
\inf_{\substack{v\in H^1(\R^N)\\ \int_{\R^N}m_Av^2>0}}
\frac{\int_{\R^N}|\nabla v|^2\,dy}
{\int_{\R^N}m_Av^2\,dy}.
\end{equation}
Thus $\Lambda_{\R^N}=\Lambda_{\R^N}(1)$.

\begin{lemma}[Exact whole-space scaling]\label{lem:whole-scaling}
For every $V>0$,
\begin{equation}\label{eq:whole-scaling}
\Lambda_{\R^N}(V)=V^{-2/N}\Lambda_{\R^N}.
\end{equation}
\end{lemma}

\begin{proof}
Let $r=V^{1/N}$. If $|A|=1$, set $A_r=rA$ and, for $v\in H^1(\R^N)$, set $v_r(x)=v(x/r)$. Then $|A_r|=V$ and
\[
\int_{\R^N}|\nabla v_r|^2\,dx
=r^{N-2}\int_{\R^N}|\nabla v|^2\,dy,
\qquad
\int_{\R^N}m_{A_r}v_r^2\,dx
=r^N\int_{\R^N}m_Av^2\,dy.
\]
Hence the Rayleigh quotient is multiplied by $r^{-2}=V^{-2/N}$. Taking infima gives
$\Lambda_{\R^N}(V)\le V^{-2/N}\Lambda_{\R^N}$. Applying the same argument with dilation factor $r^{-1}$ to an arbitrary set of volume $V$ gives the reverse inequality.
\end{proof}

\begin{proposition}[Whole-space optimiser classification]\label{prop:whole-space-classification}
For every $V>0$, the value $\Lambda_{\R^N}(V)$ is strictly positive and is attained.  Up to translation, every optimal favourable set is the ball $B_{R_V}$ of volume $V$, and the corresponding positive normalised principal eigenfunction is radial about the same centre and strictly decreasing in the radial variable.  In particular, for $V=1$ there is, up to translation and multiplication before normalisation, a unique positive radial profile $\Phi$, and its volume-one bathtub set is $B_{R_1}$.  The threshold level $\Phi(R_1)$ is positive and its level set has zero Lebesgue measure.
\end{proposition}

\begin{proof}
Existence and the ball characterisation, including uniqueness of the optimal weight up to translations, are the whole-space rearrangement results used in the sharp small-volume theories; see \citet{FerreriVerzini2024,Verzini2026} and the references therein.  We also need the following strict radial monotonicity.  Once the favourable ball is fixed and the positive principal eigenfunction is chosen radial, its radial equation satisfies
\[
(r^{N-1}\Phi')'=-\kappa\Lambda_{\R^N}(V)r^{N-1}\Phi<0
\qquad (0<r<R_V),
\]
with $\Phi'(0)=0$.  Hence $\Phi'(r)<0$ on $(0,R_V)$.  Outside the ball,
\[
(r^{N-1}\Phi')'=\Lambda_{\R^N}(V)r^{N-1}\Phi>0,
\]
and the decaying positive solution is a positive multiple of
$r^{-(N-2)/2}K_{(N-2)/2}(\sqrt{\Lambda_{\R^N}(V)}\,r)$, which is strictly decreasing.  Continuity of the radial derivative across $r=R_V$ therefore gives strict decrease on all of $(0,\infty)$.  The assertions about the bathtub level follow immediately: the unique volume-$V$ upper level set is $B_{R_V}$, its threshold is $\Phi(R_V)>0$, and strict radial monotonicity makes the corresponding level sphere negligible.
\end{proof}

The translation invariance of \eqref{eq:whole-space-Lambda} is the main source of noncompactness. It will reappear in the half-space problem when an almost optimal configuration escapes arbitrarily far from the boundary.

\subsection{The Robin half-space problem}

For $\tau\ge0$ and $A\subset\Hh$ measurable with $|A|=1$, define
\begin{equation}\label{eq:halfspace-A}
\lambda_{\Hh,\tau}(A)
:=
\inf\left\{
\int_{\Hh}|\nabla v|^2\,dy
+\tau\int_{\partial\Hh}v^2\,dS:
\begin{array}{l}
 v\in H^1(\Hh),\\[1mm]
 \displaystyle\int_{\Hh}m_Av^2\,dy=1
\end{array}
\right\}.
\end{equation}
Then set
\begin{equation}\label{eq:halfspace-Lambda}
\Lambda_{\Hh}(\tau)
:=
\inf_{\substack{A\subset\Hh\\|A|=1}}
\lambda_{\Hh,\tau}(A).
\end{equation}
This half-space value governs the small-volume limit.  We establish its properties below.

\subsection{Basic properties of the half-space value}

The following reduction eliminates the favourable set from the half-space minimisation and will be used throughout.

\begin{lemma}[Bathtub reduction for the favourable set]\label{lem:bathtub-halfspace}
Fix $v\in H^1(\Hh)$, $v\not\equiv0$. Among all measurable sets $A\subset\Hh$ with $|A|=1$, the quantity
\[
\int_{\Hh}m_Av^2
=(\kappa+1)\int_Av^2-\int_{\Hh}v^2
\]
is maximised by a superlevel set of $v^2$ of measure one, with the usual freedom on a level set of positive measure. Consequently, in the joint optimisation defining $\Lambda_{\Hh}(\tau)$ it is enough to consider favourable sets obtained from the upper level sets of the competing function.
\end{lemma}

\begin{proof}
The second term is independent of $A$, so one only has to maximise $\int_Av^2$ under the constraint $|A|=1$. This is the classical bathtub principle applied to the nonnegative function $v^2$.
\end{proof}

\begin{lemma}[Reduced half-space functional]\label{lem:reduced-halfspace}
For $v\in H^1(\Hh)$ define
\begin{equation}\label{eq:B-functional}
\mathcal B(v)
:=
\sup_{\substack{A\subset\Hh\\ |A|=1}}
\int_A v^2,
\end{equation}
and
\begin{equation}\label{eq:M-functional}
\mathcal M(v)
:=
(\kappa+1)\mathcal B(v)-\int_{\Hh}v^2.
\end{equation}
Then
\begin{equation}\label{eq:reduced-halfspace}
\Lambda_{\Hh}(\tau)
=
\inf_{\substack{v\in H^1(\Hh)\\ \mathcal M(v)>0}}
\frac{\displaystyle
\int_{\Hh}|\nabla v|^2
+\tau\int_{\partial\Hh}v^2}
{\mathcal M(v)}.
\end{equation}
Moreover $\mathcal B(cv)=c^2\mathcal B(v)$ and $\mathcal M(cv)=c^2\mathcal M(v)$ for every real $c$.  Hence every admissible competitor may be normalised by
\begin{equation}\label{eq:M-normalisation}
\mathcal M(v)=1.
\end{equation}
Finally, $\mathcal B$ is continuous with respect to strong $L^2(\Hh)$ convergence: for $v,w\in L^2(\Hh)$,
\begin{equation}\label{eq:B-L2-continuity}
|\mathcal B(v)-\mathcal B(w)|
\le
\|v^2-w^2\|_{L^1(\Hh)}
\le
(\|v\|_2+\|w\|_2)\|v-w\|_2.
\end{equation}
\end{lemma}

\begin{proof}
For fixed $v$, the numerator in the Rayleigh quotient is independent of the favourable set $A$, whereas the denominator equals
\[
(\kappa+1)\int_Av^2-\int_{\Hh}v^2.
\]
Lemma~\ref{lem:bathtub-halfspace} therefore shows that the best choice of $A$ is exactly the one represented by $\mathcal B(v)$, which proves \eqref{eq:reduced-halfspace}.  Homogeneity follows directly.  For continuity, use
\[
\left|\sup_A\int_Av^2-\sup_A\int_Aw^2\right|
\le
\sup_A\int_A|v^2-w^2|
\le
\int_{\Hh}|v^2-w^2|
\]
and then Cauchy--Schwarz.
\end{proof}

\begin{lemma}[Stability of a unique bathtub set]\label{lem:bathtub-stability}
Let $X$ be either $\R^N$ or $\Hh$, and let $g_n,g\in L^1(X)$ be nonnegative with $g_n\to g$ strongly in $L^1(X)$. Suppose that there is a number $t>0$ such that
\[
E:=\{g>t\},\qquad |E|=1,\qquad |\{g=t\}|=0.
\]
Let $E_n\subset X$ be measurable sets with $|E_n|=1$ and assume that
\begin{equation}\label{eq:bathtub-stability-deficit}
0\le \sup_{|F|=1}\int_Fg_n-\int_{E_n}g_n\longrightarrow0.
\end{equation}
Then
\begin{equation}\label{eq:bathtub-stability-set}
|E_n\triangle E|\longrightarrow0.
\end{equation}
In particular, the conclusion applies when every $E_n$ is an exact bathtub maximiser for $g_n$.
\end{lemma}

\begin{proof}
Because $E$ is admissible in the supremum in \eqref{eq:bathtub-stability-deficit}, strong $L^1$ convergence gives
\[
0\le \int_Eg-\int_{E_n}g
\le 2\|g_n-g\|_{L^1}+o(1)\longrightarrow0.
\]
Since $|E_n|=|E|$, the exchanged pieces have equal measure. Fix $\eta>0$. On $E\cap\{g\ge t+\eta\}$ one has $g\ge t+\eta$, whereas on $E^c\cap\{g\le t-\eta\}$ one has $g\le t-\eta$. Comparing the exchanged pieces gives
\[
\int_Eg-\int_{E_n}g
\ge \eta\Bigl(|(E\setminus E_n)\cap\{g\ge t+\eta\}|+|(E_n\setminus E)\cap\{g\le t-\eta\}|\Bigr).
\]
Thus both displayed measures tend to zero. The remaining parts lie in the strip $\{|g-t|<\eta\}$, so
\[
\limsup_{n\to\infty}|E_n\triangle E|\le 2|\{|g-t|<\eta\}|.
\]
Since $|\{g=t\}|=0$, the right-hand side tends to zero as $\eta\downarrow0$.
\end{proof}

\begin{lemma}[Monotonicity and concavity of the half-space value]\label{lem:halfspace-concavity-prelim}
The map $\tau\mapsto\Lambda_{\Hh}(\tau)$ is nondecreasing and concave on $[0,\infty)$.
\end{lemma}

\begin{proof}
Use the reduced formulation \eqref{eq:reduced-halfspace} and normalise competitors by $\mathcal M(v)=1$. For every fixed admissible $v$,
\[
\tau\longmapsto \int_{\Hh}|\nabla v|^2+\tau\int_{\partial\Hh}v^2
\]
is affine and nondecreasing. The infimum of affine functions is concave, and the infimum of nondecreasing functions is nondecreasing. No compactness or attainment is used.
\end{proof}

\subsection{Volume-dependent functionals and strict subhomogeneity}

The volume-one bathtub functional is the correct object for the original half-space problem, but a concentration--compactness decomposition naturally divides the favourable volume among different profiles.  We therefore keep the favourable volume as an explicit parameter.

For $V\ge0$ and $v\in L^2(\Hh)$, set
\begin{equation}\label{eq:B-V-half}
\mathcal B_V^{\Hh}(v)
:=
\sup_{\substack{A\subset\Hh\\ |A|=V}}
\int_A v^2,
\qquad
\mathcal M_V^{\Hh}(v)
:=
(\kappa+1)\mathcal B_V^{\Hh}(v)-\|v\|_{L^2(\Hh)}^2,
\end{equation}
with $\mathcal B_0^{\Hh}(v)=0$.  Thus $\mathcal B_1^{\Hh}=\mathcal B$ and $\mathcal M_1^{\Hh}=\mathcal M$.  For $V>0$ define
\begin{equation}\label{eq:Lambda-H-volume}
\Lambda_{\Hh}(\tau;V)
:=
\inf_{\mathcal M_V^{\Hh}(v)>0}
\frac{\displaystyle
\int_{\Hh}|\nabla v|^2+\tau\int_{\partial\Hh}v^2}
{\mathcal M_V^{\Hh}(v)}.
\end{equation}
We use analogous notation $\mathcal B_V^{\R^N}$ and $\mathcal M_V^{\R^N}$ in the whole space.  The bathtub argument used in Lemma~\ref{lem:reduced-halfspace} gives, for every $V>0$,
\begin{equation}\label{eq:whole-reduced-V}
\Lambda_{\R^N}(V)
=
\inf_{\mathcal M_V^{\R^N}(v)>0}
\frac{\int_{\R^N}|\nabla v|^2}{\mathcal M_V^{\R^N}(v)}.
\end{equation}

\begin{lemma}[Exact half-space volume scaling]\label{lem:halfspace-volume-scaling}
For every $V>0$ and $\tau\ge0$,
\begin{equation}\label{eq:halfspace-volume-scaling}
\Lambda_{\Hh}(\tau;V)
=
V^{-2/N}\Lambda_{\Hh}\!\left(\tau V^{1/N}\right).
\end{equation}
\end{lemma}

\begin{proof}
Let $r=V^{1/N}$.  Starting from a unit-volume competitor $(A,v)$, set $A_r=rA$ and $v_r(x)=v(x/r)$.  Then
\[
\int_{\Hh}|\nabla v_r|^2
=r^{N-2}\int_{\Hh}|\nabla v|^2,
\qquad
\int_{\partial\Hh}v_r^2
=r^{N-1}\int_{\partial\Hh}v^2,
\]
while
\[
\mathcal M_V^{\Hh}(v_r)=r^N\mathcal M_1^{\Hh}(v).
\]
Consequently
\[
\frac{\displaystyle
\int_{\Hh}|\nabla v_r|^2+\tau\int_{\partial\Hh}v_r^2}
{\mathcal M_V^{\Hh}(v_r)}
=
r^{-2}
\frac{\displaystyle
\int_{\Hh}|\nabla v|^2+(\tau r)\int_{\partial\Hh}v^2}
{\mathcal M_1^{\Hh}(v)}.
\]
Taking the infimum gives one inequality, and dilating an arbitrary volume-$V$ competitor by $r^{-1}$ gives the reverse inequality.
\end{proof}

\begin{corollary}[Strict subhomogeneity with respect to favourable volume]\label{cor:strict-resource-subhomogeneity}
For every $\tau\ge0$ and every $0<V<1$,
\begin{equation}\label{eq:strict-resource-subhomogeneity}
\Lambda_{\Hh}(\tau;V)>\Lambda_{\Hh}(\tau).
\end{equation}
More precisely, if $r=V^{1/N}$, then
\begin{equation}\label{eq:strict-resource-quantitative}
\Lambda_{\Hh}(\tau;V)
\ge
r^{-1}\Lambda_{\Hh}(\tau)
+(1-r)r^{-2}\Lambda_{\Hh}(0)
>
\Lambda_{\Hh}(\tau).
\end{equation}
The whole-space analogue is
\begin{equation}\label{eq:whole-strict-resource}
\Lambda_{\R^N}(V)=V^{-2/N}\Lambda_{\R^N}>\Lambda_{\R^N}
\qquad(0<V<1).
\end{equation}
\end{corollary}

\begin{proof}
By Lemma~\ref{lem:halfspace-concavity-prelim} and $0<r<1$,
\[
\Lambda_{\Hh}(r\tau)
\ge
r\Lambda_{\Hh}(\tau)+(1-r)\Lambda_{\Hh}(0).
\]
Multiply by $r^{-2}$ and use Lemma~\ref{lem:halfspace-volume-scaling}.  Since $r^{-1}>1$ and $\Lambda_{\Hh}(0)>0$, the inequality is strict.  The whole-space statement is Lemma~\ref{lem:whole-scaling}.
\end{proof}

The preceding inequality excludes a nontrivial splitting of the favourable volume: after a split, each profile receiving a strict fraction of the unit volume has a larger spectral cost per unit of positive reduced mass.

\begin{lemma}[Bathtub functional for separated profiles]\label{lem:exact-B-splitting}
Let $u,w\in L^2(\Hh)$ have essentially disjoint supports.  Then, for every $V>0$,
\begin{equation}\label{eq:exact-B-splitting}
\mathcal B_V^{\Hh}(u+w)
=
\max_{0\le a\le V}
\left\{
\mathcal B_a^{\Hh}(u)+\mathcal B_{V-a}^{\Hh}(w)
\right\}.
\end{equation}
Consequently,
\begin{equation}\label{eq:exact-M-splitting}
\mathcal M_V^{\Hh}(u+w)
=
\max_{0\le a\le V}
\left\{
\mathcal M_a^{\Hh}(u)+\mathcal M_{V-a}^{\Hh}(w)
\right\}.
\end{equation}
\end{lemma}

\begin{proof}
Since the supports are disjoint, $(u+w)^2=u^2+w^2$ almost everywhere.  The identity is most transparent from decreasing rearrangements: for a nonnegative integrable function $f$,
\[
\sup_{|A|=V}\int_A f=\int_0^V f^*(s)\,ds.
\]
For two functions carried by disjoint measure spaces, the decreasing rearrangement of their sum is obtained by merging their two distribution functions; selecting the largest total mass of measure $V$ is therefore equivalent to choosing how much measure $a$ is assigned to the first component and $V-a$ to the second.  This gives \eqref{eq:exact-B-splitting}.  The map $a\mapsto\mathcal B_a^{\Hh}(u)$ is continuous, so the supremum is attained.  Subtracting $\|u\|_2^2+\|w\|_2^2$ proves \eqref{eq:exact-M-splitting}.
\end{proof}

\begin{lemma}[Bathtub functional under an asymptotically separated decomposition]\label{lem:asymptotic-B-splitting}
Let $(q_n)$ be bounded in $H^1(\Hh)$.  Then
\begin{equation}\label{eq:B-small-volume-uniform}
\lim_{V\downarrow0}\sup_n\mathcal B_V^{\Hh}(q_n)=0.
\end{equation}
Moreover, let $V>0$ and suppose
\[
v_n=u_n+w_n+r_n,
\]
where $(u_n)$ and $(w_n)$ are bounded in $L^2(\Hh)$, have essentially disjoint supports, and $\|r_n\|_2\to0$.  Then
\begin{align}
\mathcal B_V^{\Hh}(v_n)
&=
\max_{0\le a\le V}
\left\{
\mathcal B_a^{\Hh}(u_n)+\mathcal B_{V-a}^{\Hh}(w_n)
\right\}+o(1),
\label{eq:asymptotic-B-splitting}\\
\mathcal M_V^{\Hh}(v_n)
&=
\max_{0\le a\le V}
\left\{
\mathcal M_a^{\Hh}(u_n)+\mathcal M_{V-a}^{\Hh}(w_n)
\right\}+o(1).
\label{eq:asymptotic-M-splitting}
\end{align}
\end{lemma}

\begin{proof}
For $N\ge3$, H\"older and the Sobolev inequality give, uniformly in $n$,
\[
\mathcal B_V^{\Hh}(q_n)
\le
V^{2/N}\|q_n\|_{2^*}^2
\le
C V^{2/N}.
\]
For $N=2$, the $H^1$ bound implies a uniform $L^4$ bound, and therefore
\[
\mathcal B_V^{\Hh}(q_n)
\le
V^{1/2}\|q_n\|_4^2
\le
C V^{1/2}.
\]
This proves \eqref{eq:B-small-volume-uniform}.

For the splitting statement, the continuity estimate used in Lemma~\ref{lem:reduced-halfspace} is valid for every $V>0$:
\begin{equation}\label{eq:B-V-L2-continuity}
\left|\mathcal B_V^{\Hh}(f)-\mathcal B_V^{\Hh}(g)\right|
\le
\|f^2-g^2\|_1
\le
(\|f\|_2+\|g\|_2)\|f-g\|_2.
\end{equation}
Hence
\[
\mathcal B_V^{\Hh}(v_n)
=
\mathcal B_V^{\Hh}(u_n+w_n)+o(1),
\]
and Lemma~\ref{lem:exact-B-splitting} gives \eqref{eq:asymptotic-B-splitting}.  Also
\[
\|v_n\|_2^2
=
\|u_n+w_n\|_2^2+o(1)
=
\|u_n\|_2^2+\|w_n\|_2^2+o(1),
\]
which yields \eqref{eq:asymptotic-M-splitting}.
\end{proof}

\begin{lemma}[Boundedness of normalised minimising sequences]\label{lem:normalised-bounded}
Fix $\tau\ge0$, and let $(v_n)$ satisfy
\[
\mathcal M(v_n)=1,
\qquad
\int_{\Hh}|\nabla v_n|^2
+\tau\int_{\partial\Hh}v_n^2
\le C.
\]
Then $(v_n)$ is bounded in $H^1(\Hh)$.  In particular this applies to every normalised minimising sequence for \eqref{eq:reduced-halfspace}.
\end{lemma}

\begin{proof}
The normalisation gives the exact identity
\begin{equation}\label{eq:normalisation-identity}
1+\|v_n\|_2^2=(\kappa+1)\mathcal B(v_n).
\end{equation}
Suppose first that $N\ge3$ and write $2^*=2N/(N-2)$.  Since $|A|=1$, H\"older's inequality and the Sobolev inequality on the half-space give
\[
\mathcal B(v_n)
\le \|v_n\|_{2^*}^2
\le C_N\|\nabla v_n\|_2^2.
\]
Together with \eqref{eq:normalisation-identity} and the energy bound, this controls $\|v_n\|_2$.

If $N=2$, then
\[
\mathcal B(v_n)\le\|v_n\|_4^2
\le C\|v_n\|_2\|\nabla v_n\|_2
\]
by the Gagliardo--Nirenberg inequality.  Since $\|\nabla v_n\|_2$ is bounded, \eqref{eq:normalisation-identity} yields a quadratic inequality for $\|v_n\|_2$, and therefore an $L^2$ bound.  The gradient bound is already contained in the assumed energy estimate.  Thus $(v_n)$ is bounded in $H^1(\Hh)$ in every dimension $N\ge2$.
\end{proof}

\subsection{Escape from the boundary and the whole-space threshold}

We first identify the limiting energy of a concentrated profile whose centre moves arbitrarily far from the boundary.

\begin{definition}[Tightness after translation away from the boundary]\label{def:tight-normal-escape}
A bounded sequence $(v_n)\subset H^1(\Hh)$ is said to be \emph{tight after translation away from the boundary} if there exist $h_n\to+\infty$ such that, after defining
\[
w_n(y',y_N):=v_n(y',y_N+h_n)
\quad\text{on }\Hh_n:=\{y_N>-h_n\},
\]
for every $\eta>0$ there exists $R>0$ for which, after discarding finitely many indices,
\begin{equation}\label{eq:H1-tight-escape}
\int_{\Hh_n\setminus B_R}
\bigl(|\nabla w_n|^2+w_n^2\bigr)
<\eta.
\end{equation}
\end{definition}

\begin{proposition}[Lower bound for sequences escaping from the boundary]\label{prop:tight-escape-lower}
Let $(v_n)$ satisfy $\mathcal M(v_n)=1$ and suppose that it is tight after translation away from the boundary in the sense of Definition~\ref{def:tight-normal-escape}.  Let $(\tau_n)$ be any sequence with $\tau_n\ge0$.  Then
\begin{equation}\label{eq:tight-escape-lower}
\liminf_{n\to\infty}
\left(
\int_{\Hh}|\nabla v_n|^2
+\tau_n\int_{\partial\Hh}v_n^2
\right)
\ge
\Lambda_{\R^N}.
\end{equation}
In particular, for fixed $\tau\ge0$, a normalised minimising sequence with value
$\Lambda_{\Hh}(\tau)<\Lambda_{\R^N}$ cannot escape from the boundary in this sense.
\end{proposition}

\begin{proof}
Translate by $h_ne_N$ and extend $w_n$ by zero outside $\Hh_n$ only as an $L^2(\R^N)$ function; denote this zero extension by $\bar w_n$.  Translation preserves the $L^2$ norm and the bathtub functional, while the complement of $\Hh_n$ contains only zeros.  Hence
\begin{equation}\label{eq:escape-M-exact}
\mathcal M_1^{\R^N}(\bar w_n)
=
\mathcal M_1^{\Hh}(v_n)=1.
\end{equation}
The zero extension need not lie in $H^1(\R^N)$, so we cut it off before extending it as an $H^1$ function.

Fix $\eta>0$ and choose $R$ according to \eqref{eq:H1-tight-escape}.  Let $\chi_R\in C_c^\infty(B_{2R})$ satisfy $0\le\chi_R\le1$, $\chi_R=1$ on $B_R$, and $|\nabla\chi_R|\le C/R$.  For all sufficiently large $n$, $h_n>2R$, and therefore
\[
z_{n,R}:=\chi_R w_n
\]
has compact support in $\Hh_n$ and may be extended by zero to an element of $H^1(\R^N)$.  Tightness gives
\[
\|z_{n,R}-\bar w_n\|_{L^2(\R^N)}^2\le\eta.
\]
The $L^2$ continuity estimate for the bathtub functional, applied in the whole space, together with the uniform $L^2$ bound, yields
\begin{equation}\label{eq:escape-M-approx}
\mathcal M_1^{\R^N}(z_{n,R})=1+o_\eta(1)
\end{equation}
uniformly for all sufficiently large $n$, where $o_\eta(1)\to0$ as $\eta\downarrow0$.

Expanding $\nabla(\chi_Rw_n)$ and using that the transition annulus $B_{2R}\setminus B_R$ lies in the $H^1$ tail from \eqref{eq:H1-tight-escape}, we obtain
\begin{equation}\label{eq:escape-energy-cutoff}
\int_{\R^N}|\nabla z_{n,R}|^2
\le
\int_{\Hh_n}|\nabla w_n|^2+o_\eta(1).
\end{equation}
For $\eta$ small, \eqref{eq:escape-M-approx} is positive.  The whole-space variational inequality \eqref{eq:whole-reduced-V} with $V=1$ therefore gives
\[
\Lambda_{\R^N}\,\mathcal M_1^{\R^N}(z_{n,R})
\le
\int_{\R^N}|\nabla z_{n,R}|^2.
\]
Using \eqref{eq:escape-M-approx}--\eqref{eq:escape-energy-cutoff}, translation invariance of the bulk energy, and the nonnegativity of the Robin trace term, we obtain
\[
\Lambda_{\R^N}(1-o_\eta(1))
\le
\liminf_{n\to\infty}
\left(
\int_{\Hh}|\nabla v_n|^2
+\tau_n\int_{\partial\Hh}v_n^2
\right)
+o_\eta(1).
\]
Letting $\eta\downarrow0$ proves \eqref{eq:tight-escape-lower}.
\end{proof}

\begin{theorem}[Compactness and attainment below the whole-space value]\label{thm:strict-gap-attainment}
Fix $\tau\ge0$ and assume
\begin{equation}\label{eq:strict-gap-assumption}
\Lambda_{\Hh}(\tau)<\Lambda_{\R^N}.
\end{equation}
Then every normalised minimising sequence for \eqref{eq:reduced-halfspace} is precompact in $H^1(\Hh)$ after translations parallel to $\partial\Hh$.  In particular, $\Lambda_{\Hh}(\tau)$ is attained.  More precisely, there exist $v\in H^1(\Hh)$ and a measurable set $A\subset\Hh$, $|A|=1$, such that
\[
\mathcal M(v)=1,
\qquad
\Lambda_{\Hh}(\tau)
=
\int_{\Hh}|\nabla v|^2
+\tau\int_{\partial\Hh}v^2,
\]
and $A$ is a bathtub maximiser for $v^2$.
\end{theorem}

\begin{proof}
Let $(v_n)$ be normalised by $\mathcal M(v_n)=1$ and satisfy
$E_\tau(v_n)\to L:=\Lambda_{\Hh}(\tau)<\Lambda_{\R^N}$.  Lemma~\ref{lem:normalised-bounded} gives boundedness in $H^1(\Hh)$.  We apply Lions' concentration--compactness principle to the normalised measures
$(|\nabla v_n|^2+v_n^2)\,dx$; see \citet{Lions1984}.  Only the modifications caused by the optimised denominator are recorded here.

Vanishing implies $v_n\to0$ in every admissible $L^p$, hence
$\mathcal B(v_n)\to0$ and contradicts $\mathcal M(v_n)=1$.  Under dichotomy, the standard cutoff decomposition gives
$v_n=u_n+w_n+o_{H^1}(1)$ with separated nontrivial profiles and additive Robin energy.  Lemma~\ref{lem:asymptotic-B-splitting} provides favourable-volume fractions $a_n\in[0,1]$ and reduced masses $d_{1,n},d_{2,n}$ with
$d_{1,n}+d_{2,n}=1+o(1)$.  If $a_n$ stays away from $0$ and $1$, strict subhomogeneity in the favourable volume, Corollary~\ref{cor:strict-resource-subhomogeneity}, gives a uniform spectral cost $>L$ for both pieces.  If, say, $a_n\to1$, then Lemma~\ref{lem:asymptotic-B-splitting} gives
\[
 d_{2,n}=-\|w_n\|_2^2+o(1),\qquad d_{1,n}=1+\|w_n\|_2^2+o(1).
\]
Hence the endpoint estimate from the volume-scaled variational inequality is
\[
 E_\tau(v_n)\ge L+L\|w_n\|_2^2+\|\nabla w_n\|_2^2+o(1),
\]
and the last two terms are bounded away from zero by nontriviality of the dichotomy component.  The case $a_n\to0$ is symmetric.  Thus dichotomy is impossible.

Hence, after translations parallel to $\partial\Hh$, the sequence is tight about centres $(0,h_n)$.  If $h_n\to\infty$, Proposition~\ref{prop:tight-escape-lower} yields
$L\ge\Lambda_{\R^N}$, contrary to the strict gap.  The heights are therefore bounded.  Tightness and local Rellich compactness give, after a subsequence, $v_n\to v$ strongly in $L^2(\Hh)$ and weakly in $H^1(\Hh)$.  The $L^2$ continuity of $\mathcal B$ gives $\mathcal M(v)=1$, and weak lower semicontinuity yields $E_\tau(v)=L$.  Equality of the quadratic energies then gives strong $H^1$ convergence.  A bathtub maximiser for $v^2$ produces an optimal set $A$ of volume one.
\end{proof}

\begin{proposition}[Euler--Lagrange system for an attained half-space optimum]\label{prop:EL-halfspace}
Let $\tau\ge0$ and suppose that the half-space infimum at $\tau$ is attained by a pair $(A,v)$, normalised by
\[
\int_{\Hh}m_Av^2=1.
\]
The function $v$ may be chosen nonnegative, and then is a positive principal eigenfunction.  It satisfies
\begin{equation}\label{eq:EL-weak-halfspace}
\int_{\Hh}\nabla v\cdot\nabla\varphi
+\tau\int_{\partial\Hh}v\varphi
=
\Lambda_{\Hh}(\tau)
\int_{\Hh}m_Av\varphi
\end{equation}
for every $\varphi\in H^1(\Hh)$.  Equivalently,
\begin{equation}\label{eq:EL-strong-formal}
-\Delta v=\Lambda_{\Hh}(\tau)m_Av
\quad\text{in }\Hh,
\qquad
\partial_\nu v+\tau v=0
\quad\text{on }\partial\Hh
\end{equation}
in the weak sense.  Moreover there exists $t\ge0$ such that, up to null sets,
\begin{equation}\label{eq:bathtub-level-EL}
\{v>t\}\subset A\subset\{v\ge t\},
\qquad |A|=1.
\end{equation}
If the level set $\{v=t\}$ has zero measure, then $A=\{v>t\}$ almost everywhere.
\end{proposition}

\begin{proof}
Replacing $v$ by $|v|$ does not increase the Dirichlet energy and leaves the denominator and trace term unchanged, so an optimiser may be taken nonnegative.  Once the optimal set $A$ is fixed, $v$ minimises the Robin quadratic form under the smooth constraint $\int m_Av^2=1$.  The Lagrange multiplier rule gives \eqref{eq:EL-weak-halfspace}; testing with $v$ identifies the multiplier with $\Lambda_{\Hh}(\tau)$.  Positivity follows from the standard principal-eigenfunction argument for a bounded indefinite coefficient.  Finally, $A$ maximises $\int_Av^2$ among sets of measure one, so the bathtub principle yields \eqref{eq:bathtub-level-EL}.
\end{proof}

\begin{lemma}[Sobolev derivatives vanish on a level set]\label{lem:sobolev-level-set}
Let $U\subset\R^N$ be open. If $u\in W^{1,1}_{\rm loc}(U)$, then $\nabla u=0$ almost everywhere on every level set $\{u=c\}$. If in addition $u\in W^{2,1}_{\rm loc}(U)$, then $D^2u=0$ almost everywhere on $\{u=c\}$.
\end{lemma}

\begin{proof}
For Sobolev functions, the approximate differential agrees almost everywhere with the weak gradient. At almost every density point of $\{u=c\}$ where $u$ is approximately differentiable, the approximate differential must vanish because the function equals the constant $c$ on a set of density one. Hence $\nabla u=0$ almost everywhere on the level set. If $u\in W^{2,1}_{\rm loc}$, apply the same statement to each first derivative, which belongs to $W^{1,1}_{\rm loc}$ and vanishes almost everywhere on the level set, to obtain $D^2u=0$ there almost everywhere.
\end{proof}

\begin{lemma}[Zero measure of the bathtub threshold level]\label{lem:no-plateau-halfspace}
Let \(\tau\ge0\) and let \((A,v)\) be a positive optimiser from Proposition~\ref{prop:EL-halfspace}.  If \(t\) is its bathtub threshold in \eqref{eq:bathtub-level-EL}, then
\[
t>0,
\qquad
|\{v=t\}|=0.
\]
Consequently the favourable set is unique up to null sets and equals \(\{v>t\}\).
\end{lemma}

\begin{proof}
Positivity of the principal eigenfunction and \(|A|=1<|\Hh|\) exclude \(t=0\): otherwise \(\{v>0\}\subset A\) would force an infinite-measure favourable set.  Thus \(t>0\).

Since \(m_Av\in L^2_{\rm loc}(\Hh)\), interior elliptic regularity gives \(v\in W^{2,2}_{\rm loc}(\Hh)\).  If the level set \(E:=\{v=t\}\) had positive \(N\)-dimensional measure, then Lemma~\ref{lem:sobolev-level-set}, applied to $v\in W^{2,2}_{\rm loc}(\Hh)$, gives $\nabla v=0$ and $D^2v=0$ almost everywhere on $E$.  Hence \(\Delta v=0\) almost everywhere on \(E\).  On the other hand the equation
\[
-\Delta v=\Lambda_{\Hh}(\tau)m_Av
\]
holds almost everywhere and, on \(E\), its right-hand side is either \(\Lambda_{\Hh}(\tau)\kappa t\) or \(-\Lambda_{\Hh}(\tau)t\), both nonzero.  This contradiction proves \(|E|=0\).  The bathtub characterisation then gives \(A=\{v>t\}\) up to null sets.
\end{proof}
\begin{proposition}[Monotonicity, concavity, and comparison with the whole-space value]\label{prop:halfspace-basic}
The map
\[
[0,+\infty)\ni\tau\longmapsto\Lambda_{\Hh}(\tau)
\]
is nondecreasing and concave. Moreover,
\begin{equation}\label{eq:halfspace-whole-comparison}
0<\Lambda_{\Hh}(\tau)\le\Lambda_{\R^N}
\qquad\text{for every }\tau\ge0.
\end{equation}
It is continuous on $[0,+\infty)$.
\end{proposition}

\begin{proof}
Monotonicity and concavity were proved already in Lemma~\ref{lem:halfspace-concavity-prelim}.

To prove the upper bound, fix $\eta>0$ and use the reduced whole-space formulation \eqref{eq:whole-reduced-V} with $V=1$. Choose $v\in H^1(\R^N)$ with $\mathcal M_1^{\R^N}(v)>0$ and quotient at most $\Lambda_{\R^N}+\eta/2$. By density, choose a compactly supported $v_c\in C_c^\infty(\R^N)$ sufficiently close to $v$ in $H^1$ that $\mathcal M_1^{\R^N}(v_c)>0$ and its reduced quotient is at most $\Lambda_{\R^N}+\eta$. Choose a unit-volume bathtub set $A_c$ for $v_c^2$; because $v_c$ has compact support, $A_c$ may be chosen inside a fixed compact set after adding, if necessary, a null-contributing portion where $v_c=0$. Translate both $v_c$ and $A_c$ by $Le_N$ with $L$ larger than the diameter of that compact set. The translated pair lies entirely in $\Hh$, its trace on $\partial\Hh$ vanishes, and its quotient is unchanged. Hence $\Lambda_{\Hh}(\tau)\le\Lambda_{\R^N}+\eta$, and $\eta\downarrow0$ gives the upper bound $\Lambda_{\Hh}(\tau)\le\Lambda_{\R^N}$.

Positivity at $\tau=0$ follows from Proposition~\ref{prop:neumann-reflection} below, and positivity for $\tau>0$ then follows by monotonicity.

A finite concave function is continuous on $(0,+\infty)$. It remains to check right continuity at $0$. Fix $\eta>0$ and an admissible pair $(A,v)$ for $\tau=0$ with
\[
\int_{\Hh}|\nabla v|^2\le \Lambda_{\Hh}(0)+\eta.
\]
The trace theorem gives $\int_{\partial\Hh}v^2<\infty$, and hence
\[
\Lambda_{\Hh}(0)
\le \Lambda_{\Hh}(\tau)
\le \Lambda_{\Hh}(0)+\eta
+\tau\int_{\partial\Hh}v^2.
\]
First let $\tau\downarrow0$ and then $\eta\downarrow0$.
\end{proof}

The Neumann endpoint can be identified exactly. We use the known fact that the whole-space problem is attained, up to translations, by a ball together with a radial positive eigenfunction; this is part of the full-space theory underlying the sharp Dirichlet and periodic analyses \citep{FerreriVerzini2024,Verzini2026}.

\begin{proposition}[Exact Neumann reflection identity]\label{prop:neumann-reflection}
In the present normalisation,
\begin{equation}\label{eq:neumann-reflection}
\Lambda_{\Hh}(0)=2^{-2/N}\Lambda_{\R^N}.
\end{equation}
In particular,
\begin{equation}\label{eq:strict-gap-zero}
\Lambda_{\Hh}(0)<\Lambda_{\R^N}.
\end{equation}
\end{proposition}

\begin{proof}
Let $(A,v)$ be admissible in $\Hh$ with $|A|=1$, and let $A^{\mathrm e}$ and $v^{\mathrm e}$ denote their even reflections across $\partial\Hh$. Then $|A^{\mathrm e}|=2$, while both the Dirichlet energy and the weighted denominator are doubled. Hence the quotient is unchanged. Taking the infimum over half-space competitors gives
\[
\Lambda_{\R^N}(2)\le\Lambda_{\Hh}(0).
\]
By Lemma~\ref{lem:whole-scaling}, the left-hand side is $2^{-2/N}\Lambda_{\R^N}$.

Conversely, choose a whole-space optimiser of favourable volume two. By the radial characterisation, it may be centred on the hyperplane $\partial\Hh$, and both its favourable ball and eigenfunction are invariant under reflection across that hyperplane. Restricting to $\Hh$ produces a favourable set of volume one, and again halves numerator and denominator by the same factor. Therefore
\[
\Lambda_{\Hh}(0)\le\Lambda_{\R^N}(2)
=2^{-2/N}\Lambda_{\R^N}.
\]
The identity follows. Since $N\ge2$, $2^{-2/N}<1$.
\end{proof}

These facts motivate the following transition parameter.

\begin{definition}[Transition threshold]\label{def:tau-star}
Define
\begin{equation}\label{eq:tau-star}
\tau_*
:=
\sup\left\{\tau\ge0:\Lambda_{\Hh}(\tau)<\Lambda_{\R^N}\right\}
\in(0,+\infty].
\end{equation}
\end{definition}

\begin{corollary}[Basic properties of the transition threshold]\label{cor:threshold-structure}
The number $\tau_*$ in Definition~\ref{def:tau-star} is strictly positive. Moreover,
\begin{equation}\label{eq:threshold-interval}
\Lambda_{\Hh}(\tau)<\Lambda_{\R^N}
\quad\text{for }0\le\tau<\tau_*.
\end{equation}
If $\tau_*<+\infty$, then
\begin{equation}\label{eq:threshold-equality}
\Lambda_{\Hh}(\tau)=\Lambda_{\R^N}
\quad\text{for every }\tau\ge\tau_*.
\end{equation}
In particular, a finite transition cannot consist of several disjoint intervals in $\tau$.
\end{corollary}

\begin{proof}
The strict gap at $0$ and continuity imply strict inequality on $[0,\tau_0)$ for some $\tau_0>0$, hence $\tau_*>0$. Monotonicity shows that the strict-sublevel set is an interval starting at $0$. If $\tau_*<\infty$, continuity gives $\Lambda_{\Hh}(\tau_*)=\Lambda_{\R^N}$; monotonicity and Proposition~\ref{prop:halfspace-basic} then force equality for all larger $\tau$.
\end{proof}

\subsection{A variational characterisation of the transition threshold}\label{subsec:trace-defect}

The following variational reformulation identifies the transition threshold with a sharp constant involving the boundary trace.  Finiteness of that constant will be proved below by a geometric argument rather than by estimating this quotient directly.

Write
\[
k:=\sqrt{\Lambda_{\R^N}},
\]
and, for $v\in H^1(\Hh)$, define the \emph{whole-space defect}
\begin{equation}\label{eq:defect-functional}
\mathfrak D(v)
:=
\Lambda_{\R^N}\mathcal M(v)
-
\int_{\Hh}|\nabla v|^2\,dx.
\end{equation}
The quantity $\mathfrak D(v)$ measures how far the bulk part of the half-space test function lies below the unit-volume whole-space optimum before the Robin trace cost is added.

\begin{proposition}[Variational characterisation of $\tau_*$ by the boundary trace]\label{prop:tau-star-trace-defect}
One has
\begin{equation}\label{eq:tau-star-trace-defect}
\tau_*
=
\sup_{\substack{v\in H^1(\Hh),\ \mathcal M(v)>0\\
\|\operatorname{Tr}v\|_{L^2(\partial\Hh)}>0}}
\frac{\mathfrak D(v)}{\|\operatorname{Tr}v\|_{L^2(\partial\Hh)}^2}.
\end{equation}
Moreover, if $\operatorname{Tr}v=0$, then $\mathfrak D(v)\le0$.  Consequently, $\tau_*<\infty$ is equivalent to the existence of a finite constant $C$ such that
\begin{equation}\label{eq:sharp-trace-defect-ineq}
\Lambda_{\R^N}\mathcal M(v)
\le
\int_{\Hh}|\nabla v|^2\,dx
+C\int_{\partial\Hh}v^2\,dS
\end{equation}
for every $v\in H^1(\Hh)$, and in that case the smallest admissible constant in \eqref{eq:sharp-trace-defect-ineq} is exactly $C=\tau_*$.
\end{proposition}

\begin{proof}
For $f=\operatorname{Tr}v\neq0$, the reduced Rayleigh quotient is below $k^2=\Lambda_{\R^N}$ at the parameter $\tau$ exactly when
\[
 \mathfrak D(v)>\tau\|f\|_2^2.
\]
Taking the supremum of the corresponding ratios gives \eqref{eq:tau-star-trace-defect}.  If $f=0$, extend $v$ by zero to $\widetilde v\in H^1(\R^N)$ and choose a unit-volume bathtub set $A$ for $v$.  The whole-space variational inequality gives
\[
 \int_{\R^N}|\nabla\widetilde v|^2\ge k^2\int_{\R^N}m_A\widetilde v^2=k^2\mathcal M(v),
\]
so $\mathfrak D(v)\le0$.  The sharp-constant formulation is therefore equivalent to the same supremum identity.
\end{proof}

This boundary-trace characterisation can be compared with the Dirichlet-to-Neumann operator for the equation in the unfavourable region.  We use the unitary Fourier transform in the tangential variables $x'\in\R^{N-1}$ and set
\begin{equation}\label{eq:Sk-symbol}
S_k:=\sqrt{k^2-\Delta_{x'}},
\qquad
\widehat{S_k f}(\xi)=s(\xi)\widehat f(\xi),
\qquad
s(\xi):=\sqrt{k^2+|\xi|^2}.
\end{equation}

\begin{proposition}[Dirichlet-to-Neumann control of the whole-space defect]\label{prop:DtN-defect}
For every $v\in H^1(\Hh)$ with trace $f=\operatorname{Tr}v$,
\begin{equation}\label{eq:DtN-defect}
\mathfrak D(v)
\le
\langle f,S_kf\rangle
=
\int_{\R^{N-1}}s(\xi)|\widehat f(\xi)|^2\,d\xi.
\end{equation}
In particular, every admissible function with quotient below $k^2$ at the parameter $\tau$ satisfies
\begin{equation}\label{eq:boundary-frequency-necessary}
\frac{\displaystyle\int_{\R^{N-1}}s(\xi)|\widehat f(\xi)|^2\,d\xi}
{\displaystyle\int_{\R^{N-1}}|\widehat f(\xi)|^2\,d\xi}
>\tau.
\end{equation}
\end{proposition}

\begin{proof}
Choose a bathtub set $A\subset\Hh$, $|A|=1$, such that
\[
\mathcal M(v)=\int_{\Hh}m_Av^2.
\]
Let $L:=\{x_N<0\}$ and let $z$ be the energy-minimising Yukawa extension of $f$ into $L$; equivalently,
\begin{equation}\label{eq:yukawa-extension}
(-\Delta+k^2)z=0\quad\hbox{in }L,
\qquad
z=f\quad\hbox{on }\partial L,
\end{equation}
with $z$ decaying as $x_N\to-\infty$.  In tangential Fourier variables,
\[
\widehat z(\xi,x_N)
=e^{s(\xi)x_N}\widehat f(\xi),
\qquad x_N<0,
\]
and therefore
\begin{equation}\label{eq:yukawa-energy}
\int_L\bigl(|\nabla z|^2+k^2z^2\bigr)\,dx
=
\int_{\R^{N-1}}s(\xi)|\widehat f(\xi)|^2\,d\xi.
\end{equation}
Glue $v$ and $z$ along the boundary to obtain $w\in H^1(\R^N)$.  Extend the bang--bang weight by declaring the whole lower half-space unfavourable; that is, use the same favourable set $A\subset\Hh$ as a unit-volume favourable set in $\R^N$.  The defining whole-space inequality at the optimal value $k^2=\Lambda_{\R^N}$ yields
\[
\int_{\R^N}|\nabla w|^2
\ge
k^2\int_{\R^N}m_Aw^2.
\]
Since $m_A=-1$ in $L$, rearranging gives
\[
k^2\mathcal M(v)-\int_{\Hh}|\nabla v|^2
\le
\int_L\bigl(|\nabla z|^2+k^2z^2\bigr),
\]
and \eqref{eq:DtN-defect} follows from \eqref{eq:yukawa-energy}.  If $v$ witnesses a strict gap, Proposition~\ref{prop:tau-star-trace-defect} gives $\mathfrak D(v)>\tau\|f\|_2^2$; combining this with \eqref{eq:DtN-defect} proves \eqref{eq:boundary-frequency-necessary}.
\end{proof}

\subsection{Yukawa decomposition and quantitative estimates near the whole-space value}\label{subsec:yukawa-rigidity}

The Dirichlet-to-Neumann estimate is sharp for a prescribed trace, but the unbounded symbol $S_k$ is not controlled by the $L^2$ trace norm.  Combining the same Yukawa extension with the whole-space bathtub inequality gives the quantitative remainder estimate needed below.

For $f\in H^{1/2}(\R^{N-1})$, let $\mathcal Y_k f$ denote the decaying Yukawa extension into the upper half-space,
\begin{equation}\label{eq:upper-yukawa-extension}
\widehat{\mathcal Y_k f}(\xi,x_N)
=e^{-s(\xi)x_N}\widehat f(\xi),
\qquad
s(\xi)=\sqrt{k^2+|\xi|^2}.
\end{equation}
Set
\[
\mathcal E_k(q):=\int_{\Hh}\bigl(|\nabla q|^2+k^2q^2\bigr),
\qquad
c_k:=(\kappa+1)k^2.
\]
If $v\in H^1(\Hh)$ has trace $f$, write
\begin{equation}\label{eq:yukawa-orthogonal-decomp}
z:=\mathcal Y_kf,
\qquad
w:=v-z\in H_0^1(\Hh).
\end{equation}

\begin{lemma}[Orthogonal Yukawa decomposition]\label{lem:yukawa-orthogonal}
With the notation above,
\begin{align}
\mathcal E_k(v)&=\mathcal E_k(w)+\mathcal E_k(z),\label{eq:Ek-orthogonal}\\
\mathcal E_k(z)&=\langle f,S_kf\rangle,\label{eq:Ek-z-DtN}\\
\|z\|_{L^2(\Hh)}^2
&=\frac12\int_{\R^{N-1}}\frac{|\widehat f(\xi)|^2}{s(\xi)}\,d\xi
\le \frac{1}{2k}\|f\|_2^2.\label{eq:yukawa-L2-bound}
\end{align}
Let $W$ be the zero extension of $w$ to $\R^N$.  Then
\begin{equation}\label{eq:whole-deficit-W}
\delta_{\R^N}(W)
:=\int_{\R^N}|\nabla W|^2-k^2\mathcal M_1^{\R^N}(W)
=\mathcal E_k(w)-c_k\mathcal B(w)\ge0.
\end{equation}
Finally, the difference between the whole-space reference value and the half-space energy has the exact decomposition
\begin{equation}\label{eq:defect-yukawa-exact}
\mathfrak D(v)
=c_k\bigl(\mathcal B(v)-\mathcal B(w)\bigr)
-\mathcal E_k(z)-\delta_{\R^N}(W).
\end{equation}
\end{lemma}

\begin{proof}
Because $z$ solves $(-\Delta+k^2)z=0$ in $\Hh$ and $w$ has zero trace, integration by parts gives
\[
\int_{\Hh}(\nabla w\cdot\nabla z+k^2wz)=0,
\]
which proves \eqref{eq:Ek-orthogonal}.  Formula \eqref{eq:Ek-z-DtN} follows from the Fourier representation \eqref{eq:upper-yukawa-extension}, exactly as in \eqref{eq:yukawa-energy}; integrating $e^{-2s x_N}$ in the normal variable gives \eqref{eq:yukawa-L2-bound}.  Since $w\in H_0^1(\Hh)$, its zero extension belongs to $H^1(\R^N)$, and the whole-space variational inequality at volume one yields \eqref{eq:whole-deficit-W}.  Finally,
\[
\mathfrak D(v)
=c_k\mathcal B(v)-\mathcal E_k(v),
\]
while \eqref{eq:Ek-orthogonal} and \eqref{eq:whole-deficit-W} give
\[
\mathcal E_k(v)
=c_k\mathcal B(w)+\delta_{\R^N}(W)+\mathcal E_k(z),
\]
which is \eqref{eq:defect-yukawa-exact}.
\end{proof}

\begin{proposition}[Estimate of the energy deficit by the boundary trace]\label{prop:sqrt-trace-defect}
There exists a constant $C_*=C_*(N,\kappa)>0$ such that every $v\in H^1(\Hh)$ with $\mathcal M(v)>0$ and $\mathfrak D(v)>0$ satisfies
\begin{equation}\label{eq:sqrt-trace-defect}
\mathfrak D(v)
\le
C_*\,\mathcal M(v)^{1/2}
\|\operatorname{Tr}v\|_{L^2(\partial\Hh)}.
\end{equation}
If $\mathcal M(v)=1$ and $f=\operatorname{Tr}v$, then the decomposition in Lemma~\ref{lem:yukawa-orthogonal} additionally satisfies
\begin{equation}\label{eq:yukawa-small-remainders}
\mathcal E_k(z)+\delta_{\R^N}(W)
\le C_*\|f\|_2.
\end{equation}
\end{proposition}

\begin{proof}
By homogeneity it is enough to assume $\mathcal M(v)=1$.  Since $\mathfrak D(v)>0$,
\[
\int_{\Hh}|\nabla v|^2<k^2.
\]
The proof of Lemma~\ref{lem:normalised-bounded}, with $\tau=0$ and energy bound $k^2$, therefore gives a uniform bound
\begin{equation}\label{eq:positive-defect-H1-bound}
\|v\|_{H^1(\Hh)}\le C_0(N,\kappa).
\end{equation}
The trace theorem gives a corresponding uniform bound on $\|f\|_2$.  By \eqref{eq:yukawa-L2-bound},
\[
\|z\|_2\le (2k)^{-1/2}\|f\|_2,
\]
and hence $\|w\|_2$ is uniformly bounded as well.  The $L^2$ continuity estimate \eqref{eq:B-L2-continuity} gives
\begin{align*}
|\mathcal B(v)-\mathcal B(w)|
&\le(\|v\|_2+\|w\|_2)\|z\|_2\\
&\le C_1\|f\|_2.
\end{align*}
Using the exact identity \eqref{eq:defect-yukawa-exact} and the nonnegativity of both $\mathcal E_k(z)$ and $\delta_{\R^N}(W)$ yields
\[
0<\mathfrak D(v)\le c_k C_1\|f\|_2,
\]
which proves \eqref{eq:sqrt-trace-defect} under the normalisation $\mathcal M(v)=1$.  The same identity, rearranged as
\[
\mathcal E_k(z)+\delta_{\R^N}(W)
=c_k\bigl(\mathcal B(v)-\mathcal B(w)\bigr)-\mathfrak D(v),
\]
gives \eqref{eq:yukawa-small-remainders}.  Rescaling $v$ by $\mathcal M(v)^{-1/2}$ proves the homogeneous form \eqref{eq:sqrt-trace-defect}.
\end{proof}

\begin{theorem}[Quantitative convergence to the whole-space value for large Robin parameter]\label{thm:large-tau-rate}
There exists $C_*=C_*(N,\kappa)>0$ such that, for every $\tau>0$,
\begin{equation}\label{eq:large-tau-rate}
0\le
\Lambda_{\R^N}-\Lambda_{\Hh}(\tau)
\le
\frac{C_*^2}{4\tau}.
\end{equation}
In particular,
\[
\Lambda_{\Hh}(\tau)\longrightarrow\Lambda_{\R^N}
\qquad\text{as }\tau\to\infty.
\]
Whenever $\Lambda_{\Hh}(\tau)<\Lambda_{\R^N}$ and $v_\tau$ is a minimiser normalised by $\mathcal M(v_\tau)=1$, its trace $f_\tau$ satisfies
\begin{equation}\label{eq:trace-rate}
\|f_\tau\|_2\le\frac{C_*}{\tau},
\end{equation}
and, with the Yukawa decomposition of Lemma~\ref{lem:yukawa-orthogonal},
\begin{equation}\label{eq:yukawa-remainder-rate}
\mathcal E_k(z_\tau)+\delta_{\R^N}(W_\tau)
\le\frac{C_*^2}{\tau}.
\end{equation}
\end{theorem}

\begin{proof}
If $\Lambda_{\Hh}(\tau)=k^2$, there is nothing to prove.  Otherwise Theorem~\ref{thm:strict-gap-attainment} supplies a minimiser $v_\tau$ with $\mathcal M(v_\tau)=1$.  Put $t=\|\operatorname{Tr}v_\tau\|_2$ and $\mathfrak D=\mathfrak D(v_\tau)$.  Since
\[
k^2-\Lambda_{\Hh}(\tau)
=\mathfrak D-\tau t^2>0,
\]
Proposition~\ref{prop:sqrt-trace-defect} gives
\[
0<k^2-\Lambda_{\Hh}(\tau)
\le C_*t-\tau t^2
\le\frac{C_*^2}{4\tau},
\]
which is \eqref{eq:large-tau-rate}.  The positivity of the left-hand side also implies
\[
\tau t^2<\mathfrak D\le C_*t,
\]
so \eqref{eq:trace-rate} follows.  Finally \eqref{eq:yukawa-remainder-rate} is a direct consequence of \eqref{eq:yukawa-small-remainders} and \eqref{eq:trace-rate}.
\end{proof}

The preceding theorem is stronger than a merely qualitative large-$\tau$ compactness statement and is the form used below.

\begin{proposition}[Whole-space compactness at the optimal level]\label{prop:whole-space-compactness}
Let $U_n\in H^1(\R^N)$ satisfy
\[
\mathcal M_1^{\R^N}(U_n)=1,
\qquad
\int_{\R^N}|\nabla U_n|^2\longrightarrow k^2.
\]
Then there exist translations $x_n\in\R^N$ and a whole-space optimiser $\Phi$ such that, after passing to a subsequence,
\begin{equation}\label{eq:whole-space-compactness}
U_n(x_n+\cdot)\longrightarrow\Phi
\qquad\text{strongly in }H^1(\R^N).
\end{equation}
By the known rearrangement characterisation, $\Phi$ may be chosen positive and radial and its favourable bathtub set is the ball $B_{R_1}$.  If $A_n$ are bathtub maximisers for $U_n^2$, then the translations may be chosen so that
\begin{equation}\label{eq:whole-bathtub-set-convergence}
|(A_n-x_n)\triangle B_{R_1}|\longrightarrow0.
\end{equation}
\end{proposition}

\begin{proof}
The argument is the whole-space version of the concentration--compactness proof of Theorem~\ref{thm:strict-gap-attainment}.  The normalisation and Sobolev (or, for $N=2$, Gagliardo--Nirenberg) estimates give a uniform $H^1$ bound.  Vanishing would imply $\mathcal B_1^{\R^N}(U_n)\to0$ and contradict $\mathcal M_1^{\R^N}(U_n)=1$.  Dichotomy is excluded by the exact scaling
\(
\Lambda_{\R^N}(V)=V^{-2/N}k^2
\): if both pieces receive nontrivial favourable volume, strict subhomogeneity raises the energy above $k^2$; if one volume share tends to zero, the small-volume bathtub estimate forces that piece to retain a strictly positive $H^1$ remainder.  Hence, after translations, the sequence is tight.

Local Rellich compactness and tightness then give strong $L^2$ convergence to some $\Phi$, continuity of the bathtub functional gives $\mathcal M_1^{\R^N}(\Phi)=1$, and lower semicontinuity yields
\(
 k^2\le\|\nabla\Phi\|_2^2\le\liminf_n\|\nabla U_n\|_2^2=k^2.
\)
Thus convergence is strong in $H^1$ and $\Phi$ is a whole-space optimiser.  Proposition~\ref{prop:whole-space-classification} identifies its bathtub set, up to translation, with $B_{R_1}$; strong $L^2$ convergence and Lemma~\ref{lem:bathtub-stability} then give \eqref{eq:whole-bathtub-set-convergence}.
\end{proof}

\begin{theorem}[Asymptotic structure of large-Robin optimisers below the whole-space value]\label{thm:two-scale-rigidity}
Let $\tau_n\to\infty$ and suppose
\[
\Lambda_{\Hh}(\tau_n)<\Lambda_{\R^N}
\qquad\text{for every }n.
\]
Let $v_n$ be corresponding nonnegative minimisers normalised by $\mathcal M(v_n)=1$, let $f_n=\operatorname{Tr}v_n$, and write
\[
z_n=\mathcal Y_k f_n,
\qquad
w_n=v_n-z_n,
\]
with $W_n$ the zero extension of $w_n$ to $\R^N$.  Then
\begin{align}
\|f_n\|_2&=O(\tau_n^{-1}),\label{eq:two-scale-trace}\\
\|z_n\|_{H^1(\Hh)}&=O(\tau_n^{-1/2}),\label{eq:two-scale-z}\\
\mathcal M_1^{\R^N}(W_n)&=1+O(\tau_n^{-1}),\label{eq:two-scale-MW}\\
\int_{\R^N}|\nabla W_n|^2
-k^2\mathcal M_1^{\R^N}(W_n)&=O(\tau_n^{-1}).\label{eq:two-scale-deficit}
\end{align}
Consequently, after multiplying $W_n$ by $1+O(\tau_n^{-1})$ and passing to a subsequence, there exist points
\[
x_n=(x_n',h_n)\in\Hh,
\qquad h_n\to\infty,
\]
and a positive radial whole-space optimiser $\Phi$ such that
\begin{equation}\label{eq:two-scale-ball-convergence}
W_n(x_n+\cdot)\longrightarrow\Phi
\qquad\text{strongly in }H^1(\R^N).
\end{equation}
Moreover, if $A_n$ denotes the optimal favourable set associated with $v_n$, then
\begin{equation}\label{eq:two-scale-set-convergence}
|(A_n-x_n)\triangle B_{R_1}|\longrightarrow0.
\end{equation}
Thus, if optimisers with value below the whole-space optimum were to persist for arbitrarily large $\tau$, their only possible leading-order structure would be a whole-space ball moving away from the boundary, together with an $H^1$-vanishing boundary correction.
\end{theorem}

\begin{proof}
The trace estimate \eqref{eq:two-scale-trace} is \eqref{eq:trace-rate}, and \eqref{eq:two-scale-z} follows from \eqref{eq:yukawa-remainder-rate} because $\mathcal E_k$ controls the $H^1$ norm.  Since $\|z_n\|_2=O(\tau_n^{-1})$ by \eqref{eq:yukawa-L2-bound}, the $L^2$ continuity of $\mathcal B$ and the uniform $L^2$ bounds give
\[
\mathcal M_1^{\R^N}(W_n)
=\mathcal M(w_n)
=\mathcal M(v_n)+O(\tau_n^{-1})
=1+O(\tau_n^{-1}),
\]
which proves \eqref{eq:two-scale-MW}.  Estimate \eqref{eq:two-scale-deficit} is exactly the second nonnegative term in \eqref{eq:yukawa-remainder-rate}.  After a scalar renormalisation imposing $\mathcal M_1^{\R^N}=1$, Proposition~\ref{prop:whole-space-compactness} yields translations $x_n$ and \eqref{eq:two-scale-ball-convergence}.

It remains to determine the normal component of $x_n$.  Since $W_n$ is identically zero on $\{x_N<0\}$, the translated function $W_n(x_n+\cdot)$ is zero on the half-space $\{y_N<-h_n\}$.  If $(h_n)$ had a bounded subsequence, the strong limit $\Phi$ would vanish on a nonempty lower half-space, contradicting positivity of the whole-space principal optimiser.  If $h_n\to-\infty$, the translated functions would vanish on every fixed compact set for large $n$, again contradicting \eqref{eq:two-scale-ball-convergence}.  Hence $h_n\to+\infty$.

Finally, extend $v_n$ by zero merely as an $L^2$ function.  Its $L^2$ distance from $W_n$ is $\|z_n\|_2=O(\tau_n^{-1})$.  Therefore the translated squared densities converge in $L^1$ to $\Phi^2$.  The sets $A_n$ are bathtub maximisers for these densities; since the limiting radial profile has a unique volume-one superlevel ball and no plateau at the threshold, the bathtub stability argument from Proposition~\ref{prop:whole-space-compactness} gives \eqref{eq:two-scale-set-convergence}.
\end{proof}

\subsection{Normal translation, boundary contact, and finiteness of the threshold}\label{subsec:finite-detachment}

The finiteness of the threshold follows from a geometric translation argument.  An optimiser with value below $k^2$ whose favourable set stays a positive distance from the boundary admits a one-sided normal translation toward the boundary.  The first variation of that translation is incompatible, at large Robin parameter, with the whole-space lower bound.  Thus every large-$\tau$ optimiser with value below $k^2$ must retain boundary contact.  The asymptotic description above then confines any remaining contact to a vanishing secondary component while the main ball-like core moves away from the boundary.  A blow-up at a boundary maximum on the Robin length scale $\tau^{-1}$ excludes this possibility.

\begin{lemma}[One-sided normal translation away from the boundary]\label{lem:detached-translation}
Let $\tau>0$ satisfy
\[
\lambda:=\Lambda_{\Hh}(\tau)<k^2,
\]
and let $(A,v)$ be a nonnegative optimiser, normalised by
\[
\int_{\Hh}m_Av^2=1.
\]
Write $f=\operatorname{Tr}v$.  If there exists $d>0$ such that
\begin{equation}\label{eq:detached-hostile-strip}
|A\cap\{0<x_N<d\}|=0,
\end{equation}
then $f\in H^1(\R^{N-1})$ and
\begin{equation}\label{eq:detached-frequency-upper}
\int_{\R^{N-1}}\bigl(\lambda+|\xi|^2\bigr)|\widehat f(\xi)|^2\,d\xi
\le
\tau^2\|f\|_2^2.
\end{equation}
\end{lemma}

\begin{proof}
Fix $d>0$ for which \eqref{eq:detached-hostile-strip} holds.  In the strip $0<x_N<d$ the coefficient is identically $-1$, hence
\[
(-\Delta+\lambda)v=0.
\]
Standard elliptic regularity for the constant-coefficient Robin problem in a boundary strip gives $f\in H^1(\R^{N-1})$ and justifies the following differentiation.

For $0<s<d$ define
\[
A_s:=A-se_N,
\qquad
v_s(x',x_N):=v(x',x_N+s).
\]
Then $A_s\subset\Hh$ and $(A_s,v_s)$ is admissible.  Set
\[
D(s):=\int_{\Hh}m_{A_s}v_s^2,
\qquad
N(s):=\int_{\Hh}|\nabla v_s|^2
+\tau\int_{\partial\Hh}v_s^2.
\]
Changing variables $y_N=x_N+s$ and using $m_A=-1$ on $0<y_N<d$ gives
\[
D(s)
=1+\int_{0<y_N<s}v(y)^2\,dy,
\]
so that
\begin{equation}\label{eq:normal-shift-Dprime}
D'(0+)=\|f\|_2^2.
\end{equation}
Similarly,
\[
N(s)
=\int_{y_N>s}|\nabla v(y)|^2\,dy
+\tau\int_{\R^{N-1}}v(x',s)^2\,dx'.
\]
Since the Robin condition in inward-normal coordinates is
\[
\partial_{x_N}v(\cdot,0)=\tau f,
\]
we obtain
\begin{align}
N'(0+)
&=-\int_{\R^{N-1}}
\left(|\nabla_{x'}f|^2+|\partial_{x_N}v(\cdot,0)|^2\right)
+2\tau\int_{\R^{N-1}}f\,\partial_{x_N}v(\cdot,0)\notag\\
&=-\|\nabla_{x'}f\|_2^2+\tau^2\|f\|_2^2.
\label{eq:normal-shift-Nprime}
\end{align}
The quotient $N(s)/D(s)$ is bounded below by the optimal value $\lambda$ and equals $\lambda$ at $s=0$.  Hence its right derivative is nonnegative.  Combining \eqref{eq:normal-shift-Dprime} and \eqref{eq:normal-shift-Nprime} yields
\[
\|\nabla_{x'}f\|_2^2
\le
(\tau^2-\lambda)\|f\|_2^2,
\]
which is exactly \eqref{eq:detached-frequency-upper} by Plancherel.
\end{proof}

\begin{proposition}[Boundary contact for large-$\tau$ optimisers below the whole-space value]\label{prop:large-tau-contact}
There exists $\tau_{\rm c}=\tau_{\rm c}(N,\kappa)<\infty$ with the following property.  If $\tau\ge\tau_{\rm c}$ and
\[
\Lambda_{\Hh}(\tau)<k^2,
\]
then every optimal favourable set $A$ associated with a normalised nonnegative minimiser has essential boundary contact, in the sense that
\begin{equation}\label{eq:contact-boundary}
|A\cap\{0<x_N<d\}|>0
\qquad\text{for every }d>0.
\end{equation}
\end{proposition}

\begin{proof}
Suppose, to the contrary, that a sequence $\tau\to\infty$ admits optimisers with values below $k^2$ and a strip free of favourable material, i.e., satisfying \eqref{eq:detached-hostile-strip} for some $d=d_\tau>0$.  Write
\[
\lambda=\Lambda_{\Hh}(\tau),
\qquad
q:=k^2-\lambda>0,
\qquad
f=\operatorname{Tr}v.
\]
By Theorem~\ref{thm:large-tau-rate},
\begin{equation}\label{eq:contact-trace-small}
\|f\|_2\le \frac{C_*}{\tau},
\qquad
\lambda\longrightarrow k^2.
\end{equation}

Let
\[
s_\lambda(\xi):=\sqrt{\lambda+|\xi|^2},
\qquad
E_\lambda(f):=\int s_\lambda(\xi)|\widehat f(\xi)|^2\,d\xi.
\]
Lemma~\ref{lem:detached-translation} and Cauchy--Schwarz give
\begin{equation}\label{eq:Elambda-upper}
E_\lambda(f)^2
\le
\|f\|_2^2
\int s_\lambda(\xi)^2|\widehat f(\xi)|^2\,d\xi
\le
\tau^2\|f\|_2^4,
\end{equation}
so
\[
E_\lambda(f)\le\tau\|f\|_2^2.
\]

Now extend $f$ into the lower half-space $L=\{x_N<0\}$ by the decaying $\lambda$-Yukawa extension
\[
\widehat z_-(\xi,x_N)
=e^{s_\lambda(\xi)x_N}\widehat f(\xi).
\]
Glue $z_-$ to $v$ across $\partial\Hh$ and call the resulting whole-space function $U$.  Put
\[
L_\lambda(f):=\|z_-\|_{L^2(L)}^2
=\frac12\int\frac{|\widehat f(\xi)|^2}{s_\lambda(\xi)}\,d\xi.
\]
The lower-half-space Dirichlet energy is
\[
\int_L|\nabla z_-|^2
=E_\lambda(f)-\lambda L_\lambda(f).
\]
Since $v$ is normalised and has energy
\[
\int_{\Hh}|\nabla v|^2
=\lambda-\tau\|f\|_2^2,
\]
we obtain
\begin{equation}\label{eq:glued-lambda-energy}
\int_{\R^N}|\nabla U|^2
=\lambda\bigl(1-L_\lambda(f)\bigr)
+E_\lambda(f)-\tau\|f\|_2^2.
\end{equation}
Extend the favourable set by keeping the whole lower half-space unfavourable.  For this fixed unit-volume set the whole-space optimal inequality gives
\[
\int_{\R^N}|\nabla U|^2
\ge
k^2\int_{\R^N}m_AU^2
=k^2\bigl(1-L_\lambda(f)\bigr).
\]
Consequently,
\begin{equation}\label{eq:Elambda-lower}
E_\lambda(f)-\tau\|f\|_2^2
\ge
q\bigl(1-L_\lambda(f)\bigr).
\end{equation}
But \eqref{eq:contact-trace-small} and $\lambda\to k^2>0$ imply
\[
L_\lambda(f)
\le
\frac{1}{2\sqrt\lambda}\|f\|_2^2
=o(1).
\]
Thus $1-L_\lambda(f)>0$ for all sufficiently large $\tau$, and because $q>0$, \eqref{eq:Elambda-lower} yields
\[
E_\lambda(f)>\tau\|f\|_2^2,
\]
contradicting \eqref{eq:Elambda-upper}.  Hence an optimiser with value below $k^2$ cannot remain separated from the boundary by a strip of positive width for all sufficiently large $\tau$, which is exactly \eqref{eq:contact-boundary}.
\end{proof}

\begin{lemma}[Convergence of the bathtub levels]\label{lem:bathtub-level-convergence}
Assume that $\tau_n\to\infty$ and that minimisers $(A_n,v_n)$ with values below $k^2$ exist, normalised as in Theorem~\ref{thm:two-scale-rigidity}.  Let $x_n$ and $\Phi$ be supplied by that theorem, and choose bathtub levels $\theta_n\ge0$ such that
\begin{equation}\label{eq:theta-n-bathtub}
\{v_n>\theta_n\}\subset A_n\subset\{v_n\ge\theta_n\}.
\end{equation}
If $R_1$ is the radius of the unit-volume ball, then
\begin{equation}\label{eq:theta-n-limit}
\theta_n\longrightarrow\theta_*:=\Phi(R_1)>0.
\end{equation}
\end{lemma}

\begin{proof}
Extend $v_n$ by zero to the lower half-space merely as an $L^2$ function and translate by $x_n$.  The proof of Theorem~\ref{thm:two-scale-rigidity} gives
\[
v_n(x_n+\cdot)\longrightarrow\Phi
\qquad\text{strongly in }L^2(\R^N),
\]
where the notation includes the zero extension.  The radial optimiser $\Phi$ is strictly decreasing through the radius $R_1$, and therefore
\[
|\{\Phi>\theta_*+\varepsilon\}|<1
<|\{\Phi>\theta_*-\varepsilon\}|
\qquad(\varepsilon>0\ \hbox{small}).
\]
Strong $L^2$ convergence implies convergence in measure, and hence the corresponding strict superlevel measures at every fixed level which is not a plateau level of $\Phi$.

If $\theta_n\ge\theta_*+\varepsilon$ along a subsequence, then by \eqref{eq:theta-n-bathtub}
\[
1=|A_n|
\le |\{v_n\ge\theta_*+\varepsilon\}|,
\]
whereas the right-hand side has limsup strictly smaller than one.  If $\theta_n\le\theta_*-\varepsilon$, then
\[
1=|A_n|
\ge |\{v_n>\theta_*-\varepsilon\}|,
\]
whose liminf is strictly larger than one.  Both alternatives are impossible, proving \eqref{eq:theta-n-limit}.
\end{proof}

\begin{lemma}[Uniform boundedness of normalised eigenfunctions below the whole-space value]\label{lem:uniform-Linfty-branch}
There exists $C_\infty=C_\infty(N,\kappa)$ such that every nonnegative minimiser $v$ with value below $k^2$ normalised by $\mathcal M(v)=1$ satisfies
\begin{equation}\label{eq:uniform-Linfty-branch}
\|v\|_{L^\infty(\Hh)}\le C_\infty.
\end{equation}
\end{lemma}

\begin{proof}
The normalised $H^1$ bound \eqref{eq:positive-defect-H1-bound} applies to every such minimiser because its Dirichlet energy is strictly smaller than $k^2$.  Moreover
\[
-\Delta v=\lambda m_Av,
\qquad
0<\lambda\le k^2,
\qquad
-1\le m_A\le\kappa.
\]
A Moser iteration in the half-space gives a uniform $L^\infty$ estimate from the uniform $L^2$ bound.  For $p\ge2$, test first with $\eta_R^2v^{p-1}$, where $\eta_R$ is a standard cutoff, and then let $R\to\infty$.  The Robin contribution is
\[
\tau\int_{\partial\Hh}v^p\,dS\ge0,
\]
so it may be discarded in the upper estimate, whereas the right-hand side is at most $\kappa k^2\int_{\Hh}v^p$.  After the usual absorption of the cutoff-gradient term, the half-space Sobolev inequality (with any fixed finite Sobolev exponent when $N=2$) yields the standard exponent iteration.  Since the starting $L^2$ norm is uniformly bounded by \eqref{eq:positive-defect-H1-bound}, all constants depend only on $N$ and $\kappa$, and in particular are independent of $\tau$ and of $A$.
\end{proof}

\begin{lemma}[Boundary contact gives a lower bound for the boundary values]\label{lem:contact-forces-boundary-amplitude}
Let $\tau>0$, let $a\in L^\infty(\Hh)$, and suppose that a bounded nonnegative weak solution $v$ satisfies
\[
-\Delta v=a(x)v\quad\text{in }\Hh,
\qquad
\partial_{x_N}v=\tau v\quad\text{on }\partial\Hh.
\]
Let $A\subset\Hh$ be measurable and let $\theta>0$ be such that
\[
A\subset\{v\ge\theta\}\quad\text{a.e.}
\]
and
\[
|A\cap\{0<x_N<d\}|>0
\qquad\text{for every }d>0.
\]
Then
\begin{equation}\label{eq:contact-forces-boundary-amplitude}
\|v\|_{L^\infty(\partial\Hh)}\ge\theta.
\end{equation}
\end{lemma}

\begin{proof}
For each fixed solution, local up-to-boundary $W^{2,p}$ estimates with any $p>N$ give a Hölder modulus in unit half-balls whose constant is uniform under tangential translations; the constant may depend on $\tau$, $\|a\|_\infty$, and $\|v\|_\infty$, which is sufficient here.  Choose Lebesgue points $x_j=(x_j',x_{j,N})\in A$ with $x_{j,N}\downarrow0$ and with $v(x_j)\ge\theta$.  The uniform tangentially translated boundary modulus gives
\[
|v(x_j',x_{j,N})-v(x_j',0)|\longrightarrow0.
\]
Hence
\[
\sup_{x'\in\R^{N-1}}v(x',0)
\ge\limsup_j v(x_j',0)
\ge\theta,
\]
which proves \eqref{eq:contact-forces-boundary-amplitude}.
\end{proof}

\begin{lemma}[Boundary maximum estimate at the Robin scale]\label{lem:robin-scale-obstruction}
Let $\tau_n\to\infty$.  Suppose $v_n\ge0$ satisfy
\begin{equation}\label{eq:robin-scale-general-PDE}
-\Delta v_n=a_n(x)v_n\quad\text{in }\Hh,
\qquad
\partial_{x_N}v_n=\tau_nv_n\quad\text{on }\partial\Hh,
\end{equation}
with
\[
\|a_n\|_{L^\infty(\Hh)}\le C,
\qquad
\|v_n\|_{L^\infty(\Hh)}\le C.
\]
Then
\begin{equation}\label{eq:boundary-max-positive}
\|v_n\|_{L^\infty(\partial\Hh)}\longrightarrow0.
\end{equation}
\end{lemma}

\begin{proof}
Assume the conclusion fails.  After passing to a subsequence there exists $c_0>0$ such that
\[
\|v_n\|_{L^\infty(\partial\Hh)}\ge c_0.
\]
Choose $p_n'\in\R^{N-1}$ so that
\[
v_n(p_n',0)
\ge
\|v_n\|_{L^\infty(\partial\Hh)}-\frac1n.
\]
After passing to a subsequence the boundary suprema converge to some $M>0$.  Define the Robin-scale blow-up
\[
V_n(y',y_N)
:=v_n\left(p_n'+\frac{y'}{\tau_n},\frac{y_N}{\tau_n}\right).
\]
Then
\begin{equation}\label{eq:robin-scale-rescaled}
-\Delta_yV_n
=\tau_n^{-2}a_n\left(p_n'+\frac{y'}{\tau_n},\frac{y_N}{\tau_n}\right)V_n,
\qquad
\partial_{y_N}V_n=V_n
\quad\text{on }y_N=0.
\end{equation}
The functions $V_n$ are uniformly bounded.  Local boundary $W^{2,p}$ estimates for the flat Robin problem, followed by Sobolev embedding and a diagonal extraction, yield
\[
V_n\longrightarrow V
\qquad\text{in }C^1_{\mathrm{loc}}(\overline{\Hh})
\]
for a bounded nonnegative function satisfying
\begin{equation}\label{eq:robin-scale-limit}
\Delta V=0\quad\text{in }\Hh,
\qquad
\partial_{y_N}V=V\quad\text{on }\partial\Hh.
\end{equation}
Moreover,
\[
V(0,0)=M,
\qquad
V(y',0)\le M\quad\text{for all }y'\in\R^{N-1}.
\]
A bounded harmonic function on the half-space is the Poisson extension of its bounded boundary trace; in particular it cannot exceed the boundary supremum.  Hence
\[
V(y',y_N)\le M
\qquad\text{throughout }\Hh.
\]
On the other hand, \eqref{eq:robin-scale-limit} gives
\[
\partial_{y_N}V(0,0)=V(0,0)=M>0.
\]
Therefore
\[
V(0,y_N)=M+My_N+o(y_N)>M
\]
for all sufficiently small $y_N>0$, a contradiction.
\end{proof}

\begin{theorem}[Finiteness of the transition threshold]\label{thm:finite-detachment}
The strict inequality $\Lambda_{\Hh}(\tau)<\Lambda_{\R^N}$ cannot persist for arbitrarily large Robin parameter:
\begin{equation}\label{eq:finite-detachment}
0<\tau_*<\infty.
\end{equation}
Consequently,
\begin{equation}\label{eq:ceiling-after-tau-star}
\Lambda_{\Hh}(\tau)=\Lambda_{\R^N}
\qquad\text{for every }\tau\ge\tau_*.
\end{equation}
Equivalently, the global boundary-trace inequality
\begin{equation}\label{eq:global-quadratic-trace-final}
\Lambda_{\R^N}\mathcal M(v)
\le
\int_{\Hh}|\nabla v|^2
+\tau_*\int_{\partial\Hh}v^2
\qquad(v\in H^1(\Hh))
\end{equation}
holds with the sharp constant $\tau_*<\infty$.
\end{theorem}

\begin{proof}
Positivity follows from Corollary~\ref{cor:threshold-structure}; it remains to prove finiteness. Assume for contradiction that $\tau_*=\infty$.  Choose $\tau_n\to\infty$.  Then
\[
\Lambda_{\Hh}(\tau_n)<k^2
\]
for every $n$, and Theorem~\ref{thm:strict-gap-attainment} supplies nonnegative normalised optimisers $(A_n,v_n)$.  By Proposition~\ref{prop:large-tau-contact}, after discarding finitely many terms,
\begin{equation}\label{eq:An-contact-infinite-branch}
|A_n\cap\{0<x_N<d\}|>0
\qquad\text{for every }d>0.
\end{equation}

Let $\theta_n$ be bathtub levels as in \eqref{eq:theta-n-bathtub}.  Theorem~\ref{thm:two-scale-rigidity} and Lemma~\ref{lem:bathtub-level-convergence} give
\[
\theta_n\longrightarrow\theta_*=\Phi(R_1)>0.
\]
Because $A_n\subset\{v_n\ge\theta_n\}$ and \eqref{eq:An-contact-infinite-branch} holds for every $d>0$, Lemma~\ref{lem:contact-forces-boundary-amplitude} gives
\begin{equation}\label{eq:trace-sup-lower-theta}
\|v_n\|_{L^\infty(\partial\Hh)}\ge\theta_n.
\end{equation}

Lemma~\ref{lem:uniform-Linfty-branch} gives a uniform $L^\infty$ bound on $v_n$, while
\[
a_n(x):=\Lambda_{\Hh}(\tau_n)m_{A_n}(x)
\]
is uniformly bounded.  Equations \eqref{eq:EL-strong-formal} therefore place the sequence exactly in the setting of Lemma~\ref{lem:robin-scale-obstruction}.  But \eqref{eq:trace-sup-lower-theta} and $\theta_n\to\theta_*>0$ give
\[
\liminf_n\|v_n\|_{L^\infty(\partial\Hh)}\ge\theta_*>0,
\]
whereas Lemma~\ref{lem:robin-scale-obstruction} says that these boundary suprema converge to zero.  Hence $\tau_*<\infty$.

The equality \eqref{eq:ceiling-after-tau-star} is Corollary~\ref{cor:threshold-structure}.  Finally, Proposition~\ref{prop:tau-star-trace-defect} identifies the finite number $\tau_*$ with the sharp constant in \eqref{eq:global-quadratic-trace-final}.
\end{proof}

\begin{corollary}[Nonattainment above the transition threshold]\label{cor:nonattainment-above-threshold}
For every $\tau>\tau_*$,
\[
\Lambda_{\Hh}(\tau)=\Lambda_{\R^N}=k^2,
\]
but the half-space infimum is not attained.  In particular, no minimising sequence at such a parameter can be strongly precompact in $H^1(\Hh)$ modulo tangential translations.
\end{corollary}

\begin{proof}
Suppose that $\tau>\tau_*$ and that a nonzero minimiser $v$ exists.  By homogeneity normalise $\mathcal M(v)=1$.  Theorem~\ref{thm:finite-detachment} gives
\[
\int_{\Hh}|\nabla v|^2+\tau\|\operatorname{Tr}v\|_2^2=k^2.
\]
Applying the sharp boundary-trace inequality \eqref{eq:global-quadratic-trace-final} at $\tau_*$ to the same function gives
\[
k^2
\le
\int_{\Hh}|\nabla v|^2+\tau_*\|\operatorname{Tr}v\|_2^2
=
k^2-(\tau-\tau_*)\|\operatorname{Tr}v\|_2^2.
\]
Hence $\operatorname{Tr}v=0$.  Let $A$ be a bathtub maximiser for $v$, so that
\[
\int_{\Hh}m_Av^2=\mathcal M(v)=1.
\]
The zero extension $\widetilde v$ then belongs to $H^1(\R^N)$ and satisfies
\[
\int_{\R^N}|\nabla\widetilde v|^2=k^2,
\qquad
\int_{\R^N}m_A\widetilde v^2=1,
\]
where the lower half-space is declared unfavourable.  Thus $(A,\widetilde v)$ attains the whole-space optimum.  A nonnegative whole-space principal optimiser is strictly positive by the strong maximum principle (equivalently, by the radial classification recalled above), whereas $\widetilde v$ vanishes identically on the open lower half-space.  This contradiction proves nonattainment.

If a half-space minimising sequence were strongly precompact modulo tangential translations, its strong limit would attain the infimum, which has just been ruled out.
\end{proof}

\begin{lemma}[Boundary values vanish under escape from the boundary]\label{lem:boundary-sup-escape}
Let \(\lambda_n\) remain bounded, let \(\tau_n\ge0\), and let \(A_n\subset\Hh\) be measurable.  Suppose nonnegative functions \(v_n\in H^1(\Hh)\) solve
\[
-\Delta v_n=\lambda_n m_{A_n}v_n\quad\hbox{in }\Hh,
\qquad
\partial_{x_N}v_n=\tau_n v_n\quad\hbox{on }\partial\Hh,
\]
and are uniformly \(H^1\)-tight about centres \((0,h_n)\) with
\(h_n\to+\infty\).  Then
\begin{equation}\label{eq:boundary-sup-escape}
\|v_n\|_{L^\infty(\partial\Hh)}\longrightarrow0.
\end{equation}
\end{lemma}

\begin{proof}
Tightness implies
\[
\int_{\{0<x_N<1\}}v_n^2\longrightarrow0.
\]
Indeed, for every \(\eta>0\) choose a fixed tightness radius \(R\); once
\(h_n>R+1\), every point with \(0<x_N<1\) has vertical distance greater than \(R\) from \((0,h_n)\), so the entire unit boundary strip lies outside \(B_R((0,h_n))\).

A local boundary Moser estimate, applied in unit half-balls centred on
\(\partial\Hh\), gives
\[
\sup_{B_{1/2}^+(x',0)}v_n
\le
C\left(\int_{B_1^+(x',0)}v_n^2\right)^{1/2}.
\]
The constant depends only on \(N\), \(\kappa\), and a uniform bound for
\(\lambda_n\); it is independent of the tangential centre and of
\(\tau_n\), because the Robin term has the favourable sign in the Moser energy inequality and may be discarded.  Taking the supremum in \(x'\) and bounding each half-ball integral by the total \(L^2\)-mass of the unit strip proves \eqref{eq:boundary-sup-escape}.
\end{proof}

\begin{theorem}[Attainment at the transition threshold]\label{thm:threshold-attainment}
At the critical parameter, $\Lambda_{\Hh}(\tau_*)=k^2$ and the half-space infimum is attained.  More precisely, let \(\tau_n\uparrow\tau_*\) and let \((A_n,v_n)\) be nonnegative optimisers with values below $k^2$, normalised by
\[
\mathcal M(v_n)=1,
\qquad
\lambda_n:=\Lambda_{\Hh}(\tau_n)<k^2.
\]
Then, after passing to a subsequence and translating parallel to
\(\partial\Hh\), there is a normalised threshold optimiser \((A_*,v_*)\) such that
\begin{equation}\label{eq:subcritical-to-threshold-H1}
 v_n\longrightarrow v_*
 \qquad\text{strongly in }H^1(\Hh),
\end{equation}
and
\begin{equation}\label{eq:subcritical-to-threshold-set}
 |A_n\triangle A_*|\longrightarrow0.
\end{equation}
\end{theorem}

\begin{proof}
Continuity of $\Lambda_{\Hh}$ and Theorem~\ref{thm:finite-detachment} give
$\lambda_n\to k^2$.  The normalisation and Lemma~\ref{lem:normalised-bounded} give a uniform $H^1$ bound, hence a uniform trace bound, and therefore
\begin{equation}\label{eq:tau-star-minseq-from-below}
E_{\tau_*}(v_n)=\lambda_n+(\tau_*-\tau_n)\|\operatorname{Tr}v_n\|_2^2\longrightarrow k^2.
\end{equation}
Thus $(v_n)$ is a threshold minimising sequence.  The vanishing and dichotomy alternatives are excluded exactly as in the proof of Theorem~\ref{thm:strict-gap-attainment}: the argument there uses only the unit-volume value and strict subhomogeneity in the favourable volume, which remain valid at $\tau_*$ with value $k^2$.  Hence, after tangential translations, $(v_n)$ is tight about centres $(0,h_n)$.

If $(h_n)$ is bounded, tightness and local Rellich compactness give strong $L^2$ convergence to $v_*$, the normalisation passes to the limit, and lower semicontinuity in \eqref{eq:tau-star-minseq-from-below} gives a threshold minimiser.  Equality of energies yields strong $H^1$ convergence.  Proposition~\ref{prop:EL-halfspace}, Lemma~\ref{lem:no-plateau-halfspace}, and Lemma~\ref{lem:bathtub-stability} then give $|A_n\triangle A_*|\to0$.

It remains to exclude $h_n\to\infty$.  Cut off $v_n$ below the escaping core and extend by zero to $W_n\in H^1(\R^N)$.  Tightness gives
\begin{equation}\label{eq:threshold-escape-cutoff}
\|v_n-W_n\|_{H^1(\Hh)}\to0,
\qquad \|\operatorname{Tr}v_n\|_2\to0,
\end{equation}
where the first expression is understood after restriction of $W_n$ to $\Hh$.  Consequently
$\mathcal M_1^{\R^N}(W_n)=1+o(1)$ and
$\int_{\R^N}|\nabla W_n|^2=k^2+o(1)$.  Proposition~\ref{prop:whole-space-compactness} therefore gives translations for which $W_n\to\Phi$ strongly in $H^1(\R^N)$.  The bathtub inclusions and strict radial monotonicity of $\Phi$ imply, exactly as in Lemma~\ref{lem:bathtub-level-convergence}, that the levels satisfy
\begin{equation}\label{eq:threshold-escape-level}
\theta_n\longrightarrow\Phi(R_1)>0.
\end{equation}
Lemma~\ref{lem:boundary-sup-escape} gives
$\|v_n\|_{L^\infty(\partial\Hh)}\to0$; together with Lemma~\ref{lem:contact-forces-boundary-amplitude} this forces a strip free of favourable material between $A_n$ and the boundary for all large $n$.

For $f_n=\operatorname{Tr}v_n$ and $s_n(\xi)=\sqrt{\lambda_n+|\xi|^2}$, the one-sided translation inequality of Lemma~\ref{lem:detached-translation} gives
\begin{equation}\label{eq:threshold-translation-upper}
\int s_n|\widehat f_n|^2\le\tau_n\|f_n\|_2^2.
\end{equation}
Repeating the lower-half-space $\lambda_n$-Yukawa gluing from Proposition~\ref{prop:large-tau-contact} gives
\begin{equation}\label{eq:threshold-gluing-lower}
\int s_n|\widehat f_n|^2-\tau_n\|f_n\|_2^2
\ge (k^2-\lambda_n)(1-L_n),
\qquad
L_n:=\frac12\int\frac{|\widehat f_n|^2}{s_n}.
\end{equation}
By \eqref{eq:threshold-escape-cutoff}, $L_n\to0$, whereas $k^2-\lambda_n>0$.  The last two inequalities are incompatible for large $n$.  Escape from the boundary is impossible, and the bounded-height alternative proves the theorem.
\end{proof}

\begin{proposition}[Threshold optimisers have essential boundary contact]\label{prop:threshold-contact}
Let $(A,v)$ be a nonnegative threshold optimiser, normalised by
\[
\int_{\Hh}m_Av^2=1,
\qquad
E_{\tau_*}(v)=k^2.
\]
Then its favourable set has essential boundary contact:
\begin{equation}\label{eq:threshold-contact}
|A\cap\{0<x_N<d\}|>0
\qquad\text{for every }d>0.
\end{equation}
\end{proposition}

\begin{proof}
Suppose instead that $A$ leaves a strip $0<x_N<d$ free of favourable material.  The one-sided normal-translation calculation from Lemma~\ref{lem:detached-translation} does not use the strict inequality $\lambda<k^2$ until later in that argument.  Repeating only its variational derivative at the present value $\lambda=k^2$ gives, for $f=\operatorname{Tr}v$,
\begin{equation}\label{eq:threshold-contact-rms}
\int_{\R^{N-1}}(k^2+|\xi|^2)|\widehat f(\xi)|^2\,d\xi
\le \tau_*^2\|f\|_2^2.
\end{equation}
On the other hand the threshold identity
\[
\int_{\Hh}|\nabla v|^2+\tau_*\|f\|_2^2=k^2
\]
shows that the whole-space defect from \eqref{eq:defect-functional} is exactly
\[
\mathfrak D(v)=\tau_*\|f\|_2^2.
\]
Proposition~\ref{prop:DtN-defect}, followed by Cauchy--Schwarz and
\eqref{eq:threshold-contact-rms}, therefore gives
\[
\tau_*\|f\|_2^2
\le
\int s(\xi)|\widehat f(\xi)|^2\,d\xi
\le
\|f\|_2
\left(\int s(\xi)^2|\widehat f(\xi)|^2\,d\xi\right)^{1/2}
\le
\tau_*\|f\|_2^2,
\]
where $s(\xi)=\sqrt{k^2+|\xi|^2}$.  Thus equality holds throughout.  The trace cannot vanish: otherwise its zero extension would attain the whole-space optimum while vanishing on the lower half-space, contradicting positivity of the whole-space principal optimiser.  Equality in Cauchy--Schwarz therefore forces $s(\xi)=\tau_*$ for almost every $\xi$ carrying $|\widehat f|^2$.  The level set $\{\xi:s(\xi)=\tau_*\}$ is either empty, a singleton, or an $(N-2)$-dimensional sphere, hence has zero $(N-1)$-dimensional Lebesgue measure. Therefore $f=0$, a contradiction.
\end{proof}

\begin{theorem}[Critical half-space compactness--escape dichotomy]\label{thm:critical-halfspace-profile}
Let \((v_n)\subset H^1(\Hh)\) be any normalised minimising sequence at the threshold,
\[
\mathcal M(v_n)=1,
\qquad
E_{\tau_*}(v_n)\longrightarrow k^2.
\]
After passing to a subsequence and translating parallel to \(\partial\Hh\), exactly one of the following profile alternatives occurs.

\smallskip
\noindent\emph{Compact critical profile.}
There is a threshold optimiser \(v_*\) such that
\[
v_n\longrightarrow v_*
\qquad\text{strongly in }H^1(\Hh).
\]
For bathtub maximisers \(A_n\), after a further subsequence their sets converge in measure to the unique bathtub set \(A_*\) of \(v_*\).

\smallskip
\noindent\emph{Escape from the boundary.}
There are heights \(h_n\to+\infty\), vertical cutoffs \(\zeta_n\), and full-space translations \(x_n\) such that, with \(W_n\) the zero extension of \(\zeta_nv_n\),
\[
\|v_n-\zeta_nv_n\|_{H^1(\Hh)}\to0,
\qquad
W_n(x_n+\cdot)\to\Phi
\quad\text{strongly in }H^1(\R^N),
\]
and the translated bathtub sets converge in measure to \(B_{R_1}\).

Both kinds of minimising sequence exist at \(\tau_*\): the first by Theorem~\ref{thm:threshold-attainment}, and the second by translating the whole-space optimiser arbitrarily far into the half-space.
\end{theorem}

\begin{proof}
Apply concentration--compactness to the normalised $H^1$-mass measures. Vanishing is excluded by $\mathcal M(v_n)=1$. For dichotomy, use the standard separated cutoff decomposition together with Lemma~\ref{lem:asymptotic-B-splitting}. Strict subhomogeneity in the favourable volume at $\tau_*$ gives a uniform cost strictly above $k^2$ when both favourable-volume fractions stay away from $0$ and $1$, while the small-volume estimate in that lemma excludes endpoint allocations: a component receiving vanishing favourable volume retains a strictly positive $H^1$ cost. Thus the sequence is tight about centres $(0,h_n)$ after tangential translations.

If $(h_n)$ is bounded, local Rellich compactness plus tightness yields strong $L^2$ convergence. Continuity of $\mathcal M$ preserves the normalisation, weak lower semicontinuity gives a threshold minimiser, and equality of the quadratic energies upgrades the convergence to strong $H^1$. Lemma~\ref{lem:bathtub-stability} then gives convergence of the bathtub sets.

If $h_n\to+\infty$, choose the vertical cutoffs used in \eqref{eq:threshold-escape-cutoff}. Tightness gives $\|v_n-\zeta_nv_n\|_{H^1}\to0$ and vanishing boundary trace. The zero extensions $W_n$ therefore satisfy $\mathcal M_1^{\R^N}(W_n)=1+o(1)$ and $\int|\nabla W_n|^2=k^2+o(1)$. After scalar normalisation, Proposition~\ref{prop:whole-space-compactness} gives translations for which $W_n\to\Phi$ strongly in $H^1(\R^N)$. Strong $L^2$ convergence of the squared densities and Proposition~\ref{prop:whole-space-classification}, followed by Lemma~\ref{lem:bathtub-stability}, give convergence of the translated bathtub sets to $B_{R_1}$.

For existence of an escaping sequence, translate the whole-space ball optimiser by \(he_N\), restrict it to \(\Hh\), and normalise its reduced mass.  As $h\to\infty$, the omitted lower-half-space mass and energy and the boundary trace all tend to zero by exponential decay of the whole-space profile, so the threshold quotient tends to $k^2$.
\end{proof}

\begin{proposition}[Left slope at the threshold and trace selection]\label{prop:tau-star-left-slope}
Let $\mathfrak O_*$ denote the class of nonnegative threshold optimisers normalised by
\[
\int_{\Hh}m_Av^2=1.
\]
Define
\begin{equation}\label{eq:T-star}
T_*:=\sup_{(A,v)\in\mathfrak O_*}
\int_{\partial\Hh}v^2\,dS.
\end{equation}
Then $0<T_*<\infty$, the supremum is attained, and
\begin{equation}\label{eq:tau-star-left-derivative}
\Lambda_{\Hh}'(\tau_*-) = T_*.
\end{equation}
Equivalently,
\begin{equation}\label{eq:tau-star-linear-left}
\Lambda_{\Hh}(\tau)
=
 k^2-T_*(\tau_*-\tau)+o(\tau_*-\tau)
\qquad\text{as }\tau\uparrow\tau_*.
\end{equation}
Moreover, if $\tau_n\uparrow\tau_*$ and $(A_n,v_n)$ are exact subcritical optimisers as in Theorem~\ref{thm:threshold-attainment}, then every compact threshold limit supplied by that theorem satisfies
\begin{equation}\label{eq:max-trace-threshold-limit}
\int_{\partial\Hh}v_*^2\,dS=T_*.
\end{equation}
Thus the subcritical branch selects maximal-trace threshold optimisers.
\end{proposition}

\begin{proof}
Every threshold optimiser has nonzero trace by the proof of Proposition~\ref{prop:threshold-contact}, so $T_*>0$.  Uniform $H^1$ bounds for normalised threshold minimisers and the trace theorem give $T_*<\infty$.

Let $(A_j,v_j)\in\mathfrak O_*$ be a sequence whose trace norms tend to $T_*$.  Since $T_*>0$, an escaping subsequence is impossible: in the escape alternative of Theorem~\ref{thm:critical-halfspace-profile} the trace norm tends to zero.  Hence, after tangential translations and a subsequence, $v_j$ converges strongly in $H^1(\Hh)$ to a threshold optimiser.  Continuity of the trace map then shows that the supremum in \eqref{eq:T-star} is attained.

Now let $h_n:=\tau_*-\tau_n\downarrow0$ and write
\[
\lambda_n:=\Lambda_{\Hh}(\tau_n),
\qquad
T_n:=\|\operatorname{Tr}v_n\|_2^2.
\]
For any $(A,v)\in\mathfrak O_*$,
\[
\lambda_n
\le E_{\tau_n}(v)
=k^2-h_n\|\operatorname{Tr}v\|_2^2,
\]
so
\begin{equation}\label{eq:left-slope-lower}
\frac{k^2-\lambda_n}{h_n}\ge T_*.
\end{equation}
Conversely, threshold optimality gives
\[
k^2\le E_{\tau_*}(v_n)=\lambda_n+h_nT_n,
\]
and therefore
\begin{equation}\label{eq:left-slope-upper}
\frac{k^2-\lambda_n}{h_n}\le T_n.
\end{equation}
By Theorem~\ref{thm:threshold-attainment}, after tangential recentering and subsequence extraction, $v_n\to v_*$ strongly in $H^1(\Hh)$ for some threshold optimiser.  Hence $T_n\to\|\operatorname{Tr}v_*\|_2^2\le T_*$.  Combining \eqref{eq:left-slope-lower} and \eqref{eq:left-slope-upper} forces both the quotient and $T_n$ to converge to $T_*$.  This proves \eqref{eq:max-trace-threshold-limit} and the left derivative formula.  The expansion \eqref{eq:tau-star-linear-left} follows.
\end{proof}

\begin{corollary}[Critical scaling near the transition threshold]\label{cor:critical-window-halfspace}
For every fixed $\sigma\in\R$,
\begin{equation}\label{eq:critical-window-halfspace}
\Lambda_{\Hh}(\tau_*+\eps\sigma)
=
 k^2+\eps\min\{0,\sigma T_*\}+o(\eps)
\qquad(\eps\downarrow0).
\end{equation}
In particular, changes of $\tau$ on the scale $O(\eps)$ compete with first-order geometric corrections in the bounded-domain problem.
\end{corollary}

\begin{proof}
If $\sigma\ge0$, then $\tau_*+\eps\sigma\ge\tau_*$ and Theorem~\ref{thm:finite-detachment} gives the exact value $k^2$.  If $\sigma<0$, apply \eqref{eq:tau-star-linear-left} with $\tau=\tau_*+\eps\sigma$.
\end{proof}

Theorem~\ref{thm:finite-detachment} resolves the phase transition: $0<\tau_*<\infty$ and $\Lambda_{\Hh}(\tau)=\Lambda_{\R^N}$ for every $\tau\ge\tau_*$.  Compact minimisers exist for $0\le\tau<\tau_*$; at the endpoint a compact optimiser coexists with whole-space minimising sequences that escape from the boundary.

\subsection{Compactness below the whole-space value and escape above the threshold}

Theorem~\ref{thm:strict-gap-attainment} gives compactness below the whole-space value.  Vanishing is excluded by the bathtub normalisation, exact and asymptotic splitting of the bathtub functional combine with strict subhomogeneity in the favourable volume to exclude dichotomy, and a remaining one-core sequence cannot escape from the boundary when the energy lies strictly below the whole-space value because Proposition~\ref{prop:tight-escape-lower} would force the whole-space cost.

Consequently, for every
\[
\Lambda_{\Hh}(\tau)<\Lambda_{\R^N},
\]
the half-space problem possesses a compact optimiser modulo tangential translations.  Loss of compactness can occur only at the whole-space value.  Theorem~\ref{thm:finite-detachment} shows that equality is reached at a finite $\tau_*$; Theorem~\ref{thm:threshold-attainment} proves that the endpoint is attained, and Theorem~\ref{thm:critical-halfspace-profile} classifies the endpoint loss of compactness.

\begin{theorem}[Minimising sequences above the threshold escape from the boundary]\label{thm:postthreshold-profile}
Fix \(\tau>\tau_*\), and let \((v_n)\subset H^1(\Hh)\) be a normalised minimising sequence for the reduced half-space problem,
\[
\mathcal M(v_n)=1,
\qquad
E_\tau(v_n):=\int_{\Hh}|\nabla v_n|^2+\tau\int_{\partial\Hh}v_n^2
\longrightarrow k^2=\Lambda_{\R^N}.
\]
Then, after translations parallel to \(\partial\Hh\), there are heights \(h_n\to+\infty\) such that the sequence is tight in \(H^1\) about \((0,h_n)\).  More precisely, one may choose cutoffs \(\zeta_n=\zeta_n(x_N)\) with
\[
0\le\zeta_n\le1,
\qquad
\zeta_n=0\ \text{near }\partial\Hh,
\qquad
\zeta_n=1\ \text{on a neighbourhood of }(0,h_n),
\]
so that
\begin{equation}\label{eq:postthreshold-cutoff-small}
\|v_n-\zeta_nv_n\|_{H^1(\Hh)}\longrightarrow0.
\end{equation}
If \(W_n\in H^1(\R^N)\) denotes the zero extension of \(\zeta_nv_n\), then, after a scalar normalisation and full-space translations \(x_n\in\R^N\),
\begin{equation}\label{eq:postthreshold-whole-profile}
W_n(x_n+\cdot)\longrightarrow\Phi
\qquad\text{strongly in }H^1(\R^N),
\end{equation}
where \(\Phi\) is the positive radial whole-space optimiser.  In particular every minimising sequence above the threshold loses compactness only through escape from the boundary.

If \(A_n\) are bathtub maximisers for \(v_n^2\), the translations may be chosen so that
\begin{equation}\label{eq:postthreshold-set-profile}
|(A_n-x_n)\triangle B_{R_1}|\longrightarrow0,
\end{equation}
where sets are extended by the empty set below the moving boundary.
\end{theorem}

\begin{proof}
Apply Lions' concentration--compactness principle to the normalised $H^1$-mass measures. Vanishing is excluded by $\mathcal M(v_n)=1$. If dichotomy occurs, the standard separated cutoff decomposition and Lemma~\ref{lem:asymptotic-B-splitting} assign limiting favourable-volume fractions to the two nontrivial components. Fractions in $(0,1)$ are excluded by Corollary~\ref{cor:strict-resource-subhomogeneity}; endpoint fractions are excluded by the uniform small-volume estimate in Lemma~\ref{lem:asymptotic-B-splitting}, because a component receiving vanishing favourable volume still has a strictly positive $H^1$ cost. Thus dichotomy is impossible, and the only remaining alternative is tightness about centres
\[
y_n=(y_n',h_n)\in\overline{\Hh}.
\]
Translate tangentially so that \(y_n'=0\).

The heights cannot remain bounded.  Indeed, if a subsequence had \(h_n\le C\), tightness and local Rellich compactness would give, as in the compactness part of Theorem~\ref{thm:strict-gap-attainment}, a strong \(L^2\) limit \(v\) with \(\mathcal M(v)=1\) and
\[
E_\tau(v)\le\liminf_nE_\tau(v_n)=k^2.
\]
Since \(\Lambda_{\Hh}(\tau)=k^2\), the limit would attain the half-space infimum, contradicting Corollary~\ref{cor:nonattainment-above-threshold}.  Therefore \(h_n\to+\infty\).

Choose a smooth one-dimensional cutoff \(\zeta_n\) such that
\[
\zeta_n=0\quad\text{on }0\le x_N\le h_n/4,
\qquad
\zeta_n=1\quad\text{for }x_N\ge h_n/2,
\qquad
|\zeta_n'|\le C h_n^{-1}.
\]
For every fixed \(R\), the strip \(\{x_N<h_n/2\}\) lies outside \(B_R((0,h_n))\) once \(n\) is large.  The \(H^1\)-tightness about \((0,h_n)\), followed by \(R\to\infty\), therefore gives
\[
\int_{\{x_N<h_n/2\}}(|\nabla v_n|^2+v_n^2)\to0.
\]
Together with \(|\zeta_n'|\le Ch_n^{-1}\), this proves \eqref{eq:postthreshold-cutoff-small}.  The same tightness, combined with a local trace inequality in the fixed strip \(0<x_N<1\), gives
\[
\|\operatorname{Tr}v_n\|_{L^2(\partial\Hh)}\longrightarrow0.
\]
Hence \(\int_{\Hh}|\nabla v_n|^2\to k^2\).  Since \(\zeta_nv_n\) has zero trace, its zero extension \(W_n\) belongs to \(H^1(\R^N)\).  The \(L^2\)-continuity of the bathtub functional and \eqref{eq:postthreshold-cutoff-small} now give
\[
\mathcal M_1^{\R^N}(W_n)=1+o(1),
\qquad
\int_{\R^N}|\nabla W_n|^2=k^2+o(1).
\]  Multiplying \(W_n\) by \(1+o(1)\) normalizes its whole-space reduced mass to one.  Proposition~\ref{prop:whole-space-compactness} then yields \eqref{eq:postthreshold-whole-profile}.

Finally, \eqref{eq:postthreshold-cutoff-small} implies convergence of the translated squared densities in \(L^1\) to \(\Phi^2\).  Since each \(A_n\) is a unit-volume bathtub maximiser for \(v_n^2\), while the radial optimiser \(\Phi\) has a unique unit-volume bathtub ball and no plateau at its threshold, Lemma~\ref{lem:bathtub-stability} gives \eqref{eq:postthreshold-set-profile}.
\end{proof}

\subsection{Tangential symmetry of threshold optimisers}\label{subsec:threshold-symmetry}

The half-space problem is invariant under translations and rotations parallel to the boundary.  Tangential Schwarz rearrangement therefore suggests rotational symmetry of compact threshold optimisers.  To obtain the equality case needed here, we combine moment control with the classical rigidity theorem for Schwarz rearrangement and then show that the centres of the horizontal slices cannot depend on height.  We use \citet{BrothersZiemer1988}; see also \citet{FeroneVolpicelli2003}.  The partial-rearrangement framework of \citet{Capriani2014} is related, but only the classical Schwarz result is needed below.

\begin{lemma}[Moment finiteness for threshold optimisers]\label{lem:threshold-exponential-moments}
Let $(A,v)\in\mathfrak O_*$.  Then $A$ is bounded.  Moreover there is $\gamma>0$ such that
\begin{equation}\label{eq:threshold-exp-moment}
\int_{\Hh}e^{2\gamma|y|}\bigl(v^2+|\nabla v|^2\bigr)\,dy<\infty.
\end{equation}
In particular all quantities in \eqref{eq:critical-response-T}--\eqref{eq:critical-response-M}, and all first and second moments used below, are finite.  In addition, for every compact interval $J\Subset(0,\infty)$,
\begin{equation}\label{eq:threshold-uniform-slice-moment}
\sup_{s\in J}
\int_{\R^{N-1}}e^{2\gamma|x'|}v(x',s)^2\,dx'<\infty.
\end{equation}
\end{lemma}

\begin{proof}
Let $t>0$ be the bathtub level of Lemma~\ref{lem:no-plateau-halfspace}, so that $A=\{v>t\}$ up to null sets.  Since $v\in H^1(\Hh)$ and solves an equation with uniformly bounded zeroth-order coefficient, the standard local interior and boundary Moser estimates give
\[
\sup_{B_{1/2}^+(y)}v
\le C\|v\|_{L^2(B_1^+(y))}
\]
with a constant independent of the centre.  Because the $L^2$ tail of the fixed function $v$ tends to zero, the right-hand side tends to zero uniformly as $|y|\to\infty$.  Thus $v<t$ outside a sufficiently large half-ball, and $A$ is bounded.  The same estimate will also be used below in the following elementary form:
\begin{equation}\label{eq:threshold-uniform-decay}
\sup_{\substack{(x',s)\in\Hh\\ |(x',s)|\ge R}}v(x',s)\longrightarrow0
\qquad(R\to\infty).
\end{equation}

Outside a fixed half-ball containing $A$, $v$ solves the constant-coefficient equation $(-\Delta+k^2)v=0$
\[
(-\Delta+k^2)v=0,
\]
with the nonnegative Robin term $\tau_*\int_{\partial\Hh}v\varphi$ in the weak formulation.  A standard Agmon test with
$\varphi=\chi_R^2e^{2\gamma|y|}v$, where $0<\gamma<k$ and $\chi_R$ vanishes on the bounded favourable region and is cut off again at large radius, gives a bound independent of the outer cutoff for the corresponding weighted $H^1$ norm.  The Robin boundary contribution has the favourable sign and may be discarded.  Letting the outer cutoff tend to infinity gives \eqref{eq:threshold-exp-moment}.

Finally fix $J\Subset(0,\infty)$.  Since $e^{\gamma|x'|}\le e^{\gamma|(x',s)|}$, \eqref{eq:threshold-exp-moment} implies
\[
e^{\gamma|x'|}v\in H^1\bigl(J;L^2(\R^{N-1})\bigr).
\]
The Hilbert-valued one-dimensional Sobolev embedding
$H^1(J;L^2)\hookrightarrow C(\overline J;L^2)$ therefore gives \eqref{eq:threshold-uniform-slice-moment}.
\end{proof}

For a normalised threshold optimiser $(A,v)\in\mathfrak O_*$ define
\begin{align}
T(v)&:=\int_{\partial\Hh}v^2\,dS,\label{eq:critical-response-T}\\
a(A)&:=\int_A y_N\,dy,\label{eq:critical-response-a}\\
\mathsf M_T(v)&:=\int_{\Hh}y_N\,\nabla_Tv\otimes\nabla_Tv\,dy,
\label{eq:critical-response-M}
\end{align}
where $\nabla_T=(\partial_1,\ldots,\partial_{N-1})$.  These quantities are finite by Lemma~\ref{lem:threshold-exponential-moments}.

For a nonnegative $v\in H^1(\Hh)$ and $s>0$, write
\[
v_s(x'):=v(x',s),\qquad x'\in\R^{N-1},
\]
and let $v^\sharp_s$ be the Schwarz symmetric decreasing rearrangement of $v_s$ in $\R^{N-1}$.  The function $v^\sharp(x',s):=v_s^\sharp(x')$ will be called the tangential Schwarz rearrangement of $v$.

\begin{lemma}[Tangential Schwarz rearrangement]\label{lem:tangential-schwarz}
For every nonnegative $v\in H^1(\Hh)$,
\begin{align}
\|v^\sharp\|_{L^2(\Hh)}&=\|v\|_{L^2(\Hh)},
&\mathcal B(v^\sharp)&=\mathcal B(v),
\label{eq:tangential-rearr-mass}\\[1mm]
\|\operatorname{Tr}v^\sharp\|_{L^2(\partial\Hh)}
&=\|\operatorname{Tr}v\|_{L^2(\partial\Hh)},
\label{eq:tangential-rearr-trace}\\[1mm]
\int_{\Hh}|\nabla_Tv^\sharp|^2
&\le\int_{\Hh}|\nabla_Tv|^2,
&
\int_{\Hh}|\partial_Nv^\sharp|^2
&\le\int_{\Hh}|\partial_Nv|^2.
\label{eq:tangential-rearr-energy}
\end{align}
Consequently
\[
\mathcal M(v^\sharp)=\mathcal M(v),
\qquad
E_\tau(v^\sharp)\le E_\tau(v)
\quad\text{for every }\tau\ge0.
\]
\end{lemma}

\begin{proof}
For almost every $s$, Schwarz rearrangement preserves every $L^p$ norm of the slice and does not increase its tangential Dirichlet integral.  Integrating in $s$ gives the first identity in \eqref{eq:tangential-rearr-mass} and the first inequality in \eqref{eq:tangential-rearr-energy}.  Slice-wise equimeasurability also gives, for every $t>0$,
\[
|\{v^\sharp>t\}|
=\int_0^\infty |\{v_s^\sharp>t\}|\,ds
=\int_0^\infty |\{v_s>t\}|\,ds
=|\{v>t\}|.
\]
Thus $v$ and $v^\sharp$ are globally equimeasurable and the bathtub functional is unchanged.

For the normal derivative we use the $L^2$ contraction property of Schwarz rearrangement.  For $h>0$,
\[
\|v^\sharp_{s+h}-v^\sharp_s\|_{L^2(\R^{N-1})}
\le
\|v_{s+h}-v_s\|_{L^2(\R^{N-1})}.
\]
After division by $h$, integration in $s$, and passage to the limit through the difference-quotient characterisation of
$H^1((0,\infty);L^2(\R^{N-1}))$, this yields the second inequality in \eqref{eq:tangential-rearr-energy}.  The same contraction identifies the trace.  Indeed
$v_s\to\operatorname{Tr}v$ in $L^2(\R^{N-1})$ as $s\downarrow0$, so
\[
v_s^\sharp\longrightarrow (\operatorname{Tr}v)^\sharp
\quad\text{in }L^2(\R^{N-1}),
\]
and hence
$\operatorname{Tr}v^\sharp=(\operatorname{Tr}v)^\sharp$.
This proves \eqref{eq:tangential-rearr-trace}.
\end{proof}

We use the following equality case for Schwarz rearrangement.

\begin{lemma}[Schwarz equality rigidity for compactly supported functions]\label{lem:schwarz-equality-rigidity}
Let $m\ge1$, let $f\in H^1(\R^m)$ be nonnegative with compact support, and let $f^\sharp$ denote its Schwarz symmetric decreasing rearrangement.  Put
$M_f:=\operatorname*{ess\,sup}f^\sharp$.  Assume
\begin{equation}\label{eq:BZ-critical-set-condition}
\bigl|\{x\in\R^m:\nabla f^\sharp(x)=0,\ 0<f^\sharp(x)<M_f\}\bigr|=0.
\end{equation}
If
\begin{equation}\label{eq:BZ-energy-equality}
\int_{\R^m}|\nabla f|^2
=
\int_{\R^m}|\nabla f^\sharp|^2,
\end{equation}
then there exists $c\in\R^m$ such that
\[
f(x)=f^\sharp(x-c)
\qquad\text{for a.e. }x\in\R^m.
\]
\end{lemma}

\begin{proof}
This is the $p=2$ equality case in the Schwarz P\'olya--Szeg\H{o} theorem of \citet{BrothersZiemer1988}; see also \citet{FeroneVolpicelli2003}.  Condition \eqref{eq:BZ-critical-set-condition} is precisely the no-flat-zone hypothesis used there.  We have stated the result only for compactly supported functions because this is the form to which we reduce each horizontal slice by positive truncation.
\end{proof}

The next lemma shows that if every horizontal slice is a translate of a centred radial slice, equality in the normal Dirichlet energy forces the translation to be independent of height.

\begin{lemma}[Rigidity of a moving tangential centre]\label{lem:moving-centre-rigidity}
Let $m\ge1$.  Let $v,w\in H^1(\R^m\times(0,\infty))$ be nonnegative and suppose that, for almost every $s>0$, the slice $w_s$ is radial about the origin and nonzero.  Assume that there is a measurable map $c:(0,\infty)\to\R^m$ such that
\begin{equation}\label{eq:moving-centre-representation}
v(x,s)=w(x-c(s),s)
\quad\text{for a.e. }(x,s).
\end{equation}
Assume moreover that the continuous representative of
\[
\rho(s):=\|v_s\|_{L^2(\R^m)}^2
\]
is strictly positive on $(0,\infty)$ and that, for every compact interval $J\Subset(0,\infty)$,
\begin{equation}\label{eq:moving-centre-second-moment}
\sup_{s\in J}\int_{\R^m}|x|^2v(x,s)^2\,dx<\infty.
\end{equation}
Then $c$ has a locally $H^1$ representative.  For every compact $J\Subset(0,\infty)$ one has
\begin{equation}\label{eq:moving-centre-energy-identity}
\int_J\!\|\partial_sv_s\|_2^2\,ds
=
\int_J\!\|\partial_sw_s\|_2^2\,ds
+
\int_J a(s)|c'(s)|^2\,ds,
\end{equation}
where
\begin{equation}\label{eq:moving-centre-a}
a(s):=\frac1m\int_{\R^m}|\nabla_xw(x,s)|^2\,dx>0
\quad\text{for a.e. }s.
\end{equation}
Consequently, if
\begin{equation}\label{eq:moving-centre-normal-equality}
\int_0^\infty\|\partial_sv_s\|_2^2\,ds
=
\int_0^\infty\|\partial_sw_s\|_2^2\,ds,
\end{equation}
then $c$ is almost everywhere equal to one constant vector.
\end{lemma}

\begin{proof}
We divide the proof into three steps.

\emph{Step 1: the centre is locally $H^1$.}
Set
\[
b(s):=\int_{\R^m}x\,v(x,s)^2\,dx.
\]
Since $w_s$ is centred and radial, \eqref{eq:moving-centre-representation} gives
\begin{equation}\label{eq:barycentre-centre}
b(s)=\rho(s)c(s)
\quad\text{for a.e. }s.
\end{equation}
Fix $J\Subset(0,\infty)$.  The Hilbert-valued Sobolev property
$v\in H^1(J;L^2(\R^m))$ implies that $s\mapsto v_s$ is continuous into $L^2$ and
\[
\rho'(s)=2\int_{\R^m}v\,\partial_sv
\quad\text{for a.e. }s.
\]
Thus $\rho\in H^1(J)$.  By the strict positivity hypothesis and compactness of $J$,
\begin{equation}\label{eq:rho-away-zero}
\inf_{s\in J}\rho(s)>0.
\end{equation}

For the barycentre, first insert a smooth radial cutoff in $x$, differentiate under the integral, and then let the cutoff radius tend to infinity.  Assumption \eqref{eq:moving-centre-second-moment} and Cauchy--Schwarz justify the limit and give, componentwise,
\begin{equation}\label{eq:barycentre-derivative}
b'(s)=2\int_{\R^m}x\,v(x,s)\,\partial_sv(x,s)\,dx,
\end{equation}
with
\[
|b'(s)|
\le
2\left(\int |x|^2v_s^2\right)^{1/2}
\|\partial_sv_s\|_2.
\]
Hence $b\in H^1(J;\R^m)$.  Equations \eqref{eq:barycentre-centre} and \eqref{eq:rho-away-zero} imply that
$c=b/\rho\in H^1(J;\R^m)$.

\emph{Step 2: Sobolev chain rule for the moving translation.}
We claim that on $\R^m\times J$,
\begin{equation}\label{eq:moving-translation-chain-rule}
\partial_sv(x,s)
=
\partial_sw(x-c(s),s)
-c'(s)\cdot\nabla_xw(x-c(s),s)
\end{equation}
in the sense of distributions, and hence almost everywhere after identifying the $L^2$ representatives.

To justify the formula, choose $c_j\in C^\infty(\overline J;\R^m)$ with
$c_j\to c$ in $H^1(J)$ and uniformly on $J$, and choose smooth compactly supported $w_j$ converging to $w$ in $H^1(\R^m\times J)$.  For
$v_j(x,s):=w_j(x-c_j(s),s)$ the classical chain rule gives
\[
\partial_sv_j
=(\partial_sw_j)(x-c_j(s),s)
-c_j'(s)\cdot(\nabla_xw_j)(x-c_j(s),s).
\]
We use the following elementary variable-translation continuity fact.  If
$F\in L^2(\R^m\times J)$ and $d_j\in L^\infty(J;\R^m)$ with
$\|d_j\|_{L^\infty(J)}\to0$, then
\begin{equation}\label{eq:variable-translation-continuity}
\|F(\,\cdot-d_j(s),s)-F(\cdot,s)\|_{L^2(\R^m\times J)}\longrightarrow0.
\end{equation}
For $F\in C_c^\infty$ this follows from the fundamental theorem of calculus in the $x$ variables and the bound by $\|d_j\|_\infty\|\nabla_xF\|_2$; the general case follows by density and the fact that every translation is an $L^2$ isometry.

Now split
\[
v_j-v
=
\bigl[w_j(\,\cdot-c_j(s),s)-w(\,\cdot-c_j(s),s)\bigr]
+
\bigl[w(\,\cdot-c_j(s),s)-w(\,\cdot-c(s),s)\bigr].
\]
The first bracket tends to zero in $L^2$ because translations are isometries, while the second tends to zero by \eqref{eq:variable-translation-continuity} with $d_j=c_j-c$ after translating by $c(s)$.  Thus $v_j\to v$ in $L^2$.  Applying the same argument to $\partial_sw_j\to\partial_sw$ and to $\nabla_xw_j\to\nabla_xw$ gives
\[
(\partial_sw_j)(\,\cdot-c_j(s),s)\to
(\partial_sw)(\,\cdot-c(s),s),
\qquad
(\nabla_xw_j)(\,\cdot-c_j(s),s)\to
(\nabla_xw)(\,\cdot-c(s),s)
\]
in $L^2(\R^m\times J)$.  To pass the product term to the limit, fix a bounded set $K\Subset\R^m$.  Since $c_j$ and $c$ are uniformly bounded on $J$, all translations that meet $K$ remain in one fixed bounded enlargement $K_1$.  Cauchy--Schwarz gives
\[
\begin{aligned}
&\|c_j'\cdot\nabla_xw_j(\,\cdot-c_j,\cdot)
-c'\cdot\nabla_xw(\,\cdot-c,\cdot)\|_{L^1(K\times J)}
\\
&\quad\le |K|^{1/2}
\|c_j'-c'\|_{L^2(J)}
\|\nabla_xw_j(\,\cdot-c_j,\cdot)\|_{L^2(K\times J)}
\\
&\qquad
+|K|^{1/2}\|c'\|_{L^2(J)}
\|\nabla_xw_j(\,\cdot-c_j,\cdot)
-\nabla_xw(\,\cdot-c,\cdot)\|_{L^2(K\times J)}
\longrightarrow0.
\end{aligned}
\]
Thus the product converges in $L^1_{\rm loc}$, which is enough to pass to the limit against test functions.  Passing to the limit in the distributional identity gives \eqref{eq:moving-translation-chain-rule}.  Since the first two terms in that identity belong to $L^2$, it follows in particular that
$c'(s)\cdot\nabla_xw_s\in L^2(\R^m\times J)$.

\emph{Step 3: orthogonal energy splitting.}
Because $w$ is invariant under every rotation of the $x$ variables, the same invariance holds for its weak derivative $\partial_sw$; hence $\partial_sw_s$ is radial for almost every $s$.  Each component of $\nabla_xw_s$ is odd under the corresponding reflection.  Hence
\begin{equation}\label{eq:radial-cross-zero}
\int_{\R^m}\partial_sw_s\,\nabla_xw_s\,dx=0.
\end{equation}
Rotational invariance similarly gives
\begin{equation}\label{eq:radial-gradient-matrix}
\int_{\R^m}\nabla_xw_s\otimes\nabla_xw_s\,dx
=a(s)I_m,
\qquad
a(s)=\frac1m\int_{\R^m}|\nabla_xw_s|^2\,dx.
\end{equation}
Using \eqref{eq:moving-translation-chain-rule}, translating the $x$ variable, and then using \eqref{eq:radial-cross-zero}--\eqref{eq:radial-gradient-matrix}, we obtain for almost every $s\in J$
\[
\|\partial_sv_s\|_2^2
=
\|\partial_sw_s\|_2^2+a(s)|c'(s)|^2.
\]
Integration gives \eqref{eq:moving-centre-energy-identity}.  If $a(s)=0$, then $w_s$ is almost everywhere constant; since $w_s\in L^2(\R^m)$, that constant is zero, contrary to the nonzero-slice assumption.  Thus \eqref{eq:moving-centre-a} holds.

Finally, apply \eqref{eq:moving-centre-energy-identity} to an increasing exhaustion of $(0,\infty)$ by compact intervals.  The two normal energies are integrable, so \eqref{eq:moving-centre-normal-equality} yields
\[
\int_0^\infty a(s)|c'(s)|^2\,ds=0.
\]
Since $a(s)>0$ almost everywhere, $c'(s)=0$ almost everywhere.  The locally absolutely continuous representative of $c$ is therefore constant.
\end{proof}

\begin{theorem}[Tangential rotational symmetry at the threshold]\label{thm:threshold-tangential-symmetry}
Let $(A,v)\in\mathfrak O_*$ be a normalised threshold optimiser.  Then there exists one vector $c\in\R^{N-1}$ and a function $V:[0,\infty)\times(0,\infty)\to(0,\infty)$ such that, after modification on a null set,
\begin{equation}\label{eq:threshold-common-axis}
v(x',s)=V(|x'-c|,s),
\qquad
V(\cdot,s)\ \text{is nonincreasing for a.e. }s>0.
\end{equation}
If $t>0$ is the bathtub level, then
\[
A=\{v>t\}
\]
up to null sets and almost every horizontal section of $A$ is a ball centred at the same point $c$.

In particular the tangential energy-moment tensor is isotropic:
\begin{equation}\label{eq:MT-isotropic}
\mathsf M_T(v)=\mu_T(v)I_{N-1},
\qquad
\mu_T(v):=\frac1{N-1}
\int_{\Hh}y_N|\nabla_Tv|^2\,dy.
\end{equation}
\end{theorem}

\begin{proof}
Put $m:=N-1$ and set $w:=v^\sharp$.

\emph{Step 1: rearrangement is an equality at a threshold optimiser.}
Lemma~\ref{lem:tangential-schwarz} preserves the reduced mass, so
$\mathcal M(w)=\mathcal M(v)=1$, and
\[
E_{\tau_*}(w)\le E_{\tau_*}(v)=k^2.
\]
Since $k^2=\Lambda_{\Hh}(\tau_*)$, threshold optimality forces equality.  The $L^2$ norm and trace norm are already preserved by Lemma~\ref{lem:tangential-schwarz}; therefore the only possible deficits are the two nonnegative Dirichlet deficits in \eqref{eq:tangential-rearr-energy}.  Their sum is zero, so each vanishes:
\begin{equation}\label{eq:tangential-normal-energy-equality}
\int_{\Hh}|\nabla_Tw|^2=\int_{\Hh}|\nabla_Tv|^2,
\qquad
\int_{\Hh}|\partial_Nw|^2=\int_{\Hh}|\partial_Nv|^2.
\end{equation}
For almost every $s$, the slice-wise P\'olya--Szeg\H{o} deficit
\[
d(s):=
\int_{\R^m}|\nabla_Tv_s|^2
-
\int_{\R^m}|\nabla_Tw_s|^2
\]
is nonnegative.  The first equality in \eqref{eq:tangential-normal-energy-equality} says $\int_0^\infty d(s)\,ds=0$.  Hence
\begin{equation}\label{eq:slice-PS-equality}
\int_{\R^m}|\nabla_Tw_s|^2
=
\int_{\R^m}|\nabla_Tv_s|^2
\quad\text{for a.e. }s>0.
\end{equation}

The rearranged function is itself an attained threshold minimiser of the reduced problem.  Let $A^\sharp$ be a unit-volume bathtub set for $w$.  Since $\mathcal M(w)=1$ and $E_{\tau_*}(w)=k^2$, the pair $(A^\sharp,w)$ attains the original half-space problem.  Proposition~\ref{prop:EL-halfspace} therefore applies to $w$.  In particular $w$ is positive and solves the Euler--Lagrange equation with the bang--bang coefficient $m_{A^\sharp}$.  Lemma~\ref{lem:no-plateau-halfspace} gives a unique bathtub level $t>0$ up to null sets, with
\begin{equation}\label{eq:w-bathtub-set}
A^\sharp=\{w>t\}
\quad\text{a.e.},
\qquad
|\{w=t\}|=0.
\end{equation}
Lemma~\ref{lem:threshold-exponential-moments}, applied to $(A^\sharp,w)$, also gives boundedness of $A^\sharp$ and the decay \eqref{eq:threshold-uniform-decay} for $w$.

\emph{Step 2: the rearranged optimiser has no tangential flat zone below the slice maximum.}
For almost every $s$ define
\[
M(s):=\operatorname*{ess\,sup}_{x'\in\R^m}w(x',s).
\]
We claim
\begin{equation}\label{eq:no-flat-tangential-critical-set}
\left|
\{(x',s):\nabla_Tw(x',s)=0,\ 0<w(x',s)<M(s)\}
\right|=0.
\end{equation}

Indeed, away from the negligible interface $\{w=t\}$ there are two open phases
\[
U_+:=\{w>t\},
\qquad
U_-:=\{0<w<t\}.
\]
On each phase the coefficient in the Euler--Lagrange equation is constant.  Standard elliptic regularity therefore makes $w$ real analytic on every connected component of $U_+$ and on $U_-$.

We first consider $U_+$.  Suppose that on some connected component $C\subset U_+$ the common zero set of the tangential derivatives has positive $N$-dimensional measure.  Each analytic function $\partial_iw$ then vanishes on a set of positive measure in $C$ and consequently vanishes identically on $C$.  Thus $w$ is independent of $x'$ on $C$.  Choose $(x'_0,s_0)\in C$.  Because $w_{s_0}$ is radial and nonincreasing, the horizontal section
$\{x':w(x',s_0)>t\}$ is a centred ball $B_{R(s_0)}$.  It is bounded because $A^\sharp$ is bounded, and the whole ball belongs to the same connected component $C$.  Hence $w(\cdot,s_0)$ would be constant on $B_{R(s_0)}$.  Continuity at the lateral boundary, where $w=t$, would force this constant to equal $t$, contradicting $w>t$ inside.  Therefore the common tangential critical set has measure zero in $U_+$.

For $U_-$ we first note that it is connected.  Let
$K:=\{w\ge t\}$.  By continuity, tangential radial monotonicity, and the uniform decay of $w$, $K$ is bounded and each horizontal section of $K$ is a centred closed ball, possibly empty.  Given two points of $U_-$, choose $R$ larger than all sectional radii of $K$ and also larger than the tangential radii of the two chosen points, and choose a height $S$ above the vertical projection of $K$.  First move each point radially outwards within its own horizontal slice to some point of tangential radius $R$.  Radial monotonicity keeps these two segments below $t$.  From the first exterior point move vertically to height $S$; at height $S$ move along an arbitrary horizontal path to the tangential position of the second exterior point; then move vertically down to the target height and finally radially inwards to the second point.  The two vertical exterior segments and the entire horizontal segment at height $S$ lie outside $K$.  This construction works also for $m=1$, where the two points of tangential radius $R$ may have opposite signs.  Hence the whole path lies in $U_-$ and $U_-$ is connected.

If the common zero set of $\partial_1w,\ldots,\partial_mw$ had positive measure in $U_-$, analyticity and connectedness would imply
$\partial_iw\equiv0$ on $U_-$ for every $i$.  Choose $s_0$ above the vertical projection of the bounded set $K$.  Then
$U_-\cap\{s=s_0\}=\R^m\times\{s_0\}$, so $w(\cdot,s_0)$ would be constant on $\R^m$.  Since $w_{s_0}\in L^2(\R^m)$ the constant must be zero, contradicting positivity of the principal eigenfunction.  This proves \eqref{eq:no-flat-tangential-critical-set}.

By Fubini, there is a full-measure set $S\subset(0,\infty)$ such that for every $s\in S$ both \eqref{eq:slice-PS-equality} holds and
\begin{equation}\label{eq:slice-no-flat}
\bigl|\{x'\in\R^m:\nabla_Tw_s(x')=0,\ 0<w_s(x')<M(s)\}\bigr|=0.
\end{equation}

\emph{Step 3: each horizontal slice is a translate of its Schwarz rearrangement.}
Fix $s\in S$.  The classical equality theorem in Lemma~\ref{lem:schwarz-equality-rigidity} is stated for compactly supported functions, while the positive slice $v_s$ need not have compact support.  We therefore apply it to positive truncations.

For $0<\eta<M(s)$ set
\[
f_\eta:=(v_s-\eta)_+,
\qquad
f_\eta^\sharp=(w_s-\eta)_+,
\]
and also
\[
g_\eta:=\min\{v_s,\eta\},
\qquad
g_\eta^\sharp=\min\{w_s,\eta\}.
\]
The displayed rearrangement identities follow from equimeasurability.  The chain rule for Sobolev truncations and the fact that the weak gradient vanishes almost everywhere on a level set give the orthogonal decompositions
\begin{align*}
\int|\nabla_Tv_s|^2
&=
\int|\nabla f_\eta|^2+
\int|\nabla g_\eta|^2,\\
\int|\nabla_Tw_s|^2
&=
\int|\nabla f_\eta^\sharp|^2+
\int|\nabla g_\eta^\sharp|^2.
\end{align*}
P\'olya--Szeg\H{o} applied separately to $f_\eta$ and $g_\eta$ gives
\[
\int|\nabla f_\eta^\sharp|^2\le\int|\nabla f_\eta|^2,
\qquad
\int|\nabla g_\eta^\sharp|^2\le\int|\nabla g_\eta|^2.
\]
Their deficits are nonnegative and sum to the zero deficit in \eqref{eq:slice-PS-equality}; therefore
\begin{equation}\label{eq:truncated-PS-equality}
\int_{\R^m}|\nabla f_\eta^\sharp|^2
=
\int_{\R^m}|\nabla f_\eta|^2
\qquad(0<\eta<M(s)).
\end{equation}

The uniform decay \eqref{eq:threshold-uniform-decay}, now applied to the original threshold optimiser $v$, implies that $\{v_s>\eta\}$ is bounded.  Thus $f_\eta$ has compact support.  Moreover, by \eqref{eq:slice-no-flat},
\[
\bigl|
\{\nabla f_\eta^\sharp=0,\ 0<f_\eta^\sharp<M(s)-\eta\}
\bigr|=0.
\]
Lemma~\ref{lem:schwarz-equality-rigidity} and \eqref{eq:truncated-PS-equality} therefore give a vector $c_\eta(s)\in\R^m$ such that
\begin{equation}\label{eq:truncated-slice-translation}
(v_s-\eta)_+(x')
=(w_s-\eta)_+(x'-c_\eta(s))
\quad\text{for a.e. }x'.
\end{equation}

The centre does not depend on the truncation level.  Indeed, let
$0<\eta_1<\eta_2<M(s)$ and choose any
$q\in(\eta_2,M(s))$.  From \eqref{eq:truncated-slice-translation}, the nonempty set
$\{v_s>q\}$ is simultaneously a Euclidean ball centred at $c_{\eta_1}(s)$ and a Euclidean ball centred at $c_{\eta_2}(s)$.  A nonempty ball has a unique centre, hence
$c_{\eta_1}(s)=c_{\eta_2}(s)$.  Denote the common value by $c(s)$.

Letting $\eta\downarrow0$ through a countable sequence in \eqref{eq:truncated-slice-translation} and using positivity of $v_s$ gives
\begin{equation}\label{eq:slice-translated-radial}
v(x',s)=w(x'-c(s),s)
\quad\text{for a.e. }x'.
\end{equation}
Thus \eqref{eq:slice-translated-radial} holds for almost every $s>0$.

\emph{Step 4: the slice centre is independent of height.}
We verify the hypotheses of Lemma~\ref{lem:moving-centre-rigidity}.  For every compact $J\Subset(0,\infty)$, \eqref{eq:threshold-uniform-slice-moment} implies
\[
\sup_{s\in J}\int_{\R^m}|x'|^2v(x',s)^2\,dx'<\infty.
\]
The slices of $w$ are radial by construction.  Moreover the principal eigenfunction $v$ is continuous and strictly positive in $\Hh$.  Hence $\rho(s)=\|v_s\|_2^2>0$ for every $s>0$; since $v\in H^1(J;L^2)$, $\rho$ is continuous and is therefore bounded away from zero on every compact $J\Subset(0,\infty)$.  Thus all hypotheses of Lemma~\ref{lem:moving-centre-rigidity} apply to \eqref{eq:slice-translated-radial}.  Its normal-energy hypothesis is exactly the second equality in \eqref{eq:tangential-normal-energy-equality}.  Therefore $c(s)$ is almost everywhere equal to one fixed vector $c\in\R^m$.  This proves \eqref{eq:threshold-common-axis}.

The bathtub statement follows from $A=\{v>t\}$ up to null sets and \eqref{eq:threshold-common-axis}.  Finally, for almost every $s$, rotational symmetry around $c$ gives
\[
\int_{\R^m}\nabla_Tv_s\otimes\nabla_Tv_s\,dx'
=
\frac{I_m}{m}
\int_{\R^m}|\nabla_Tv_s|^2\,dx'.
\]
Multiplying by $s$ and integrating over $s>0$ gives \eqref{eq:MT-isotropic}.
\end{proof}

\begin{corollary}[Mean-curvature dependence at the threshold]\label{cor:scalar-critical-response}
For every threshold optimiser $(A,v)\in\mathfrak O_*$ define
\begin{equation}\label{eq:C-critical-def}
\mathcal C(A,v)
:=
2\mu_T(v)-\frac12T(v)
+\frac{a(A)}N\bigl(\tau_*T(v)-2k^2\bigr).
\end{equation}
Then the response in \eqref{eq:G-critical-def} satisfies
\begin{equation}\label{eq:G-critical-scalar}
\mathcal G_\sigma(P;A,v)
=\sigma T(v)+h_P\mathcal C(A,v).
\end{equation}
In particular the first-order domain dependence is purely through mean curvature; orientation relative to principal directions plays no role.
\end{corollary}

\begin{proof}
By Theorem~\ref{thm:threshold-tangential-symmetry},
$\mathsf M_T(v)=\mu_T(v)I_{N-1}$.  Hence
$\mathsf S_P:\mathsf M_T(v)=\mu_T(v)\operatorname{tr}\mathsf S_P=\mu_T(v)h_P$.
Substitution into \eqref{eq:G-critical-def} gives \eqref{eq:G-critical-scalar}.
\end{proof}

\section{Bounded-domain critical-scale limit}\label{sec:bounded-targets}

\subsection{Exact leading-order asymptotics}

Let $\eps:=\delta^{1/N}$ and $\tau_\delta:=\alpha_\delta\eps$.  We transfer the half-space result to the original bounded domain.  At every finite limit of the scaled Robin parameter, the half-space problem gives the leading coefficient of the bounded-domain optimum.  The proof uses the standard localisation and boundary-flattening scheme from small-volume spectral optimisation; see, for example, \citet{FerreriVerzini2024,FerreriMazzoleniPellacciVerzini2026}.  The estimates below address the optimised denominator and the scaled Robin boundary term.  Choose a mesoscopic scale $\eps\ll r_\delta\ll1$ so that each localised piece sees either a nearly flat half-space or the whole space, while the artificial localisation error is negligible compared with the natural energy scale \(\eps^{-2}\).

The following small-set estimate makes the localisation uniform.  For a measurable favourable set \(D\subset\Omega\), write
\[
d_D(u):=\int_\Omega m_Du^2.
\]

\begin{lemma}[Small-set $L^2$ estimate under positive weighted mass]\label{lem:small-set-L2}
Assume that \(\Omega\) is bounded and Lipschitz and fix \(\kappa>0\).  There exist \(\delta_0>0\) and \(C>0\), depending only on \(\Omega,N,\kappa\), such that for every measurable \(D\subset\Omega\) with \(|D|=\delta<\delta_0\) and every \(u\in H^1(\Omega)\) satisfying \(d_D(u)>0\),
\begin{equation}\label{eq:small-set-L2}
\|u\|_{L^2(\Omega)}^2
\le C\rho_N(\delta)\|\nabla u\|_{L^2(\Omega)}^2,
\end{equation}
where
\begin{equation}\label{eq:rhoN}
\rho_N(\delta):=
\begin{cases}
\delta^{2/N},&N\ge3,\\[1mm]
\delta\bigl(1+|\log\delta|\bigr),&N=2.
\end{cases}
\end{equation}
In particular, since the Robin term is nonnegative,
\[
\|u\|_2^2\le C\rho_N(\delta)\,\calE_{\alpha_\delta}(u)
\qquad(\alpha_\delta\ge0).
\]
\end{lemma}

\begin{proof}
Since $d_D(u)>0$,
$\|u\|_2^2<(\kappa+1)\int_Du^2$.  For $N\ge3$, H\"older and the fixed-domain Sobolev embedding give
$\int_Du^2\le C\delta^{2/N}(\|\nabla u\|_2^2+\|u\|_2^2)$; for $N=2$, the standard estimate
$\|u\|_{L^p}^2\le Cp(\|\nabla u\|_2^2+\|u\|_2^2)$ with $p=|\log\delta|$ gives
$\int_Du^2\le C\delta(1+|\log\delta|)(\|\nabla u\|_2^2+\|u\|_2^2)$.  For $\delta$ small the $L^2$ term is absorbed, proving \eqref{eq:small-set-L2}.
\end{proof}

We also use monotonicity with respect to the local favourable volume.

\begin{lemma}[Monotonicity with respect to local favourable volume]\label{lem:local-volume-monotonicity}
For every \(\alpha\ge0\), the maps
\[
V\longmapsto\Lambda_{\Hh}(\alpha;V),
\qquad
V\longmapsto\Lambda_{\R^N}(V)
\]
are strictly decreasing on \((0,\infty)\).  Consequently, for every bounded interval \(0\le \alpha\eps\le T\) there is a modulus \(\omega_T(r)\downarrow0\) such that, whenever \(0<V\le(1+Cr)\delta\),
\begin{equation}\label{eq:local-halfspace-volume-lower}
\Lambda_{\Hh}(\alpha;V)
\ge
\eps^{-2}\Bigl(\Lambda_{\Hh}(\alpha\eps)-\omega_T(r)\Bigr).
\end{equation}
\end{lemma}

\begin{proof}
For $0<V_1<V_2$, rescale the volume-$V_2$ problem to unit volume and apply Corollary~\ref{cor:strict-resource-subhomogeneity}; this gives strict decrease of both the half-space and whole-space values.  If $V\le(1+Cr)\delta$, monotonicity and Lemma~\ref{lem:halfspace-volume-scaling} give
\[
\Lambda_{\Hh}(\alpha;V)\ge
\eps^{-2}(1+Cr)^{-2/N}
\Lambda_{\Hh}\!\left(\alpha\eps(1+Cr)^{1/N}\right).
\]
Uniform continuity of $\Lambda_{\Hh}$ on compact $\tau$-intervals proves \eqref{eq:local-halfspace-volume-lower}.
\end{proof}

The following localisation estimate yields the matching lower bound uniformly in the favourable set and the test function.

\begin{proposition}[Mesoscopic lower bound by the flat limit problem]\label{prop:mesoscopic-lower}
Assume that \(\Omega\) is bounded with \(C^2\) boundary.  Fix \(T<\infty\).  There exists a quantity \(\eta_\delta\downarrow0\), depending only on \(\Omega,N,\kappa,T\), such that whenever
\[
|D|=\delta,
\qquad
\alpha_\delta\eps\le T,
\qquad
\nu\in H^1(\Omega),
\qquad
d_D(\nu)>0,
\]
one has
\begin{equation}\label{eq:mesoscopic-lower}
\calE_{\alpha_\delta}(\nu)
\ge
\eps^{-2}\Bigl(\Lambda_{\Hh}(\tau_\delta)-\eta_\delta\Bigr)d_D(\nu).
\end{equation}
The estimate is uniform over all measurable \(D\subset\Omega\) of volume \(\delta\).
\end{proposition}

\begin{proof}
This is the Robin version of the standard IMS/local-flattening argument used in the Dirichlet and Neumann small-volume analyses; see \citet{FerreriVerzini2024,FerreriMazzoleniPellacciVerzini2026}.  The proof below keeps track of the dependence on $D$ and of the Robin boundary term.

Choose $r=r_\delta$ by
$r_\delta=\eps^{1/2}$ for $N\ge3$ and
$r_\delta=(1+|\log\delta|)^{-2}$ for $N=2$.  Then
\begin{equation}\label{eq:mesoscopic-scales}
r_\delta\to0,
\qquad \rho_N(\delta)r_\delta^{-2}\to0,
\qquad r_\delta\rho_N(\delta)\eps^{-2}\to0.
\end{equation}
Take a uniformly finite interior/boundary cover at scale $r$ and an IMS partition $\sum_j\chi_j^2=1$ with $\sum_j|\nabla\chi_j|^2\le Cr^{-2}$.  Set $\nu_j=\chi_j\nu$ and $d_j=\int m_D\nu_j^2$.  Then $\sum_jd_j=d_D(\nu)$ and
\[
\sum_j\calE_{\alpha_\delta}(\nu_j)
=\calE_{\alpha_\delta}(\nu)+O(r^{-2})\|\nu\|_2^2.
\]
For an interior chart, zero extension and whole-space scaling give
$\calE_{\alpha_\delta}(\nu_j)\ge\eps^{-2}\Lambda_{\R^N}d_j$ whenever $d_j>0$, and the inequality is automatic when $d_j\le0$.

For a boundary chart, flattening yields a half-space function $w_j$ and a flattened denominator $\widetilde d_j$ with
\begin{equation}\label{eq:chart-energy-comparison}
\calE_{\alpha_\delta}(\nu_j)
\ge(1-Cr)\left(\int_{\Hh}|\nabla w_j|^2+\alpha_\delta\int_{\partial\Hh}w_j^2\right),
\end{equation}
\begin{equation}\label{eq:chart-denominator-comparison}
|d_j-\widetilde d_j|\le Cr\|\nu_j\|_2^2,
\qquad |\widetilde D_j|\le(1+Cr)\delta.
\end{equation}
The boundary Jacobian gives the same $1+O(r)$ error as the bulk metric, so no additional Robin-scale loss occurs.  Lemma~\ref{lem:local-volume-monotonicity} therefore gives, uniformly for $\tau_\delta\le T$,
\[
\calE_{\alpha_\delta}(\nu_j)
\ge \eps^{-2}\bigl(\Lambda_{\Hh}(\tau_\delta)-o_r(1)\bigr)d_j
-Cr\eps^{-2}\|\nu_j\|_2^2;
\]
if $\widetilde d_j\le0$, the same bound follows from nonnegativity of the flattened energy and \eqref{eq:chart-denominator-comparison}.

Summation, $\Lambda_{\Hh}\le\Lambda_{\R^N}$, and the IMS identity yield
\[
\calE_{\alpha_\delta}(\nu)
\ge \eps^{-2}\bigl(\Lambda_{\Hh}(\tau_\delta)-o_r(1)\bigr)d_D(\nu)
-C(r\eps^{-2}+r^{-2})\|\nu\|_2^2.
\]
Lemma~\ref{lem:small-set-L2} and \eqref{eq:mesoscopic-scales} make the last term $o(1)\calE_{\alpha_\delta}(\nu)$, uniformly in $D$ and $\nu$.  Absorbing it on the left proves \eqref{eq:mesoscopic-lower}.
\end{proof}

The matching upper bound is supplied by the following exact-volume boundary transplantation lemma.

\begin{lemma}[Boundary-chart transplantation with exact volume]\label{lem:boundary-transplant}
Assume \(\partial\Omega\) is \(C^2\), let \(P\in\partial\Omega\), and let \((A,v)\) be an admissible half-space pair with \(|A|=1\) and compact support in a fixed ball.  Put \(\eps=\delta^{1/N}\) and assume
\[
\tau_\delta:=\alpha_\delta\eps\longrightarrow\tau\in[0,+\infty).
\]
Then there exist favourable sets \(D_\delta\subset\Omega\) with \(|D_\delta|=\delta\) and test functions \(u_\delta\in H^1(\Omega)\) such that
\begin{equation}\label{eq:transplant-limit}
\lim_{\delta\downarrow0}
\delta^{2/N}
\frac{\displaystyle\int_\Omega|\nabla u_\delta|^2
+\alpha_\delta\int_{\partial\Omega}u_\delta^2}
{\displaystyle\int_\Omega m_{D_\delta}u_\delta^2}
=
\frac{\displaystyle\int_{\Hh}|\nabla v|^2
+\tau\int_{\partial\Hh}v^2}
{\displaystyle\int_{\Hh}m_Av^2}.
\end{equation}
Consequently,
\begin{equation}\label{eq:critical-limsup}
\limsup_{\delta\downarrow0}
\delta^{2/N}\Lambda_\delta(\alpha_\delta)
\le \Lambda_{\Hh}(\tau).
\end{equation}
\end{lemma}

\begin{proof}
This is the usual boundary transplantation used in small-volume blow-up arguments; the only point worth recording is exact volume.  Let $\Phi$ be a $C^2$ flattening map at $P$, with $D\Phi(0)$ orthogonal.  On the fixed support of $(A,v)$,
$J(ry)=1+O(r)$, $G(ry)=I+O(r)$, and $J_\partial(ry)=1+O(r)$.  Hence
$V(r):=|\Phi(rA)|=r^N(1+O(r))$, and
$V'(r)=r^{N-1}(N+O(r))>0$ for small $r$.  There is therefore a unique $r_\delta$ with $V(r_\delta)=\delta$, and
\begin{equation}\label{eq:r-delta}
r_\delta=\eps(1+O(\eps)).
\end{equation}
Set $D_\delta=\Phi(r_\delta A)$ and
$u_\delta(\Phi(r_\delta y))=v(y)$, extended by zero outside the chart.  Change of variables gives
\[
\int|\nabla u_\delta|^2=r_\delta^{N-2}\left(\int_{\Hh}|\nabla v|^2+O(r_\delta)\right),\quad
\int_{\partial\Omega}u_\delta^2=r_\delta^{N-1}\left(\int_{\partial\Hh}v^2+O(r_\delta)\right),
\]
\[
\int m_{D_\delta}u_\delta^2=r_\delta^N\left(\int_{\Hh}m_Av^2+O(r_\delta)\right).
\]
Since $\alpha_\delta r_\delta\to\tau$, substitution proves \eqref{eq:transplant-limit}.  Approximation by compactly supported half-space competitors gives \eqref{eq:critical-limsup}.
\end{proof}

We can now identify the complete leading-order limit throughout the finite critical-scale regime.

\begin{theorem}[Critical-scale bounded-domain limit]\label{thm:critical-scale-bounded-limit}
Assume that \(\Omega\subset\R^N\), \(N\ge2\), is bounded with \(C^2\) boundary and that
\begin{equation}\label{eq:critical-scale-assumption}
\tau_\delta:=\alpha_\delta\delta^{1/N}\longrightarrow\tau\in[0,+\infty).
\end{equation}
Then
\begin{equation}\label{eq:critical-scale-limit}
\delta^{2/N}\Lambda_\delta(\alpha_\delta)
\longrightarrow
\Lambda_{\Hh}(\tau).
\end{equation}
Equivalently, since \(\Lambda_{\Hh}(\tau)\le\Lambda_{\R^N}\),
\[
\delta^{2/N}\Lambda_\delta(\alpha_\delta)
\longrightarrow
\min\{\Lambda_{\Hh}(\tau),\Lambda_{\R^N}\}.
\]
\end{theorem}

\begin{proof}
The upper bound is \eqref{eq:critical-limsup}.  Since \(\tau_\delta\to\tau\), the sequence \(\tau_\delta\) is bounded.  Proposition~\ref{prop:mesoscopic-lower}, applied to every admissible pair \((D,\nu)\) with positive weighted denominator, gives
\[
\delta^{2/N}\lambda_1(D;\alpha_\delta)
\ge
\Lambda_{\Hh}(\tau_\delta)-\eta_\delta.
\]
Taking the infimum over \(|D|=\delta\), using \(\eta_\delta\to0\), and invoking continuity of \(\Lambda_{\Hh}\) yield
\[
\liminf_{\delta\downarrow0}
\delta^{2/N}\Lambda_\delta(\alpha_\delta)
\ge\Lambda_{\Hh}(\tau).
\]
Together with the upper bound this proves \eqref{eq:critical-scale-limit}.
\end{proof}

\begin{corollary}[Transfer of the universal transition threshold to bounded domains]\label{cor:bounded-leading-phase}
Under the assumptions of Theorem~\ref{thm:critical-scale-bounded-limit},
\begin{align*}
0\le\tau<\tau_*
&\quad\Longrightarrow\quad
\delta^{2/N}\Lambda_\delta(\alpha_\delta)
\longrightarrow\Lambda_{\Hh}(\tau)<\Lambda_{\R^N},\\
\tau\ge\tau_*
&\quad\Longrightarrow\quad
\delta^{2/N}\Lambda_\delta(\alpha_\delta)
\longrightarrow\Lambda_{\R^N}.
\end{align*}
Thus the same finite universal threshold \(\tau_*\) appears in the leading-order bounded-domain value.  The scalar limit does not, by itself, distinguish a compact boundary profile from escape from the boundary when \(\tau\ge\tau_*\); that is the next profile-classification problem.
\end{corollary}

The theorem contains as a special case the full physical-Robin subcritical regime previously obtained only by comparison with the Neumann problem.

\begin{proposition}[Concentration in a single mesoscopic chart]\label{prop:one-chart-concentration}\label{prop:one-chart-supercritical}
Assume the hypotheses of Theorem~\ref{thm:critical-scale-bounded-limit}, let $(D_\delta,u_\delta)$ be normalised almost minimisers,
\begin{equation}\label{eq:bounded-almostmin-normalisation}
\int_\Omega m_{D_\delta}u_\delta^2=1,
\qquad
\eps^2\calE_{\alpha_\delta}(u_\delta)\longrightarrow\Lambda_{\Hh}(\tau),
\end{equation}
and use the mesoscopic IMS cover from Proposition~\ref{prop:mesoscopic-lower}.  Then some chart $U_{j_\delta,\delta}$ satisfies
\begin{equation}\label{eq:one-chart-supercritical-volume}
\frac{|D_\delta\cap U_{j_\delta,\delta}|}{\delta}\longrightarrow1.
\end{equation}
If $0\le\tau<\tau_*$, this chart is necessarily a boundary chart; hence there are $P_\delta\in\partial\Omega$ and a fixed $C$ such that
\begin{equation}\label{eq:one-chart-resource-concentration}
\frac{|D_\delta\setminus B_{Cr_\delta}(P_\delta)|}{\delta}\longrightarrow0.
\end{equation}
\end{proposition}

\begin{proof}
Fix $\eta>0$ and suppose every chart contains at most $(1-\eta)\delta$ favourable volume.  After boundary flattening the corresponding fraction is at most $1-\eta/2$.  Exact volume scaling and Corollary~\ref{cor:strict-resource-subhomogeneity} then give a uniform $c_\eta>0$ such that every local positive reduced mass costs at least
\[
\eps^{-2}\bigl(\Lambda_{\Hh}(\tau)+c_\eta\bigr)
\]
in a boundary chart, while an interior chart costs at least
$\eps^{-2}\Lambda_{\R^N}$.  When $\tau\ge\tau_*$ both values equal $k^2$ at unit favourable volume, but strict subhomogeneity still gives the same positive gap for favourable-volume fractions bounded by $1-\eta$; the quantitative small-volume estimate in Lemma~\ref{lem:asymptotic-B-splitting} makes this uniform as the fraction tends to zero.  When $\tau<\tau_*$ the interior gap is already
$\Lambda_{\R^N}-\Lambda_{\Hh}(\tau)>0$.

Repeating the IMS summation in Proposition~\ref{prop:mesoscopic-lower}, now with positive local reduced masses, yields
\[
\eps^2\calE_{\alpha_\delta}(u_\delta)
\ge \Lambda_{\Hh}(\tau)+c_\eta-o(1),
\]
contrary to almost minimality and Theorem~\ref{thm:critical-scale-bounded-limit}.  Thus one chart contains $(1-o(1))\delta$ favourable volume.  When $\Lambda_{\Hh}(\tau)<k^2$, an interior dominant chart would have leading cost at least $k^2>\Lambda_{\Hh}(\tau)$, so it must be a boundary chart.  The bounded overlap and diameter $O(r_\delta)$ give \eqref{eq:one-chart-resource-concentration}.
\end{proof}

\begin{lemma}[Reduction of the dominant chart to the unit-volume limit problem]\label{lem:dominant-core-reduction}
Let \((D_\delta,u_\delta)\) be normalised almost minimisers under \(\tau_\delta\to\tau<\infty\), and suppose a mesoscopic chart \(U_\delta\) satisfies
\[
|D_\delta\cap U_\delta|=(1-o(1))\delta.
\]
Choose a slightly larger chart \(U_\delta^+\) and a two-function IMS partition
\[
\chi_\delta^2+\eta_\delta^2=1,
\qquad
\chi_\delta=1\text{ on }U_\delta,
\qquad
\operatorname{supp}\chi_\delta\subset U_\delta^+,
\qquad
|\nabla\chi_\delta|+|\nabla\eta_\delta|\le Cr_\delta^{-1}.
\]
Put \(u_\delta^{\rm c}=\chi_\delta u_\delta\) and \(u_\delta^{\rm r}=\eta_\delta u_\delta\).  Then
\begin{equation}\label{eq:core-remainder-mass-energy}
 d_D(u_\delta^{\rm c})\to1,
\qquad
 d_D(u_\delta^{\rm r})\to0,
\qquad
 \eps^2\calE_{\alpha_\delta}(u_\delta^{\rm r})\to0,
\end{equation}
and
\begin{equation}\label{eq:core-energy-limit}
\eps^2\calE_{\alpha_\delta}(u_\delta^{\rm c})
\longrightarrow\Lambda_{\Hh}(\tau).
\end{equation}
If \(U_\delta\) is an interior chart and \(\tau<\tau_*\), such a dominant chart is impossible.

If \(U_\delta\) is a boundary chart, flatten it and rescale by \(\eps\), with the amplitude factor \(\eps^{N/2}\).  The resulting functions \(q_\delta\in H^1(\Hh)\), extended by zero across the artificial sides, satisfy after multiplication by \(1+o(1)\)
\begin{equation}\label{eq:boundary-core-minseq}
\mathcal M_1^{\Hh}(q_\delta)=1,
\qquad
E_\tau(q_\delta)\longrightarrow\Lambda_{\Hh}(\tau).
\end{equation}
If \(U_\delta\) is an interior chart, the corresponding Euclidean rescaling produces \(q_\delta\in H^1(\R^N)\) with
\begin{equation}\label{eq:interior-core-minseq}
\mathcal M_1^{\R^N}(q_\delta)=1,
\qquad
\int_{\R^N}|\nabla q_\delta|^2\longrightarrow k^2.
\end{equation}
\end{lemma}

\begin{proof}
Let $u_\delta^{\rm c}=\chi_\delta u_\delta$ and $u_\delta^{\rm r}=\eta_\delta u_\delta$, and write the corresponding reduced masses as $d_\delta^{\rm c}+d_\delta^{\rm r}=1$.  The two-function IMS identity, Lemma~\ref{lem:small-set-L2}, and the choice of $r_\delta$ give
\begin{equation}\label{eq:two-core-IMS-scaled}
\eps^2\bigl(\calE_{\alpha_\delta}(u_\delta^{\rm c})+\calE_{\alpha_\delta}(u_\delta^{\rm r})\bigr)
=\Lambda_{\Hh}(\tau)+o(1).
\end{equation}
The core sees unit favourable volume up to $o(1)$, while the remainder sees $s_\delta\delta$ with $s_\delta\to0$.  The local flat/whole-space lower bound gives
$\eps^2\calE(u_\delta^{\rm c})\ge(\Lambda_{\Hh}(\tau)-o(1))(d_\delta^{\rm c})_+$ for a boundary core (and $k^2(d_\delta^{\rm c})_+$ for an interior one).  If $d_\delta^{\rm r}>0$, Theorem~\ref{thm:critical-scale-bounded-limit} applied at volume $s_\delta\delta$ gives
\[
\eps^2\Lambda_{s_\delta\delta}(\alpha_\delta)
=s_\delta^{-2/N}\bigl(\Lambda_{\Hh}(0)+o(1)\bigr)\to\infty;
\]
if $d_\delta^{\rm r}\le0$, its energy is nonnegative.  Equation \eqref{eq:two-core-IMS-scaled} therefore forces
$d_\delta^{\rm r}\to0$, $d_\delta^{\rm c}\to1$, zero scaled remainder energy, and \eqref{eq:core-energy-limit}.  For $\tau<\tau_*$ an interior core is impossible because $k^2>\Lambda_{\Hh}(\tau)$.

For a boundary core, flatten and rescale by $\eps$.  The chart errors are $o(1)$ and the favourable volume is $1+o(1)$, so after a scalar normalisation
$\mathcal M_1^{\Hh}(q_\delta)=1$ and
$E_\tau(q_\delta)\to\Lambda_{\Hh}(\tau)$; boundedness follows as in Lemma~\ref{lem:normalised-bounded}.  The interior rescaling is identical without the boundary term and gives \eqref{eq:interior-core-minseq}.  This proves \eqref{eq:core-remainder-mass-energy}--\eqref{eq:interior-core-minseq}.
\end{proof}

\begin{theorem}[Bounded-domain profile classification]\label{thm:bounded-profile-away-critical}\label{thm:bounded-profile-critical}
Assume $\Omega\subset\R^N$, $N\ge2$, is bounded with $C^2$ boundary, let $D_\delta$ be optimisers, and normalise the positive eigenfunctions by $\int_\Omega m_{D_\delta}u_\delta^2=1$.  Suppose $\tau_\delta\to\tau<\infty$ and put $\eps_n=\delta_n^{1/N}$ along an arbitrary sequence $\delta_n\downarrow0$.

\emph{(i) Subcritical regime $0\le\tau<\tau_*$.}  After a subsequence there are boundary points $P_n$, flattening maps, tangential shifts $a_n'$, and a normalised half-space optimiser $(A,v)$ at $\tau$ such that the dominant rescaled cores satisfy
\begin{equation}\label{eq:subcritical-bounded-profile}
q_n(a_n'+\cdot)\longrightarrow v\qquad\text{strongly in }H^1(\Hh).
\end{equation}
If $F_n\subset\Hh$ denotes the flattened and $\eps_n$-rescaled favourable part of the dominant core before tangential recentering, then
\begin{equation}\label{eq:subcritical-set-profile}
 |(F_n-a_n')\triangle A|\longrightarrow0.
\end{equation}
In particular the full physical patch is tight on the $\eps_n$ scale around the corresponding boundary centres: for every $\eta>0$ there is $R<\infty$ such that
\begin{equation}\label{eq:subcritical-patch-tightness}
\frac{|D_{\delta_n}\setminus\Psi_n(\eps_n(B_R^++a_n'))|}{\delta_n}<\eta
\end{equation}
for all large $n$.

\emph{(ii) Supercritical regime $\tau>\tau_*$.}  There are points $x_n\in\Omega$ with
\begin{equation}\label{eq:supercritical-normal-distance}
\frac{\operatorname{dist}(x_n,\partial\Omega)}{\eps_n}\longrightarrow\infty
\end{equation}
and localised rescalings $Q_n$ such that
\begin{equation}\label{eq:supercritical-bounded-profile}
Q_n\longrightarrow\Phi\qquad\text{strongly in }H^1(\R^N),
\end{equation}
while
\begin{equation}\label{eq:supercritical-set-profile}
\left|\left(\frac{D_{\delta_n}-x_n}{\eps_n}\right)\triangle B_{R_1}\right|\longrightarrow0.
\end{equation}

\emph{(iii) Critical regime $\tau=\tau_*$.}  Every sequence has a subsequence with exactly one of two alternatives: either a compact boundary profile as in (i), with a threshold optimiser $(A_*,v_*)$, or an escaping whole-space profile as in (ii).  The statement is a subsequential classification and does not assert that both alternatives occur in a fixed bounded domain.
\end{theorem}

\begin{proof}
Proposition~\ref{prop:one-chart-concentration} gives a dominant mesoscopic core and Lemma~\ref{lem:dominant-core-reduction} reduces it to a unit-volume model.  If $\tau<\tau_*$ the core is on the boundary and Theorem~\ref{thm:strict-gap-attainment} gives \eqref{eq:subcritical-bounded-profile} after tangential recentering.  If $\tau>\tau_*$, an interior core is a whole-space minimising sequence and a boundary core is a half-space minimising sequence at the whole-space value; Proposition~\ref{prop:whole-space-compactness} and Theorem~\ref{thm:postthreshold-profile}, respectively, give the same whole-space limit $\Phi$.  Interior IMS charts lie at distance $\gtrsim r_{\delta_n}$ from the boundary, while boundary-chart escape occurs at a height tending to infinity in patch variables; since $r_{\delta_n}/\eps_n\to\infty$, both yield \eqref{eq:supercritical-normal-distance}.  At $\tau=\tau_*$ the same reduction, followed by Theorem~\ref{thm:critical-halfspace-profile}, gives the compact/escape dichotomy.

We finally identify the favourable sets.  In every compact core the actual rescaled favourable portion has volume $1+o(1)$ and, by the core normalisation, has bathtub deficit $o(1)$.  After modifying it on a set of measure $o(1)$ to obtain exact unit volume, strong $L^2$ convergence of the profiles and Lemma~\ref{lem:bathtub-stability} give convergence in measure to the unique bathtub set of the limiting profile: $A$ (or $A_*$) in the half-space case and $B_{R_1}$ in the whole-space case.  Proposition~\ref{prop:one-chart-concentration} leaves only $o(\delta_n)$ favourable volume outside the dominant core, so the same convergence holds for the full physical sets.  This proves all three cases.
\end{proof}

\begin{corollary}[Fixed or subcritical Robin coefficients]\label{thm:tau-zero-leading}
If
\[
\alpha_\delta\delta^{1/N}\longrightarrow0,
\]
then
\begin{equation}\label{eq:tau-zero-limit}
\delta^{2/N}\Lambda_\delta(\alpha_\delta)
\longrightarrow
\Lambda_{\Hh}(0)
=2^{-2/N}\Lambda_{\R^N}.
\end{equation}
In particular this holds for every fixed finite Robin coefficient \(\alpha_\delta\equiv\alpha\ge0\), and more generally whenever \(\alpha_\delta=o(\delta^{-1/N})\).
\end{corollary}

\begin{remark}[Complete leading-order profile classification]\label{rem:profile-after-value}
Theorem~\ref{thm:bounded-profile-critical} upgrades the scalar limit to a subsequential profile classification for every finite $\tau$.  In the subcritical regime the natural statement is compactness modulo tangential translations and subsequences, because uniqueness of the half-space optimiser has not been proved.  The supercritical limit is unique up to whole-space translation.  At the threshold, compact boundary-critical and escaping whole-space profiles are the only two leading alternatives.
\end{remark}

\subsection{First-order selection in the critical window}\label{subsec:critical-first-development}

The leading statement $\tau_\delta\to\tau_*$ is not fine enough to decide between the two critical configurations.  By Corollary~\ref{cor:critical-window-halfspace}, changing $\tau$ by $O(\eps)$ already changes the flat half-space value at the same order at which curvature enters boundary flattening.  The natural next scale is therefore
\begin{equation}\label{eq:critical-window-sigma}
\tau_\delta
=\tau_*+\sigma\eps+o(\eps),
\qquad\text{equivalently}\qquad
\alpha_\delta=\frac{\tau_*}{\eps}+\sigma+o(1),
\end{equation}
with $\sigma\in\R$ fixed.

We next derive the first-order boundary transplantation coefficient.  The Fermi-coordinate expansion contains the full second fundamental form.  Theorem~\ref{thm:threshold-tangential-symmetry}, however, shows that every threshold optimiser has an isotropic tangential energy tensor.  Consequently the final response depends on the domain boundary only through its mean curvature.

Let $P\in\partial\Omega$.  With the convention fixed in Section~\ref{sec:framework},
$\mathsf S_P\xi=D_\xi\nu(P)$ on $T_P\partial\Omega$; in particular a sphere has positive outward principal curvatures.  Write $\kappa_1(P),\ldots,\kappa_{N-1}(P)$ for the eigenvalues of $\mathsf S_P$ and set
\begin{equation}\label{eq:mean-curvature-trace}
h_P:=\operatorname{tr}\mathsf S_P
=\sum_{i=1}^{N-1}\kappa_i(P)
=(N-1)H_{\partial\Omega}(P).
\end{equation}
We use the threshold-profile quantities $T(v)$, $a(A)$, and $\mathsf M_T(v)$ defined in \eqref{eq:critical-response-T}--\eqref{eq:critical-response-M}.

For the next formula define the critical curvature--detuning response
\begin{equation}\label{eq:G-critical-def}
\begin{aligned}
\mathcal G_\sigma(P;A,v)
:=\;&\sigma T(v)
+2\,\mathsf S_P:\mathsf M_T(v)\\
&+h_P\left[
-\frac12T(v)
+\frac{a(A)}{N}\bigl(\tau_*T(v)-2k^2\bigr)
\right].
\end{aligned}
\end{equation}
Here $:$ denotes the Frobenius contraction after identifying $T_P\partial\Omega$ with $\R^{N-1}$ by an orthonormal frame.  By Corollary~\ref{cor:scalar-critical-response}, this tensorial form simplifies for every threshold optimiser to the scalar expression \eqref{eq:G-critical-scalar}.

\begin{proposition}[First-order transplantation of a threshold optimiser]\label{prop:critical-first-order-transplant}
Assume $\partial\Omega$ is of class $C^{2,1}$ and \eqref{eq:critical-window-sigma} holds.  Fix $P\in\partial\Omega$ and a normalised threshold optimiser $(A,v)\in\mathfrak O_*$.  Then there are sets $D_\delta\subset\Omega$ with $|D_\delta|=\delta$ and test functions $u_\delta\in H^1(\Omega)$ such that
\begin{equation}\label{eq:critical-first-order-transplant}
\eps^2
\frac{\displaystyle\int_\Omega|\nabla u_\delta|^2
+\alpha_\delta\int_{\partial\Omega}u_\delta^2}
{\displaystyle\int_\Omega m_{D_\delta}u_\delta^2}
=
 k^2+\eps\,\mathcal G_\sigma(P;A,v)+o(\eps).
\end{equation}
\end{proposition}

\begin{proof}
Choose geodesic principal coordinates on $\partial\Omega$ centred at $P$ and inward Fermi distance $z_N>0$.  In these coordinates the volume Jacobian and the energy metric satisfy, as $z\to0$,
\begin{align}
J(z)
&=1-h_Pz_N+O(|z|^2),\label{eq:Fermi-J-first}\\
J(z)\,G^{-1}(z)
&=I+z_N\bigl(2\widetilde{\mathsf S}_P-h_PI\bigr)+O(|z|^2),
\label{eq:Fermi-metric-first}
\end{align}
where $\widetilde{\mathsf S}_P$ acts as $\mathsf S_P$ on tangential components and vanishes on the normal component.  The boundary Jacobian in geodesic coordinates is
\begin{equation}\label{eq:Fermi-boundary-first}
J_\partial(z',0)=1+O(|z'|^2),
\end{equation}
so there is no linear surface-measure term.

Let $r=r_\delta$ be chosen so that the Fermi image of $rA$ has volume exactly $\delta=\eps^N$.  Lemma~\ref{lem:threshold-exponential-moments} and the boundedness of $A$ give
\[
|\Phi_P(rA)|
=r^N\left(1-rh_Pa(A)+O(r^2)\right),
\]
and hence
\begin{equation}\label{eq:r-critical-first}
\frac r\eps
=1+\eps\frac{h_Pa(A)}N+O(\eps^2).
\end{equation}
Set $D_\delta=\Phi_P(rA)$.  Transplant $v$ through the same Fermi chart and multiply it by a cutoff which equals one on a fixed physical neighbourhood of $P$ and vanishes before the edge of the chart.  By \eqref{eq:threshold-exp-moment}, the cutoff error is exponentially small in $1/r$, hence $o(r)$ after the common leading factors are removed.

Write
\[
M_1:=\int_{\Hh}y_Nm_Av^2,
\qquad
E_1:=\int_{\Hh}y_N|\nabla v|^2.
\]
Using \eqref{eq:Fermi-J-first}--\eqref{eq:Fermi-boundary-first} and the normalisation $\int m_Av^2=1$ gives
\begin{align}
\int_\Omega|\nabla u_\delta|^2
&=r^{N-2}\left[
\int_{\Hh}|\nabla v|^2
+r\left(2\mathsf S_P:\mathsf M_T(v)-h_PE_1\right)
+O(r^2)
\right],\label{eq:first-order-bulk-exp}\\
\alpha_\delta\int_{\partial\Omega}u_\delta^2
&=r^{N-2}\left[
\alpha_\delta r\,T(v)+O(r^2)\right],\label{eq:first-order-robin-exp}\\
\int_\Omega m_{D_\delta}u_\delta^2
&=r^N\left[1-rh_PM_1+O(r^2)\right].\label{eq:first-order-den-exp}
\end{align}
The critical-window assumption and \eqref{eq:r-critical-first} yield
\begin{equation}\label{eq:alpha-r-critical}
\alpha_\delta r
=\tau_*+\eps\left(
\sigma+\tau_*\frac{h_Pa(A)}N
\right)+o(\eps).
\end{equation}
Finally, testing the threshold Euler--Lagrange equation with $y_Nv$ (after a cutoff and then passing to the limit using Lemma~\ref{lem:threshold-exponential-moments}) gives the moment identity
\begin{equation}\label{eq:normal-moment-identity}
E_1-\frac12T(v)=k^2M_1.
\end{equation}
Indeed, $\int v\partial_Nv=-T(v)/2$, while the Robin boundary term vanishes because $y_N=0$ on $\partial\Hh$.

Insert \eqref{eq:r-critical-first}--\eqref{eq:normal-moment-identity} into the quotient defined by \eqref{eq:first-order-bulk-exp}--\eqref{eq:first-order-den-exp}.  The zeroth-order term is
$E_{\tau_*}(v)=k^2$, and the coefficient of $\eps$ is exactly \eqref{eq:G-critical-def}.  This proves \eqref{eq:critical-first-order-transplant}.
\end{proof}

\begin{proposition}[Vanishing first-order correction for interior concentration]\label{prop:critical-interior-zero-correction}
Assume only that $\tau_\delta$ stays bounded.  Fix $x_0\in\Omega$ with $\operatorname{dist}(x_0,\partial\Omega)>0$.  There are unit-volume-ball favourable sets centred at $x_0$ and compactly supported test functions for which
\begin{equation}\label{eq:critical-interior-zero}
\eps^2 Q_\delta
=k^2+O(e^{-c/\eps})
\end{equation}
for some $c>0$.  In particular the escaping/interior benchmark has coefficient zero at order $\eps$.
\end{proposition}

\begin{proof}
Take the whole-space unit-volume optimiser $(B_{R_1},\Phi)$, place the ball $x_0+\eps B_{R_1}$ inside $\Omega$, and scale $\Phi((x-x_0)/\eps)$.  Multiply by a cutoff which is identically one in a fixed neighbourhood of $x_0$ and vanishes before reaching $\partial\Omega$.  The favourable volume is exactly $\eps^N=\delta$, the boundary trace of the test function is zero, and the exponential Yukawa decay of $\Phi$ makes both numerator and denominator cutoff errors $O(e^{-c/\eps})$ after rescaling.  The whole-space quotient is exactly $k^2$.
\end{proof}

Define
\begin{equation}\label{eq:Gamma-critical-boundary}
\Gamma_{\rm bd}(\sigma)
:=
\inf_{P\in\partial\Omega}
\inf_{(A,v)\in\mathfrak O_*}
\mathcal G_\sigma(P;A,v).
\end{equation}
The preceding two propositions already give a rigorous one-sided development of the optimal value.

\begin{corollary}[First-order upper bound in the critical window]\label{cor:critical-first-order-upper}
Under the assumptions of Proposition~\ref{prop:critical-first-order-transplant},
\begin{equation}\label{eq:critical-first-order-upper}
\limsup_{\delta\downarrow0}
\frac{\eps^2\Lambda_\delta(\alpha_\delta)-k^2}{\eps}
\le
\min\{0,\Gamma_{\rm bd}(\sigma)\}.
\end{equation}
\end{corollary}

\begin{proof}
Proposition~\ref{prop:critical-first-order-transplant} gives the upper bound by every compact threshold pair placed at every boundary point.  Proposition~\ref{prop:critical-interior-zero-correction} gives the competing coefficient $0$.  Take the infimum.
\end{proof}

\begin{remark}[Mean curvature is the complete first-order geometric datum]\label{rem:critical-full-curvature}
The Fermi-coordinate expansion produces the tensor contraction
$2\mathsf S_P:\mathsf M_T(v)$, but Theorem~\ref{thm:threshold-tangential-symmetry} forces
$\mathsf M_T(v)=\mu_T(v)I_{N-1}$ for every threshold optimiser.  Thus the complete first-order geometric dependence is the scalar mean-curvature factor $h_P$ in \eqref{eq:G-critical-scalar}; individual principal curvatures and tangential orientation do not survive optimisation.  This mirrors the sharp Neumann expansion of \citet{FerreriMazzoleniPellacciVerzini2026}, although the Robin threshold coefficient $\mathcal C(A,v)$ is determined by the critical half-space optimiser rather than by the reflected Neumann ball.
\end{remark}

\begin{lemma}[Coercivity for small favourable volume at the patch scale]\label{lem:small-volume-coercivity}
Let $\eps_n\downarrow0$ and let $\lambda_n>0$ satisfy
\[
0<c_0\le \eps_n^2\lambda_n\le C_0<\infty.
\]
If $E_n\subset\Omega$ is measurable with
\[
|E_n|=s_n\eps_n^N,
\qquad s_n\longrightarrow0,
\]
then there is a modulus $\omega(s)\downarrow0$ such that, for every $w\in H^1(\Omega)$,
\begin{equation}\label{eq:small-volume-coercivity}
\lambda_n\int_{E_n}w^2
\le
\omega(s_n)
\left(
\int_\Omega|\nabla w|^2
+\lambda_n\int_\Omega w^2
\right).
\end{equation}
The modulus depends only on $\Omega,N,c_0,C_0$.
\end{lemma}

\begin{proof}
For $N\ge3$, H\"older and the Sobolev embedding on the fixed bounded domain give
\[
\int_{E_n}w^2
\le C|E_n|^{2/N}
\left(\|\nabla w\|_2^2+\|w\|_2^2\right).
\]
Since $|E_n|^{2/N}=s_n^{2/N}\eps_n^2$ and
$\lambda_n\eps_n^2$ stays in a compact subset of $(0,+\infty)$,
\eqref{eq:small-volume-coercivity} follows with
$\omega(s)=Cs^{2/N}$.

For $N=2$, use the Gagliardo--Nirenberg inequality
\[
\|w\|_4^2
\le C\|w\|_2\bigl(\|\nabla w\|_2+\|w\|_2\bigr).
\]
Then
\[
\int_{E_n}w^2
\le |E_n|^{1/2}\|w\|_4^2
\le C\eps_n s_n^{1/2}
\|w\|_2\bigl(\|\nabla w\|_2+\|w\|_2\bigr).
\]
Using $\sqrt{\lambda_n}\eps_n\asymp1$ and Young's inequality yields
\eqref{eq:small-volume-coercivity} with
$\omega(s)=Cs^{1/2}$.  This proves the claim in every dimension $N\ge2$.
\end{proof}

\begin{lemma}[Liouville alternatives in the unfavourable region]\label{lem:hostile-liouville}
Let $k>0$.
\begin{enumerate}[label=(\roman*)]
\item If $V\in L^\infty(\R^N)$ satisfies $(-\Delta+k^2)V=0$ in distributions on $\R^N$, then $V\equiv0$.
\item If $V\in L^\infty(\Hh)\cap H^1_{\rm loc}(\Hh)$ satisfies
\[
(-\Delta+k^2)V=0\quad\text{in }\Hh,
\qquad
\partial_\nu V+\tau V=0\quad\text{on }\partial\Hh
\]
weakly for some $\tau\ge0$, then $V\equiv0$.
\end{enumerate}
\end{lemma}

\begin{proof}
For (i), boundedness makes $V$ a tempered distribution. Fourier transformation gives
\[
(|\xi|^2+k^2)\widehat V(\xi)=0.
\]
The multiplier is everywhere strictly positive and has a smooth reciprocal with polynomially bounded derivatives, so multiplication by its reciprocal is legitimate on tempered distributions. Hence $\widehat V=0$.

For (ii), take the tangential Fourier transform in $x'$ in the sense of tempered distributions and put $a(\xi)=\sqrt{k^2+|\xi|^2}$. The normal equation is
\[
-\partial_N^2\widehat V+a(\xi)^2\widehat V=0.
\]
Boundedness as $x_N\to\infty$ eliminates the growing mode, so
\[
\widehat V(\xi,x_N)=\widehat f(\xi)e^{-a(\xi)x_N},
\]
where $f$ is the boundary trace. Since the outward normal to $\Hh$ is $-e_N$, the Robin condition becomes
\[
(a(\xi)+\tau)\widehat f(\xi)=0.
\]
Because $a+\tau\ge k>0$ has a smooth reciprocal of admissible growth, $\widehat f=0$, and the displayed representation gives $V\equiv0$.
\end{proof}

\begin{lemma}[Uniform localisation at the patch scale]\label{lem:uniform-exp-localisation}
Assume $\partial\Omega$ is $C^{2,1}$ and
\[
\tau_{\delta_n}=\tau_*+O(\eps_n),
\qquad \eps_n=\delta_n^{1/N}.
\]
Let $(D_n,u_n)$ be bounded-domain optimisers, normalised by
\[
\int_\Omega m_{D_n}u_n^2=1,
\]
and suppose that, along the subsequence under consideration, Theorem~\ref{thm:bounded-profile-critical} gives either a compact threshold boundary profile or an escaping whole-space profile.  Let $z_n$ denote the corresponding physical core centre: $z_n\in\partial\Omega$ in the compact case, after absorbing the tangential recentering into the boundary chart, and $z_n=x_n$ in the escaping case.

Then there are constants $R_0,C,c>0$, independent of $n$, such that for all sufficiently large $n$,
\begin{equation}\label{eq:full-set-patch-containment}
D_n\subset B_{R_0\eps_n}(z_n)\cap\Omega
\qquad\text{up to null sets},
\end{equation}
and
\begin{equation}\label{eq:uniform-exp-localisation}
\eps_n^{N/2}u_n(x)
+\eps_n^{N/2+1}|\nabla u_n(x)|
\le
C\exp\!\left(-c\frac{|x-z_n|}{\eps_n}\right)
\end{equation}
whenever $x$ is outside $B_{2R_0\eps_n}(z_n)$; the gradient estimate is understood almost everywhere.  Consequently, after the corresponding boundary or interior blow-up, all polynomially weighted $L^2$ and $H^1$ tails are uniformly integrable.
\end{lemma}

\begin{proof}
We first upgrade the dominant-core reduction from energy tightness to global $L^2$ tightness.  Let $\chi_n^2+\eta_n^2=1$ be the two-function partition from Lemma~\ref{lem:dominant-core-reduction}, with $\chi_n=1$ on the dominant mesoscopic chart, and set
\[
w_n:=\eta_nu_n.
\]
The favourable volume meeting $\operatorname{supp}\eta_n$ is $o(\delta_n)$ by the one-core property.  Writing
\[
\lambda_n:=\Lambda_{\delta_n}(\alpha_{\delta_n}),
\]
Theorem~\ref{thm:critical-scale-bounded-limit} gives
$\eps_n^2\lambda_n\to k^2>0$.  Testing the exact eigenvalue equation with $\eta_n^2u_n$ and using
\[
|\nabla(\eta_nu_n)|^2
=\eta_n^2|\nabla u_n|^2
+2\eta_nu_n\nabla\eta_n\cdot\nabla u_n
+u_n^2|\nabla\eta_n|^2
\]
gives the exact identity
\begin{align}
&\int_\Omega|\nabla w_n|^2
+\alpha_{\delta_n}\int_{\partial\Omega}w_n^2
+\lambda_n\int_\Omega w_n^2\notag\\
&\qquad
=(\kappa+1)\lambda_n\int_{D_n}w_n^2
+\int_\Omega u_n^2|\nabla\eta_n|^2.
\label{eq:remainder-coercive-identity}
\end{align}
Lemma~\ref{lem:small-volume-coercivity} absorbs the first term on the right into the left-hand side.  The second term is negligible at the patch scale quantitatively.  Indeed,
$|\nabla\eta_n|\le Cr_{\delta_n}^{-1}$ and Lemma~\ref{lem:small-set-L2}, applied to the normalised eigenfunction itself, gives
\[
\|u_n\|_2^2
\le C\rho_N(\delta_n)\|\nabla u_n\|_2^2
\le C\rho_N(\delta_n)\lambda_n.
\]
Since $\eps_n^2\lambda_n$ stays bounded,
\[
\eps_n^2\int_\Omega u_n^2|\nabla\eta_n|^2
\le
C\frac{\rho_N(\delta_n)}{r_{\delta_n}^2}
\,\eps_n^2\lambda_n
=o(1)
\]
by \eqref{eq:mesoscopic-scales}.  Hence
\begin{equation}\label{eq:remainder-L2-vanishing}
\|w_n\|_2^2
+\eps_n^2\|\nabla w_n\|_2^2
\longrightarrow0.
\end{equation}
Together with the strong $H^1$ convergence of the rescaled dominant core, this implies global patch-scale $L^2$ tightness about $z_n$:
\begin{equation}\label{eq:global-patch-L2-tightness}
\lim_{R\to\infty}\limsup_{n\to\infty}
\int_{\Omega\setminus B_{R\eps_n}(z_n)}u_n^2=0.
\end{equation}

Let $t_n>0$ be a bathtub level from Proposition~\ref{prop:bounded-existence-bathtub} and set
\[
\theta_n:=\eps_n^{N/2}t_n.
\]
The classified limit profile is strictly positive and has a positive bathtub level.  The required local compactness follows as follows.  After recentering at $z_n$ and scaling by $\eps_n$, the eigenfunctions solve
\[
-\Delta q_n=\mu_n b_n q_n,
\qquad |b_n|\le\max\{1,\kappa\},
\qquad \mu_n=\eps_n^2\lambda_n\to k^2,
\]
with a uniformly bounded Robin coefficient in boundary charts.  Local Moser bounds followed by $W^{2,p}$ estimates (interior or flat-boundary Robin estimates) therefore give uniform $C^{0,\beta}$ bounds on every fixed rescaled ball.  The already known strong $H^1$ convergence consequently improves, after subsequence extraction, to local uniform convergence to the classified profile.

If $\theta_n\to0$, choose a compact set $K$ in the limiting half-space or whole-space model with $|K|>1$ and with the positive limit profile bounded below on $K$.  Local uniform convergence then gives $q_n>\theta_n$ on $K$ for all large $n$.  Since the bathtub characterisation contains the strict superlevel set in the favourable region, the rescaled favourable volume would exceed one, contradicting $|D_n|=\eps_n^N$.  Thus
\begin{equation}\label{eq:theta-uniform-positive}
\liminf_{n\to\infty}\theta_n>0.
\end{equation}

The same local Moser estimate, now centred at arbitrary points of $\Omega$, and \eqref{eq:global-patch-L2-tightness} yield
\begin{equation}\label{eq:uniform-sup-tail-preexp}
\lim_{R\to\infty}\limsup_{n\to\infty}
\eps_n^{N/2}
\sup_{\Omega\setminus B_{R\eps_n}(z_n)}u_n=0.
\end{equation}
Indeed, if $x_n$ lies outside $B_{R\eps_n}(z_n)$, the rescaled Moser estimate on a unit ball or boundary half-ball centred at $x_n$ gives
\[
\eps_n^{N/2}u_n(x_n)
\le C
\left(
\int_{B_{c\eps_n}(x_n)\cap\Omega}u_n^2
\right)^{1/2}.
\]
For $R$ larger than a fixed constant the ball on the right lies outside
$B_{(R-c)\eps_n}(z_n)$, and \eqref{eq:global-patch-L2-tightness} makes this uniformly small.  This proves \eqref{eq:uniform-sup-tail-preexp}.  Combining it with \eqref{eq:theta-uniform-positive} and the bathtub inclusion
$D_n\subset\{u_n\ge t_n\}$ proves \eqref{eq:full-set-patch-containment} for a fixed $R_0$.

Outside $B_{R_0\eps_n}(z_n)$ the equation is therefore the constant-coefficient equation in the unfavourable region
\[
-\Delta u_n=-\lambda_nu_n.
\]
We claim that there are $L>0$ and $q\in(0,1)$ such that, for every sufficiently large $R$ and all large $n$,
\begin{equation}\label{eq:halving-iteration}
M_n(R+L)\le qM_n(R),
\qquad
M_n(R):=
\eps_n^{N/2}
\sup_{|x-z_n|\ge R\eps_n}u_n(x).
\end{equation}
Suppose that no such $L$ exists, with for instance $q=1/2$.  Then one can choose integers $L_j\to\infty$, indices $n_j\to\infty$, radii $R_j\to\infty$, and points $x_j$ satisfying
\[
|x_j-z_{n_j}|\ge (R_j+L_j)\eps_{n_j},
\qquad
\eps_{n_j}^{N/2}u_{n_j}(x_j)
\ge\frac14 M_{n_j}(R_j),
\]
while $M_{n_j}(R_j+L_j)>\frac12M_{n_j}(R_j)$.  Divide by $M_{n_j}(R_j)$ and rescale by $\eps_{n_j}$ around $x_j$.  On every ball of radius at most $L_j/2$ the normalised functions are bounded by one, because that ball remains outside $B_{R_j\eps_{n_j}}(z_{n_j})$, and their value at the origin is at least $1/4$.

After passing to a subsequence, the rescaled domains converge locally either to $\R^N$ or, after a rigid motion, to $\Hh$; in the boundary case the $C^{2,1}$ boundary becomes flat and the scaled Robin coefficient tends to $\tau_*$.  Local elliptic compactness therefore gives a bounded nonzero nonnegative solution of either
\[
(-\Delta+k^2)V=0\quad\text{in }\R^N
\]
or
\[
(-\Delta+k^2)V=0\quad\text{in }\Hh,
\qquad
\partial_\nu V+\tau_*V=0\quad\text{on }\partial\Hh.
\]
Both limiting alternatives are impossible by Lemma~\ref{lem:hostile-liouville}.  This contradicts the normalisation at the chosen point and proves \eqref{eq:halving-iteration}.

Iteration of \eqref{eq:halving-iteration} gives the pointwise part of \eqref{eq:uniform-exp-localisation}.  Standard interior and boundary elliptic estimates on $\eps_n$-balls then give the gradient estimate.  Integrating these estimates proves uniform exponential, hence polynomially weighted, $L^2$ and $H^1$ tails.
\end{proof}

\begin{proposition}[First-order lower bound for compact critical profiles]\label{prop:critical-compact-liminf}
Assume $\partial\Omega$ is $C^{2,1}$ and \eqref{eq:critical-window-sigma}.  Let $\delta_n\downarrow0$ be a sequence of bounded-domain optimisers whose critical profile is compact.  After absorbing the tangential recentering into the boundary chart, suppose
\[
P_n\longrightarrow P\in\partial\Omega,
\qquad
q_n\longrightarrow v
\quad\text{strongly in }H^1(\Hh),
\]
and the rescaled favourable sets converge in measure to the threshold optimiser $A$.  Then
\begin{equation}\label{eq:critical-compact-liminf}
\liminf_{n\to\infty}
\frac{\eps_n^2\Lambda_{\delta_n}(\alpha_{\delta_n})-k^2}{\eps_n}
\ge
\mathcal G_\sigma(P;A,v).
\end{equation}
\end{proposition}

\begin{proof}
By Lemma~\ref{lem:uniform-exp-localisation}, the entire favourable set lies in one fixed patch-scale ball and the rescaled eigenfunctions have uniform exponential tails.  We may therefore use a moving Fermi chart centred at $P_n$ on the whole part of the eigenfunction relevant at order $\eps_n$, with errors $o(\eps_n)$.  Write $F_n\subset\Hh$ for the full flattened, $\eps_n$-rescaled favourable set and keep the notation $q_n$ for the corresponding rescaled eigenfunction.  Then
\[
F_n\to A\quad\text{in measure},
\qquad
q_n\to v\quad\text{strongly in }H^1(\Hh),
\]
and Lemma~\ref{lem:uniform-exp-localisation} upgrades this to convergence of all first moments appearing below.

Set
\[
a_n:=|F_n|,
\qquad
M_n:=\int_{\Hh}m_{F_n}q_n^2,
\qquad
M_{1,n}:=\int_{\Hh}y_Nm_{F_n}q_n^2.
\]
For later reference we make the Fermi remainders quantitative.  Uniformly for $P_n$ on the compact boundary and for $\eps_n|y|$ inside a fixed tubular neighbourhood,
\begin{align*}
J_{P_n}(\eps_n y)
&=1-\eps_n h_{P_n}y_N+\eps_n^2R^J_n(y),\\
J_{P_n}(\eps_n y)G^{-1}_{P_n}(\eps_n y)
&=I+\eps_n y_N(2\widetilde{\mathsf S}_{P_n}-h_{P_n}I)
+\eps_n^2R^G_n(y),
\end{align*}
with
\[
|R^J_n(y)|+\|R^G_n(y)\|\le C|y|^2,
\qquad
J_{\partial,P_n}(\eps_n y',0)=1+\eps_n^2R^\partial_n(y'),
\quad |R^\partial_n(y')|\le C|y'|^2.
\]
After the fixed physical cutoff, Lemma~\ref{lem:uniform-exp-localisation} therefore bounds every discarded numerator or denominator remainder by
\[
C\eps_n^2
\int_{\Hh}(1+|y|^2)(q_n^2+|\nabla q_n|^2)\,dy
+O(e^{-c/\eps_n})
=O(\eps_n^2)=o(\eps_n).
\]
The favourable-set volume remainder is $O(\eps_n^2)$ as well, because \eqref{eq:full-set-patch-containment} puts $F_n$ in one fixed ball.  Thus all first-order expansions below are uniform and require no unrecorded convergence rate.

The exact physical volume constraint and the Fermi Jacobian expansion give
\begin{equation}\label{eq:compact-volume-first}
1
=\int_{F_n}J_{P_n}(\eps_n y)\,dy
=a_n-\eps_n h_{P_n}\int_{F_n}y_N\,dy+o(\eps_n),
\end{equation}
so that
\begin{equation}\label{eq:compact-an-first}
a_n
=1+\eps_n h_Pa(A)+o(\eps_n).
\end{equation}
Likewise, the physical normalisation gives
\begin{equation}\label{eq:compact-Mn-first}
M_n
=1+\eps_n h_PM_1+o(\eps_n),
\qquad
M_1:=\int_{\Hh}y_Nm_Av^2.
\end{equation}
Finally the numerator expansion is
\begin{align}
\eps_n^2\Lambda_{\delta_n}(\alpha_{\delta_n})
={}&E_{\tau_{\delta_n}}(q_n)\notag\\
&+\eps_n\left(
2\mathsf S_P:\mathsf M_T(v)-h_PE_1
\right)
+o(\eps_n),
\label{eq:compact-energy-first}
\end{align}
where
\[
E_1:=\int_{\Hh}y_N|\nabla v|^2.
\]
The replacement of the moments of $q_n$ by those of $v$ is justified by the uniform exponential tails.

It remains to obtain the correct first-order lower bound for the flat term without assuming any convergence rate.  Let
\[
s_n:=a_n^{1/N}
=1+\eps_n\frac{h_Pa(A)}N+o(\eps_n),
\]
and define the unit-volume dilation
\[
\widehat F_n:=s_n^{-1}F_n,
\qquad
\widehat q_n(z):=q_n(s_nz).
\]
Then $|\widehat F_n|=1$ and
\[
\widehat M_n
:=\int_{\Hh}m_{\widehat F_n}\widehat q_n^2
=s_n^{-N}M_n>0
\]
for large $n$.  Threshold optimality at unit favourable volume gives the exact nonnegative flat excess
\begin{equation}\label{eq:threshold-flat-excess-unit}
E_{\tau_*}(\widehat q_n)-k^2\widehat M_n\ge0.
\end{equation}
Since
\[
E_{\tau_{\delta_n}}(q_n)
=s_n^{N-2}E_{\tau_{\delta_n}s_n}(\widehat q_n),
\]
\eqref{eq:threshold-flat-excess-unit} yields
\begin{align}
E_{\tau_{\delta_n}}(q_n)
\ge{}&
M_ns_n^{-2}
\left[
 k^2
 +(\tau_{\delta_n}s_n-\tau_*)R_n
\right],
\label{eq:flat-profile-lower}
\end{align}
where
\[
R_n
:=
\frac{\displaystyle\int_{\partial\Hh}\widehat q_n^2}
{\widehat M_n}
\longrightarrow T(v).
\]
Using \eqref{eq:critical-window-sigma}, \eqref{eq:compact-an-first}, and \eqref{eq:compact-Mn-first} in \eqref{eq:flat-profile-lower}, we obtain
\begin{align}
E_{\tau_{\delta_n}}(q_n)
\ge k^2
+\eps_n\Biggl[
&\sigma T(v)
+h_Pk^2M_1\notag\\
&+\frac{h_Pa(A)}N
\bigl(\tau_*T(v)-2k^2\bigr)
\Biggr]
+o(\eps_n).
\label{eq:flat-profile-expanded}
\end{align}
Combining \eqref{eq:compact-energy-first} and \eqref{eq:flat-profile-expanded}, and using the normal-moment identity
\[
E_1-\frac12T(v)=k^2M_1,
\]
gives exactly \eqref{eq:critical-compact-liminf}.
\end{proof}

\begin{lemma}[Boundary flattening for an escaping core]\label{lem:escape-boundary-layer-flattening}
Assume the hypotheses of Lemma~\ref{lem:uniform-exp-localisation} and suppose the classified profile is escaping.  Let $x_n$ be the core centres,
\[
d_n:=\operatorname{dist}(x_n,\partial\Omega),
\qquad
h_n:=\frac{d_n}{\eps_n}\longrightarrow+\infty.
\]
If $d_n\to0$, let $P_n$ be the nearest boundary point.  Then, after cutting off only at a fixed physical distance from $P_n$ and flattening the boundary graph, one obtains a pair $(\widetilde F_n,\widetilde q_n)$ on $\Hh$ such that
\begin{equation}\label{eq:escape-flat-volume}
|\widetilde F_n|=1,
\end{equation}
\begin{equation}\label{eq:escape-flat-mass-energy}
\int_{\Hh}m_{\widetilde F_n}\widetilde q_n^2=1+o(\eps_n),
\qquad
E_{\tau_{\delta_n}}(\widetilde q_n)
=\eps_n^2\Lambda_{\delta_n}(\alpha_{\delta_n})+o(\eps_n),
\end{equation}
and
\begin{equation}\label{eq:escape-flat-trace-small}
\|\operatorname{Tr}\widetilde q_n\|_2^2\longrightarrow0.
\end{equation}
If $\liminf d_n>0$, the same conclusions hold after an interior cutoff and zero extension to $\R^N$, with no boundary term and errors $O(e^{-c/\eps_n})$.
\end{lemma}

\begin{proof}
Only the case $d_n\to0$ needs discussion.  Use rigid coordinates centred at $P_n$ with inward normal $e_N$.  Since $d_n$ lies in the tubular neighbourhood for large $n$, the nearest-point relation gives
$x_n=P_n+d_ne_N$.  In the rescaled coordinates $y=(x-P_n)/\eps_n$, the boundary is a graph
\[
y_N=\gamma_n(y'),
\]
and, on every fixed physical chart,
\begin{equation}\label{eq:escape-graph-estimates}
\gamma_n(0)=0,
\qquad
\nabla\gamma_n(0)=0,
\qquad
|\gamma_n(y')|\le C\eps_n|y'|^2,
\qquad
|\nabla\gamma_n(y')|\le C\eps_n|y'|.
\end{equation}
The $C^{2,1}$ regularity gives the sharper uniform Taylor expansion
\begin{equation}\label{eq:escape-graph-linearisation}
\nabla\gamma_n(y')
=\eps_n\mathsf S_{P_n}y'
+O(\eps_n^2|y'|^2)
\end{equation}
whenever $\eps_n|y'|$ remains in the fixed graph chart.  The core is centred at $(0,h_n)$ and, by Lemma~\ref{lem:uniform-exp-localisation}, the whole favourable set lies in a fixed ball about that point while all energy outside polynomially weighted neighbourhoods is exponentially small.

Choose a cutoff in the original domain which is identically one on a fixed physical neighbourhood containing the core and vanishes before the artificial sides of the graph chart.  Since those sides remain a fixed positive physical distance from $x_n$, Lemma~\ref{lem:uniform-exp-localisation} makes the resulting cutoff errors $O(e^{-c/\eps_n})$.  We omit this cutoff from the notation.  Flatten the graph by the exact vertical map
\[
\Theta_n(y',y_N):=(y',y_N-\gamma_n(y')).
\]
Its Jacobian determinant is one and it sends the graph to $\{z_N=0\}$.  Put
\[
\widetilde q_n(z):=q_n\!\left(z',z_N+\gamma_n(z')\right),
\qquad
\widetilde F_n:=\Theta_n(F_n).
\]
Because the cutoff is one on the entire favourable core and $\det D\Theta_n=1$, the physical volume constraint gives \eqref{eq:escape-flat-volume} exactly, while the actual indefinite denominator associated with $\widetilde F_n$ is unchanged up to the exponentially small chart cutoff.  This gives the first relation in \eqref{eq:escape-flat-mass-energy}.

For the bulk energy, the chain rule gives the exact identity
\begin{align}
\int |\nabla q_n|^2
={}&\int_{\Hh}|\nabla\widetilde q_n|^2
-2\int_{\Hh}
\partial_N\widetilde q_n\,
\nabla\gamma_n\cdot\nabla_T\widetilde q_n\notag\\
&+\int_{\Hh}|\nabla\gamma_n|^2
|\partial_N\widetilde q_n|^2
+O(e^{-c/\eps_n}).
\label{eq:escape-flatten-energy-identity}
\end{align}
The quadratic term is bounded by
\[
C\eps_n^2
\int_{\Hh}|z'|^2|\partial_N\widetilde q_n|^2
+O(\eps_n^3)
=O(\eps_n^2)
\]
by \eqref{eq:escape-graph-estimates} and the uniform exponentially weighted $H^1$ bounds from Lemma~\ref{lem:uniform-exp-localisation}.  For the linear term, \eqref{eq:escape-graph-linearisation} gives
\[
\frac1{\eps_n}
\int
\partial_N\widetilde q_n\,\nabla\gamma_n\cdot\nabla_T\widetilde q_n
=
\int
\partial_N\widetilde q_n\,
(\mathsf S_{P_n}z')\cdot\nabla_T\widetilde q_n
+O(\eps_n),
\]
where the $O(\eps_n)$ is uniform by the second weighted moment bound.  Recenter vertically by $h_n$.  After passing to a further subsequence, $P_n\to P$.  The escaping-profile theorem, together with the fact that the graph flattening is $C^1$-close to the identity on every fixed ball about the recentered core, gives strong $H^1(\R^N)$ convergence to the radial whole-space optimiser $\Phi$; the exponential bounds give convergence of the first weighted moments.  Hence
\begin{align*}
\frac1{\eps_n}
\int_{\Hh}
\partial_N\widetilde q_n\,
\nabla\gamma_n\cdot\nabla_T\widetilde q_n
\longrightarrow
\int_{\R^N}
\partial_N\Phi(z)
\,(\mathsf S_P z')\cdot\nabla_T\Phi(z)\,dz
=0.
\end{align*}
The last integral vanishes because $\Phi$ is radial: its integrand is odd in the recentered normal variable.  Thus the bulk-energy change in \eqref{eq:escape-flatten-energy-identity} is $o(\eps_n)$.

On the boundary, the surface element becomes
\[
\sqrt{1+|\nabla\gamma_n|^2}\,dz'
=(1+O(\eps_n^2|z'|^2))\,dz'.
\]
The surface-measure error in the Robin term is bounded by
$C\eps_n^2\int |z'|^2\widetilde q_n(z',0)^2\,dz'$, hence is $o(\eps_n)$ by the exponential localisation and $h_n\to\infty$.  The same estimate gives
\[
\|\operatorname{Tr}\widetilde q_n\|_2^2
\le C(1+h_n^{N-1})e^{-ch_n}\longrightarrow0.
\]
This proves \eqref{eq:escape-flat-mass-energy}--\eqref{eq:escape-flat-trace-small}.  No rate relating $h_n\to\infty$ to $\eps_n\to0$ is used: the only first-order geometric term is multiplied by $\eps_n$, and its radial limit cancels exactly.

If $d_n$ stays bounded below, take a cutoff which is one on a fixed neighbourhood of $x_n$ and vanishes before reaching $\partial\Omega$.  Lemma~\ref{lem:uniform-exp-localisation} makes all cutoff errors exponentially small in $1/\eps_n$; zero extension then gives a whole-space competitor of unit favourable volume.  This proves the final assertion.
\end{proof}

\begin{proposition}[First-order lower bound for escaping critical profiles]\label{prop:critical-escape-liminf}
Under \eqref{eq:critical-window-sigma}, let $\delta_n\downarrow0$ be a sequence of bounded-domain optimisers whose critical profile escapes normally.  Then
\begin{equation}\label{eq:critical-escape-liminf}
\liminf_{n\to\infty}
\frac{\eps_n^2\Lambda_{\delta_n}(\alpha_{\delta_n})-k^2}{\eps_n}
\ge0.
\end{equation}
\end{proposition}

\begin{proof}
If the physical distance of the core from the boundary stays bounded below, Lemma~\ref{lem:escape-boundary-layer-flattening} gives a whole-space admissible pair with unit favourable volume and with energy differing from
$\eps_n^2\Lambda_{\delta_n}$ by $o(\eps_n)$.  Whole-space optimality immediately yields \eqref{eq:critical-escape-liminf}.

It remains to consider $d_n\to0$.  Let $(\widetilde F_n,\widetilde q_n)$ be supplied by Lemma~\ref{lem:escape-boundary-layer-flattening}.  Since $|\widetilde F_n|=1$, the threshold variational inequality may be applied to this particular favourable set:
\[
E_{\tau_*}(\widetilde q_n)
\ge
k^2\int_{\Hh}m_{\widetilde F_n}\widetilde q_n^2.
\]
Therefore
\begin{align*}
E_{\tau_{\delta_n}}(\widetilde q_n)
&\ge
k^2\int_{\Hh}m_{\widetilde F_n}\widetilde q_n^2
+(\tau_{\delta_n}-\tau_*)
\|\operatorname{Tr}\widetilde q_n\|_2^2\\
&=k^2+o(\eps_n),
\end{align*}
where we used \eqref{eq:critical-window-sigma}, \eqref{eq:escape-flat-mass-energy}, and \eqref{eq:escape-flat-trace-small}.  Returning to the physical pair by \eqref{eq:escape-flat-mass-energy} proves \eqref{eq:critical-escape-liminf}.
\end{proof}

\begin{proposition}[Finiteness of $\Gamma_{\rm bd}$]\label{prop:Gamma-finite}
For every fixed $\sigma\in\R$,
\[
-\infty<\Gamma_{\rm bd}(\sigma)<+\infty.
\]
\end{proposition}

\begin{proof}
Upper finiteness is immediate: Theorem~\ref{thm:threshold-attainment} supplies at least one threshold optimiser $(A,v)\in\mathfrak O_*$, the boundary $\partial\Omega$ is compact, and Lemma~\ref{lem:threshold-exponential-moments} makes every quantity in $\mathcal G_\sigma(P;A,v)$ finite.

Suppose, for contradiction, that $\Gamma_{\rm bd}(\sigma)=-\infty$. Then there are $P_j\in\partial\Omega$ and $(A_j,v_j)\in\mathfrak O_*$ such that
\[
\mathcal G_\sigma(P_j;A_j,v_j)\le-j.
\]
For each fixed $j$, Proposition~\ref{prop:critical-first-order-transplant} applies to this profile. Choose $\eps_j>0$, with $\eps_j<j^{-2}$ and $\eps_j\downarrow0$, so small that, with $\delta_j=\eps_j^N$ and $\tau_{\delta_j}=\tau_*+\sigma\eps_j$, the corresponding transplanted competitor gives
\[
\frac{\eps_j^2\Lambda_{\delta_j}(\alpha_{\delta_j})-k^2}{\eps_j}\le-j+1.
\]
Choose an actual bounded-domain optimiser at each $\delta_j$. The critical profile theorem, Theorem~\ref{thm:bounded-profile-critical}, gives after passage to a subsequence either a compact critical profile or an escaping profile. In the compact case Proposition~\ref{prop:critical-compact-liminf} gives a finite lower bound
\[
\liminf_{j\to\infty}
\frac{\eps_j^2\Lambda_{\delta_j}-k^2}{\eps_j}
\ge \mathcal G_\sigma(P;A,v)>-\infty,
\]
whereas in the escaping case Proposition~\ref{prop:critical-escape-liminf} gives the lower bound zero. Both alternatives contradict the preceding upper bound, whose right-hand side tends to $-\infty$. Therefore $\Gamma_{\rm bd}(\sigma)>-\infty$.
\end{proof}

\begin{theorem}[First-order critical selection law]\label{thm:critical-first-order-selection}
Assume $\partial\Omega$ is $C^{2,1}$ and
\[
\tau_\delta
=\tau_*+\sigma\eps+o(\eps),
\qquad
\eps=\delta^{1/N},
\]
for some fixed $\sigma\in\R$.  Then
\begin{equation}\label{eq:critical-selection-law}
\frac{\eps^2\Lambda_\delta(\alpha_\delta)-k^2}{\eps}
\longrightarrow
\min\{0,\Gamma_{\rm bd}(\sigma)\}.
\end{equation}
Moreover the leading profile is selected as follows.
\begin{enumerate}[label=(\roman*)]
\item If $\Gamma_{\rm bd}(\sigma)<0$, every sequence $\delta_n\downarrow0$ has only compact critical profile subsequences.  Every such limit $(P;A,v)$ satisfies
\[
\mathcal G_\sigma(P;A,v)=\Gamma_{\rm bd}(\sigma).
\]
\item If $\Gamma_{\rm bd}(\sigma)>0$, every sequence $\delta_n\downarrow0$ has only escaping profile subsequences, and the first-order coefficient is $0$.
\item If $\Gamma_{\rm bd}(\sigma)=0$, both leading alternatives remain compatible with the first-order value; any compact limiting profile must satisfy $\mathcal G_\sigma(P;A,v)=0$.
\end{enumerate}
\end{theorem}

\begin{proof}
By Proposition~\ref{prop:Gamma-finite}, $\Gamma_{\rm bd}(\sigma)$ is a finite real number. Corollary~\ref{cor:critical-first-order-upper} gives
\[
\limsup_{\delta\downarrow0}
\frac{\eps^2\Lambda_\delta(\alpha_\delta)-k^2}{\eps}
\le
\min\{0,\Gamma_{\rm bd}(\sigma)\}.
\]
Take an arbitrary sequence $\delta_n\downarrow0$ and apply Theorem~\ref{thm:bounded-profile-critical}.  Along a further subsequence, either the compact or escaping alternative occurs.  Proposition~\ref{prop:critical-compact-liminf} gives, in the first case,
\[
\liminf
\frac{\eps^2\Lambda_\delta-k^2}{\eps}
\ge
\mathcal G_\sigma(P;A,v)
\ge
\Gamma_{\rm bd}(\sigma),
\]
whereas Proposition~\ref{prop:critical-escape-liminf} gives a lower bound by zero in the second case.  Hence every subsequence has a further subsequence satisfying the lower bound
\[
\liminf
\frac{\eps^2\Lambda_\delta-k^2}{\eps}
\ge
\min\{0,\Gamma_{\rm bd}(\sigma)\}.
\]
This proves \eqref{eq:critical-selection-law}.

If $\Gamma_{\rm bd}(\sigma)<0$, an escaping subsequence would have first-order liminf at least zero, contradicting \eqref{eq:critical-selection-law}; hence only compact limits occur, and the two inequalities force their response to equal $\Gamma_{\rm bd}(\sigma)$.  If $\Gamma_{\rm bd}(\sigma)>0$, a compact subsequence would have liminf at least $\Gamma_{\rm bd}(\sigma)>0$, while the value limit is zero, so only escape is possible.  The case $\Gamma_{\rm bd}(\sigma)=0$ follows in the same way.
\end{proof}

\begin{remark}[Interpretation of the first-order selection law]\label{rem:first-order-selection-scope}
Theorems~\ref{thm:critical-first-order-selection} and~\ref{thm:threshold-tangential-symmetry} determine the first-order value and the alternative between compact boundary concentration and escape.  When $\Gamma_{\rm bd}(\sigma)<0$, a compact limiting profile minimises
\[
\sigma T(v)+h_P\mathcal C(A,v)
\]
over boundary points and threshold optimisers.  For a fixed threshold profile, the preferred boundary points therefore lie at extrema of mean curvature according to the sign of $\mathcal C(A,v)$.  What remains open is uniqueness of the threshold optimiser and the sign, or possible profile dependence, of $\mathcal C(A,v)$.  If $\Gamma_{\rm bd}(\sigma)=0$, a further expansion is needed to distinguish the alternatives.
\end{remark}

\clearpage
\section{Numerical illustrations}\label{sec:numerics}

The computations in this section are intended to illustrate the analytical alternatives proved above.  The principal two-dimensional experiments use $N=2$, $\kappa=1$ and the ellipse
\[
 \Omega=\left\{(x_1,x_2):\frac{x_1^2}{4}+x_2^2<1\right\}.
\]
For the shape plots we display the normalised principal eigenfunction as a continuous colour field and draw the boundary of the favourable set $D$ in black.  This makes the competing boundary and interior configurations directly comparable.  The numerical eigenvalues are shown as the rescaled quantities $\delta\lambda$.  The scripts, raw tables and discrete fields used below are included with the source.

Near the transition the discrete rearrangement problem can have several stationary states.  We therefore use continuation in the Robin parameter and several initial configurations, and we compare the lowest boundary-attached and interior branches found in this way.  The numerical crossings are not certified global optima of the discrete free-boundary problem, and no numerical quantity is used in the proofs.

\subsection{Threshold benchmark}

For $N=2$ and $\kappa=1$, the independently computed whole-space value is
\[
 k^2=\Lambda_{\R^2}=8.1902771324,
 \qquad
 \Lambda_{\Hh}(0)=\frac{k^2}{2}=4.0951385662.
\]
At mesh size $h=1/30$ the discrete Neumann value is $4.0992361700$, a relative error of about $10^{-3}$.  The attached half-space branch meets the whole-space value at the discrete crossings
\begin{center}
\begin{tabular}{c|cccc}
$h$ & $0.05$ & $0.04$ & $1/30$ & $0.025$\\ \hline
$\tau_{*,h}$ & $3.1585$ & $3.1865$ & $3.2007$ & $3.2137$
\end{tabular}
\end{center}
A linear extrapolation against $h^2$ gives $3.233$.  Since the cellwise free-boundary representation does not by itself justify a particular convergence order, and several stationary branches coexist near the crossing, we do not attach an error bar to this extrapolation.  The computations suggest that the continuum threshold is near
\[
 \tau_*\approx3.2.
\]

\begin{figure}[!htbp]
\centering
\begin{minipage}[t]{0.49\textwidth}\centering
\includegraphics[width=\linewidth]{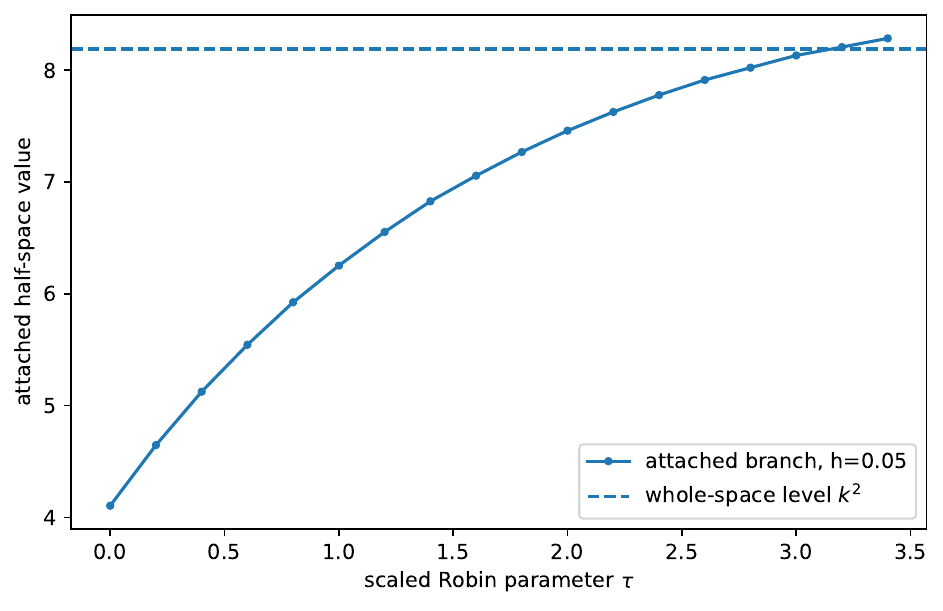}\\[-0.3em]
\small (a) Attached branch and whole-space value.
\end{minipage}\hfill
\begin{minipage}[t]{0.49\textwidth}\centering
\includegraphics[width=\linewidth]{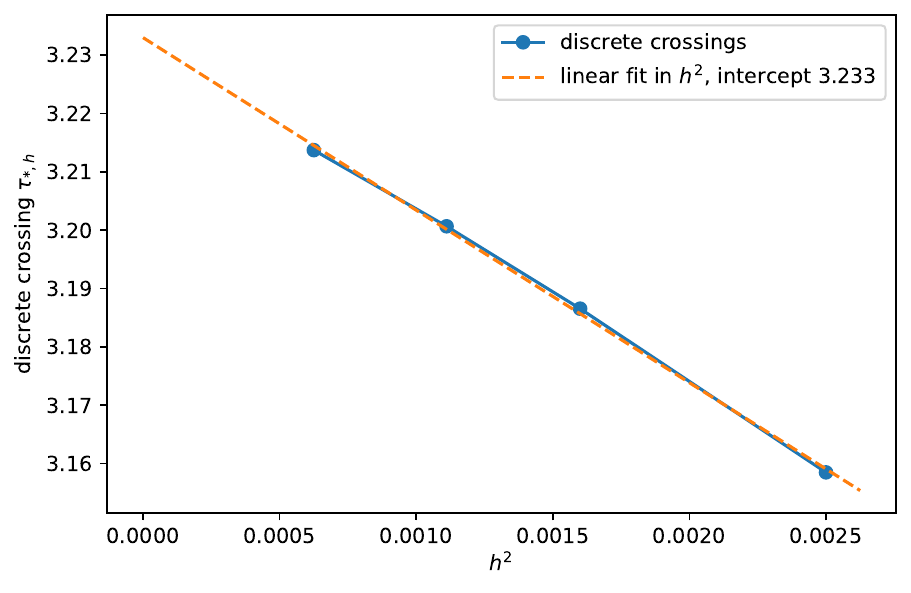}\\[-0.3em]
\small (b) Mesh convergence of the crossing.
\end{minipage}
\caption{Numerical localisation of the half-space transition threshold for $N=2$ and $\kappa=1$.}
\label{fig:numerical-halfspace-phase}
\end{figure}

At the discrete threshold one observes the two alternatives described in Theorem~\ref{thm:critical-halfspace-profile}.  The left panel of Figure~\ref{fig:numerical-threshold-alternatives} shows the lowest compact boundary-touching state found at the discrete crossing.  The right panel translates the exact whole-space ball profile to height $s$; both its quotient excess over $k^2$ and its boundary trace tend rapidly to zero.  Thus the compact and escaping alternatives at the threshold are both visible numerically.

\begin{figure}[!htbp]
\centering
\includegraphics[width=0.96\textwidth]{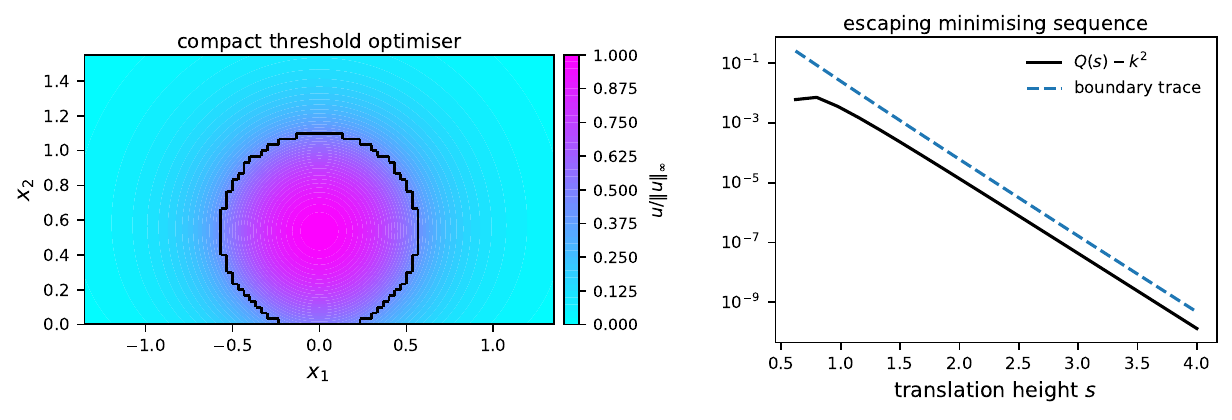}
\caption{The two threshold alternatives.  Left: compact boundary state at the discrete half-space crossing.  Right: translated whole-space profiles form an escaping minimising sequence with quotient converging to $k^2$ and vanishing boundary trace.}
\label{fig:numerical-threshold-alternatives}
\end{figure}

The derivative and variational characterisations of the threshold provide two further checks.  On the $h=1/30$ branch,
\[
 \frac{\Lambda_h(3.20)-\Lambda_h(3.18)}{0.02}=0.4151568,
 \qquad
 \frac{T_h(3.20)+T_h(3.18)}2=0.4151594,
\]
in agreement with $\Lambda_{\Hh}'(\tau_*^-)=T_*$.  Moreover, for
\[
 R_h(\tau):=\frac{k^2-\int_{\Hh}|\nabla v_h|^2}{\int_{\partial\Hh}v_h^2}
 =\tau+\frac{k^2-\Lambda_h(\tau)}{T_h(\tau)},
\]
the values at $\tau=3.00,3.10,3.15,3.18,3.20$ are $3.1767,3.1860,3.2002,3.2006,3.20065$, converging to the independently located crossing $\tau_{*,h}=3.200652$.

\subsection{Boundary or interior?}

The basic decision predicted by Theorem~\ref{thm:intro-main} is shown directly in Figure~\ref{fig:numerical-leading-regimes}.  We fix $\delta=0.03$ and compare a boundary-attached branch with an interior branch on the same mesh.  For the clearly subcritical choice $\tau=2$, the boundary branch has the smaller rescaled eigenvalue, $7.3501$ versus $8.3003$.  For the clearly supercritical choice $\tau=3.6$, the inequality reverses: the interior branch gives $8.3003$ versus $8.4030$ for the boundary branch.  The interior candidate need not sit at the centre of the ellipse; at leading order only its separation from the boundary on the patch scale matters.

\begin{figure}[!htbp]
\centering
\includegraphics[width=0.97\textwidth]{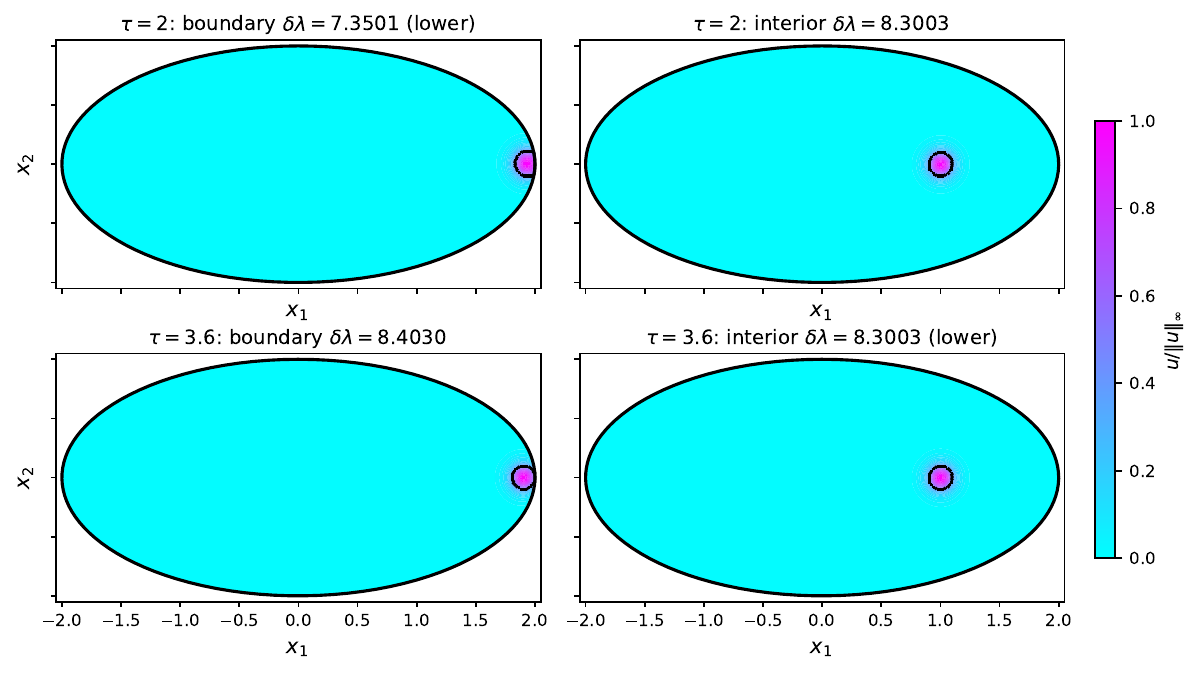}
\caption{Clear boundary and interior regimes on the ellipse for $\delta=0.03$.  The colour field is the normalised principal eigenfunction and the black curve outlines $D$.  Top: $\tau=2<\tau_*$, where the boundary branch has the smaller value.  Bottom: $\tau=3.6>\tau_*$, where the interior branch has the smaller value.}
\label{fig:numerical-leading-regimes}
\end{figure}

The fixed-coefficient regime gives a further check of the subcritical prediction.  With $\alpha=1$, so that $\tau_\delta=\sqrt\delta\to0$, the lowest computed branch remains boundary attached for $\delta=0.08,0.05,0.03$, with rescaled values $4.1321,4.1217,4.1219$, close to and consistent with convergence to the Neumann half-space value.  Thus the numerics reproduce the consequence that every fixed finite $\alpha$ is asymptotically subcritical.

\subsection{Critical window: first-order selection and the degenerate case}

For the critical-window experiments we use $\widehat\tau_*=3.225$, obtained from the $h^2$ extrapolation above, as a working numerical parameter and set
\[
 \tau_\delta=\widehat\tau_*+\sigma\sqrt\delta.
\]
This working value is used only to organise the illustrations and is not an error-certified estimate of the continuum threshold.  We again use $\delta=0.03$.  Figure~\ref{fig:numerical-critical-window} shows three comparisons.  For $\sigma=0$, the boundary correction is negative and the boundary branch has the smaller value: $8.23950<8.30033$.  For $\sigma=0.9$, the correction is positive and the interior branch has the smaller value: $8.30033<8.31841$.

The middle column addresses the first-order equality case.  The finest half-space moment computation gives the numerical zero $\sigma_0^{(h)}\simeq0.5913$ of the boundary correction.  At this value the first-order theory does not distinguish the two configurations, yet the finite-$\delta$ computation does: on the present mesh the boundary branch gives $8.29157$ and the interior branch $8.30033$.  Interpolating the same comparison for $\delta=0.08,0.05,0.03$ gives boundary-minus-interior differences approximately $-5.1\times10^{-4}$, $-2.12\times10^{-2}$ and $-8.77\times10^{-3}$.  Thus the tested finite problems favour the boundary branch, but the differences are small and not monotone in $\delta$; we therefore regard this only as a numerical indication of a higher-order preference, not as a second-order asymptotic conclusion.

\begin{figure}[!htbp]
\centering
\includegraphics[width=0.97\textwidth]{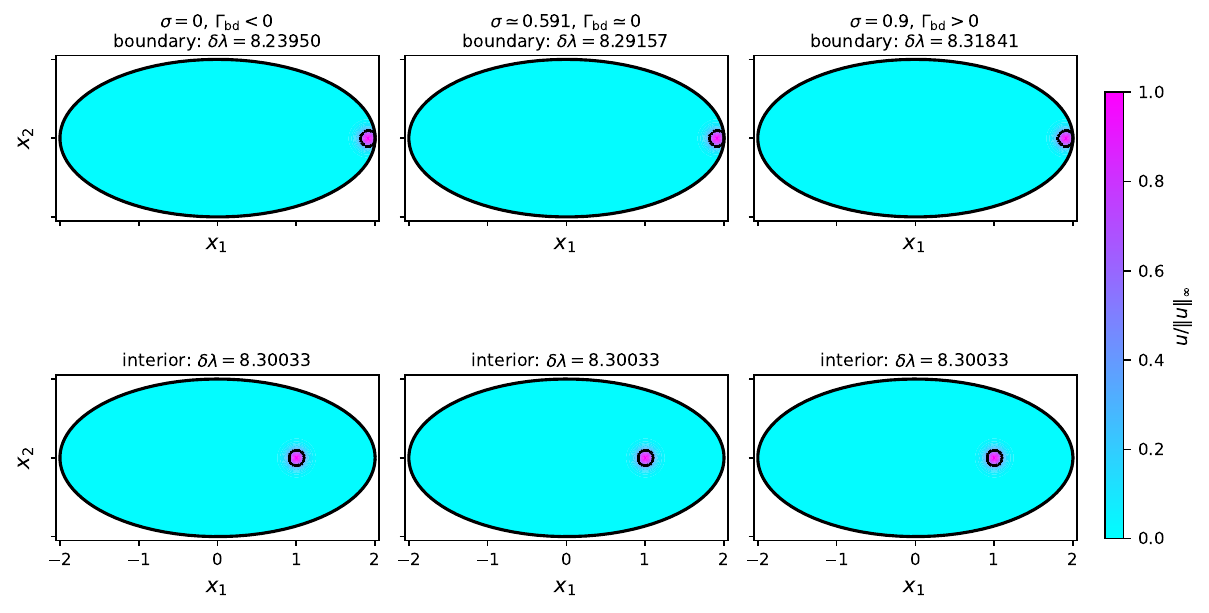}
\caption{Boundary--interior competition in the critical window on the ellipse, $\delta=0.03$.  In each column the boundary candidate is shown above the interior candidate.  Left: $\Gamma_{\rm bd}<0$ and the boundary branch has the smaller value.  Middle: the first-order boundary correction is numerically zero, while the finite-$\delta$ computation still slightly favours the boundary branch.  Right: $\Gamma_{\rm bd}>0$ and the interior branch has the smaller value.}
\label{fig:numerical-critical-window}
\end{figure}

\subsection{Where on the boundary?}

When the boundary branch is selected in the critical window, the first-order theory also determines its preferred location through mean curvature.  On the ellipse, the curvature is $2$ at $(\pm2,0)$ and $1/4$ at $(0,\pm1)$.  For the threshold profile computed here the curvature coefficient is negative.  Hence maximal curvature should be preferred.  At $\delta=0.03$ and $\sigma=0$, the high-curvature and low-curvature branches have rescaled values $8.23950$ and $8.26852$, respectively, as shown in Figure~\ref{fig:numerical-curvature-location}.

\begin{figure}[!htbp]
\centering
\includegraphics[width=0.97\textwidth]{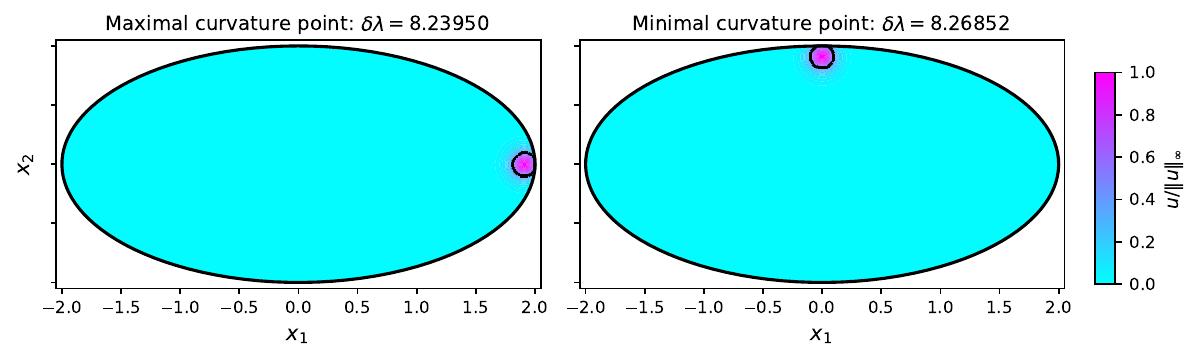}
\caption{Critical boundary placement on the ellipse.  For $N=2$, $\kappa=1$ the computed curvature coefficient is negative, and the branch attached at a point of maximal curvature has the smaller eigenvalue.}
\label{fig:numerical-curvature-location}
\end{figure}

\subsection{Contrast and dimension checks}

The threshold depends substantially on the contrast $\kappa$.  On the common $N=2$, $h=0.05$ discretisation we obtain
\begin{center}
\begin{tabular}{c|ccc}
$\kappa$ & $k^2$ & $\tau_{*,h}$ & $\mathcal C_h$ \\ \hline
$0.5$ & $20.0714$ & $4.5422$ & $-0.4571$\\
$1$   & $8.1903$  & $3.1585$ & $-0.1340$\\
$2$   & $3.2544$  & $2.1313$ & $-0.0690$
\end{tabular}
\end{center}
so increasing the favourable contrast lowers the scaled Robin loss at which the boundary-to-interior transition occurs.  The computed curvature coefficient remains negative throughout these cases; additional coarse tests at $\kappa=0.2,5,10$ found no sign reversal near the discrete crossing.  This is only numerical evidence, and the analytic sign of $\mathcal C$ remains open.

Finally, Figure~\ref{fig:numerical-3d-comparison} shows that the same geometry is present in three dimensions.  The left column comes from the axisymmetric $N=3$, $\kappa=1$ computation at $\tau=2$ and is boundary attached.  The right column shows the exact translated whole-space ball and radial profile representing the escaping core in the supercritical regime.  The corresponding discrete attached-branch crossings are $3.6540,4.0239,4.1705,4.3097$ for $h=0.10,0.0625,0.05,0.04$, with a further $h=1/30$ value near $4.365$.  Since these values are still moving appreciably, we use the three-dimensional calculation only as a qualitative check and do not quote a continuum estimate of $\tau_*(3,1)$.

\begin{figure}[!htbp]
\centering
\includegraphics[width=0.94\textwidth]{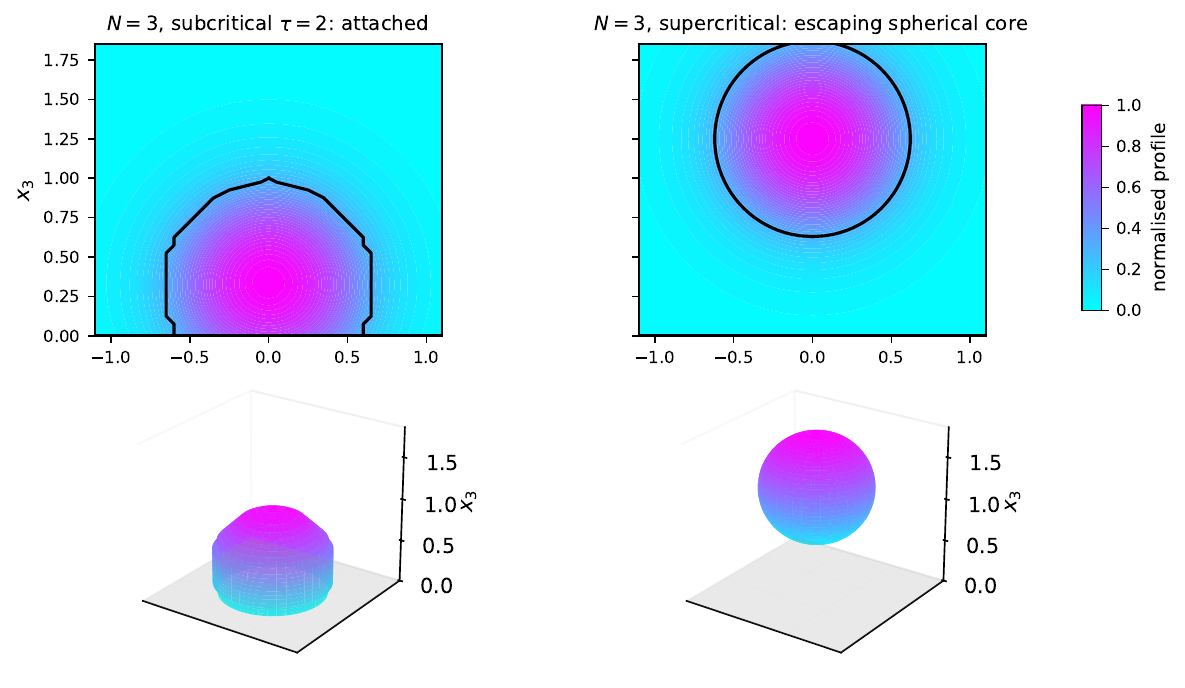}
\caption{Three-dimensional boundary and interior regimes.  Upper row: meridian sections with the normalised profile and the boundary of $D$.  Lower row: the corresponding three-dimensional favourable sets.  Left: computed subcritical attached state.  Right: translated whole-space ball representing the escaping supercritical core.}
\label{fig:numerical-3d-comparison}
\end{figure}

\section{Conclusion}\label{sec:conclusion}

The original optimisation problem asks where a favourable set of very small volume should be placed when the outer boundary is subject to Robin loss.  The answer is governed by the single scaled parameter $\tau_\delta=\alpha_\delta\delta^{1/N}$.  There is a finite threshold $\tau_*(N,\kappa)$.  If $\tau_\delta\to\tau<\tau_*$, optimal regions concentrate at the boundary.  If $\tau_\delta\to\tau>\tau_*$, their concentration centres move to distances large compared with the patch size and, after recentring and rescaling, the favourable sets converge in measure to the whole-space optimal ball.  Hence every fixed finite Robin coefficient lies asymptotically in the boundary regime; the transition away from boundary concentration occurs on the scale $\alpha_\delta\asymp\delta^{-1/N}$.

At $\tau=\tau_*$ the boundary and interior configurations have the same leading-order value.  In the finer scaling $\tau_\delta=\tau_*+\sigma\delta^{1/N}+o(\delta^{1/N})$, the first correction is the minimum of the zero interior contribution and an explicit boundary contribution.  Compact threshold optimisers are rotationally symmetric in the tangential variables, so the geometric part of that boundary contribution depends only on mean curvature.  The analysis therefore answers two successive location questions: whether the small favourable region should be near the boundary or in the interior, and, when the boundary configuration is selected at critical scale, which boundary points are preferred at first order.

The numerical experiments follow this same hierarchy.  For $N=2$ and $\kappa=1$, the half-space computations suggest a crossing near $\tau_*\approx3.2$.  On a smooth ellipse, direct comparisons show a clear boundary advantage below the threshold and a clear interior advantage above it.  In the critical window the computations display both signs of the first-order selection law.  At the numerical first-order equality parameter, the finite-$\delta$ computations show a small boundary preference; because this difference is beyond the order resolved analytically, we regard it only as evidence of a higher-order effect.  When the boundary configuration is selected, comparison of high- and low-curvature locations agrees with the predicted curvature dependence.  Further contrast values and an axisymmetric three-dimensional calculation support the same qualitative picture beyond the principal two-dimensional parameter choice.  None of these numerical observations is used in the proofs.

Open questions include uniqueness of compact threshold optimisers, the sign and possible optimiser-dependence of $\mathcal C(A,v)$, finer free-boundary regularity and quantitative shape estimates, and the next-order analysis when $\Gamma_{\rm bd}(\sigma)=0$.

\section*{Data and code availability}
The scripts, raw numerical tables, and discrete fields used to produce the numerical illustrations in Section~\ref{sec:numerics} are supplied with the manuscript as supplementary material.  No external datasets were used.

\section*{Funding}
This work was supported by the SGS project SGS05/P\v{R}F/2026 at the University of Ostrava.  Yifan Zhang was also co-funded by the European Union under REFRESH--Research Excellence For REgion Sustainability and High-tech Industries, project No.~CZ.10.03.01/00/22\_003/0000048, through the Operational Programme Just Transition.

\bibliographystyle{plainnat}
\bibliography{references}

\end{document}